\documentclass[19pt,a4paper]{amsart}

\usepackage{amsmath,amssymb}
\usepackage{amsmath,amsthm,amssymb,latexsym,mathrsfs,amsfonts,cases,indentfirst}
\usepackage[figuresright]{rotating}
\usepackage{rotating}
\usepackage{a4wide}

\usepackage{CJK, CJKnumb}
\usepackage[dvipsnames]{xcolor}
\usepackage{mathdots}
\usepackage{indentfirst}
\usepackage{latexsym, bm}
\usepackage{graphicx}
\usepackage{cases}
\usepackage{enumerate}
\usepackage{tikz-cd}
\usepackage[all]{xy}
\usepackage[active]{srcltx}
\usepackage[colorlinks,linkcolor=blue,anchorcolor=blue,citecolor=blue]{hyperref}
\usepackage{listings}
\usepackage{scalerel}

\newtheorem{theorem}{Theorem}[section]
\newtheorem{lemm}[theorem]{Lemma}
\newtheorem{prop}[theorem]{Proposition}
\newtheorem{coro}[theorem]{Corollary}

\theoremstyle{definition}
\newtheorem{defi}[theorem]{Definition}

\theoremstyle{remark}
\newtheorem{remark}[theorem]{Remark}
\numberwithin{equation}{section}
\allowdisplaybreaks

\begin{document}

\title[The Ding-Frenkel Isomorphism Theorem for $U_{r,s}\mathcal(\widehat{\mathfrak{so}_{2n+1}})$]
{The Ding-Frenkel Isomorphism Theorem for two-parameter quantum affine algebra $U_{r,s}\mathcal(\widehat{\mathfrak{so}_{2n+1}})$}
\author[Hu]{Naihong Hu{$^*$}}
\address{School of Mathematical Sciences, MOE-KLMEA \& SH-KLPMMP, East China Normal University,
MinHang Campus, DongChuan Road 500, Shanghai 200241, PR China}
\email{nhhu@math.ecnu.edu.cn}

\thanks{$^*$Corresponding author. This work is
	supported by the NNSF of China (Grant No. 12171155), and in part by the Science and Technology Commission of Shanghai Municipality (Grant No. 22DZ2229014).}
\author[Xu]{Xiao Xu}
\address{School of Mathematical Sciences, MOE-KLMEA \& SH-KLPMMP, East China Normal University,
MinHang Campus, DongChuan Road 500, Shanghai 200241, PR China}
\email{52275500004@stu.ecnu.edu.cn}

\author[Zhuang]{Rushu Zhuang}
\address{School of Mathematics and Physics,  China University of Geosciences,
Hongshan District, Lumo Road 388, Wuhan 430074, Hubei Province, PR China}
\email{zhuangrushu@cug.edu.cn}

\subjclass{Primary 17B37, 81R50; Secondary 16T25}
\date{}

\keywords{Ding-Frenkel isomorphism; regulated quantum Lyndon bases; spectral parameter-dependent $R$-matrices;
FRT presentation; $RLL$ realization}

\begin{abstract} ---
From the theory of finite-dimensional weight modules, we get the basic braided $R$-matrix $\widehat R$ of $U_{r, s}(\mathfrak{so}_{2n+1})$. For its FRT presentation $U(\widehat R)$, we achieve two word-formation methods of quantum Lyndon bases (whose bracketing rules are regulated by the $RLL$-formalism) and elucidate their distribution rule within the triangular $L$-matrix.
Consequently, we contribute an algebraic proof for establishing an isomorphism between the Drinfeld-Jimbo presentation and the FRT presentation.
In the affine setting, we first derive two spectral parameter-dependent $R$-matrices through the Yang-Baxterization.
Next, we select the only one that satisfies the intertwining property with respect to the minimal affinization.
Accordingly, we obtain the $RLL$ realization of  $U_{r, s}(\widehat{\mathfrak{so}_{2n+1}})$ through the Gauss decompositions of the generating matrices.
Finally, we contribute an algebraic proof to the Ding-Frenkel Isomorphism Theorem between the Drinfeld realization and the $RLL$ realization.

\bigskip
\noindent
{\rm R\'esum\'e.} ---
D'apr\`{e}s la th\'eorie des modules de poids de dimension finie, nous obtenons la $R$-matrice tress\'e de base $\widehat{R}$ de $U_{r, s}(\mathfrak{so}_{2n+1})$. Pour sa pr\'esentation FRT $U(\widehat{R})$, nous \'etablissons deux m\'ethodes de formation de mots des bases de Lyndon quantiques (dont les r\`{e}gles de crochetage sont r\'egies par le formalisme $RLL$) et nous \'elucidons leur r\`{e}gle de distribution au sein de la $L$-matrice triangulaire. En conséquence, nous apportons une preuve alg\'ebrique de l'\'etablissement d'un isomorphisme entre la pr\'esentation de Drinfeld-Jimbo et la pr\'esentation FRT. Dans le cadre affine, nous d\'erivons d'abord deux $R$-matrices d\'ependantes du param\`{e}tre spectral par la m\'ethode de Yang-Baxt\'erisation. Ensuite, nous choisissons l'unique élément qui satisfait la propriété d'entrelacement par rapport à l'affinisation minimale. Par cons\'equent, nous obtenons la r\'ealisation $RLL$ de $U_{r, s}(\widehat{\mathfrak{so}_{2n+1}})$ au moyen des d\'ecompositions de Gauss des matrices g\'en\'eratrices. Enfin, nous apportons une preuve alg\'ebrique du th\'eor\`{e}me d'isomorphisme de Ding-Frenkel entre la r\'ealisation de Drinfeld et la r\'ealisation $RLL$.
\end{abstract}

\maketitle

\tableofcontents

\section{Introduction}
\textbf{1.1}\;
For an affine Kac-Moody algebra $\widehat{\mathfrak{g}}$, $U_q(\widehat{\mathfrak{g}})$ can be defined as Drinfeld-Jimbo presentation via the Chevalley generators with the Serre relations \cite{Drinfeld 1985,JimboLMP 1985}.
Drinfeld also gave a celebrated new realization \cite{Drinfeld 1988}, the so-called Drinfeld realization, as the quantization of the classical loop realization.
Using the Drinfeld realization, one can investigate and classify finite-dimensional representation theory of quantum affine algebras that contribute a rich source of studies on monoidal categorifications and quantum cluster algebras (e.g., \cite{David Duke 2010,David Adv2019, Kashiwara 2020, Kashiwara 2024}, etc.),
and construct their infinite-dimensional quantum vertex representations (e.g., \cite{Frenkel 1988,Jing Invent 1990}, etc.)
In \cite{FRT Russ 1989}, Faddeev-Reshetikhin-Takhtajan gave two pesentations for the classical Lie groups $G$: the $RTT$ realization for $G_q(n)$ which is also denoted by $\mathcal O_q(G_n)$ in the literature, the quantum coordinate algebras; the FRT presentation of $U_q(\mathfrak g)$, where $\mathfrak g=\text{Lie}(G)$.
In \cite{FRT 1989}, Faddeev-Reshetikhin-Takhtajan studied the quantum Yang-Baxter equation $(\text{QYBE})$ with spectral parameters:
\begin{equation}\label{spectral qybe}
  \widehat{R}_{12}(z)\widehat{R}_{13}(zw)\widehat{R}_{23}(w)=\widehat{R}_{23}(w)\widehat{R}_{13}(zw)\widehat{R}_{12}(z),\quad z, w\in \mathbb{C},
\end{equation}
where $\widehat{R}(z)$ is a rational function of $z$ with values in End($\mathbb{C}^{n}\otimes\mathbb{C}^{n}$).
Using the solution of QYBE, they established the realization of the quantum loop algebra $U_q(\mathfrak{g}\otimes [t, t^{-1}])$.
In \cite{ResLMP 1990}, Reshetikhin and Semenov-Tian-Shanski extended this realization to $U_q(\widehat{\mathfrak{g}})$.
This realization is also known as the $RLL$ realization of (affine types).

Drinfeld early in 1987 claimed \cite{Drinfeld 1988} that the Drinfeld realization is isomorphic to the Drindeld-Jimbo presentation both for quantum affine algebras and the Yangian algebras.
For the untwisted affine types, Beck \cite{Beck1, Beck2} confirmed the Drinfeld Isomorphism in terms of the braid group action as an interpretation. A sufficiently convincing proof is what Damiani ultimately obtained \cite{D1, D2} till 2012 \& 2015 both for untwisted and twisted types. In fact, regarding the Drinfeld's new realization, there is alternative understanding via the $RLL$ realization approach aforementioned. The first detailed characterization is attributed to Ding-Frenkel \cite{Ding CMP 1993}, where using the so-called Gauss decomposition, they gave an isomorphism between the Drinfeld realization and the $RLL$ realization for $U_q(\widehat{\mathfrak{gl}_n})$.
For the Yangian algebra in type $A$, Brundan and Kelshchev also proved an analogous result \cite{Brundan 2005}.
Furthermore, Jing-Liu-Molev and Jing-Zhang-Liu worked out the result for quantum affine algebras (and some cases include the Yangian algebras) of types $B_n^{(1)}, C_n^{(1)} $ and $D_n^{(1)}$ and $A^{(2)}_{2n-1}$ \cite{JLM CMP2018,  JingLM SIGMA 2020,JingLM JMP 2020, JingZL FM 2023}, Liashyk-Pakuliak gave the $R$-matrix realization for quantum loop algebra of type $D^{(2)}_n$ in one-parameter case \cite{LP}. Wu-Lin-Zhang (\cite{WLZ}) generalized Jing-Liu-Molev's work to quantum affine superalgebra $U_q(\widehat{\mathfrak{osp}(2m+1|2n)})$.
\vspace{1em}

\textbf{1.2}\;
Takeuchi \cite{Takeuchi 1990} defined two-parameter general linear quantum groups $U_{r, s}(\mathfrak{gl}_n)$.
Benkart and Witherspoon reobtained Takeuchi's quantum groups of type $A$ in \cite{BenkartW2004}.
Afterwards, the 1st author and his collaborators systematically studied the two-parameter groups,
see \cite{BGH 2006,BGH 2007,HuWangJGP 2010} etc.
These papers demonstrate that there exist remarkable differences between $U_q(\mathfrak{g})$ and $U_{r, s}(\mathfrak{g})$.
For example, Lusztig symmetries as automorphisms don't exist in two-parameter cases in general.
In \cite{HuCMP 2008}, Hu-Rosso-Zhang originally defined $U_{r, s}(\widehat{\mathfrak{sl}_n})$,
obtained its Drinfeld realization via constructing the two-parameter vertex operators acting on the corresponding Fock space of level one (consult \cite{HuZhang2014}), proposed and constructed the quantum affine Lyndon basis.
Hu and Zhang also established the Drinfeld realization of the two-parameter quantum affine algebras corresponding to all affine untwisted types, as well as their vertex representations of level one \cite{GHZ,HuZhang2014, HuZhang JA, Zhang phd 2007} and for twisted types, see Jing-Zhang \cite{JingZ JMP2016,ZhangJ CA2007} etc.

Now a natural question is to seek the two-parameter $RLL$ realization.
In type $A$ case, it was based on the two-parameter basic braided $R$-matrix obtained by Benkart and Witherspoon \cite{Benkart baiscR} that
Jing and Liu subsequently worked out the $RLL$ realizations of quantum algebras $U_{r, s}(\mathfrak{gl}_n)$ and $U_{r, s}(\widehat{\mathfrak{gl}_n})$ \cite{JingLiuCMS 2014,JingLJA 2017}.
However, it was open for other classical types.
Indeed, there had been no information on the basic braided $R$-matrices of types $B,C,D$ for many years until a breakthrough had made in the 2nd author's Master degree thesis early finished in the end of 2022 (\cite{Xu thesis}), where the basic braided $R$-matrix for type $B$ was obtained for the first time through the weight representation theory \cite{BGH 2007} (this work stimulated his Ph.D. mates: Rushu Zhuang's and Xin Zhong's Ph. D. theses in early 2024 to finish the determination of the basic braided $R$-matrices for the classical types $D$ and $C$ and $RLL$ realizations of $U_{r,s}(\widehat{\mathfrak{so}_{2n}})$ (see \cite{ZHX D}) and $U_{r,s}(\widehat{\mathfrak{sp}_{2n}})$ (see \cite{ZhHJ C})), so that we can continue to finish the $RLL$ realization of $U_{r,s}(\widehat{\mathfrak g})$ for affine type $B_n^{(1)}$ (the first version, see \cite{HXZ}). It should be mentioned that the information on the basic $R$-matrices for types $B, C, D$ were earlier announced in arXiv on May 10, 2024, also see \cite{HXZ, ZhHJ C, ZHX D}). A recent relevant work (\cite{HJZ}) has been generalized to $U_{r,s}(\widehat{\mathfrak{gl}(m|n)})$ by Hu-Jing-Zhong.

\vspace{1em}

\textbf{1.3}\; The main goal of this paper is to establish the Ding-Frenkel Isomorphism Theorem for the two-parameter quantum affine algebra $U_{r,s}\mathcal(\widehat{\mathfrak{so}_{2n+1}})$.
 To this end, we proceed in four main steps.

 Firstly, we obtain the explicit formula of the basic braided $R$-matrix in Theorem \ref{main theorem}.
 Next we provide an intuitive proof of the Ding-Frenkel Isomorphism Theorem in the finite case in Theorem \ref{thm FRTISO}.
 One key tool here is the distribution law of the two distinct sets of Lyndon bases within the $L^+$-matrix, see Theorem \ref{thm Lyn}.

  Secondly, we derive two distinct spectral-parameter-dependent braided $R$-matrices (\ref{Rz 1}) and (\ref{Rz 2}) from the basic braided $R$-matrix of $U_{r,s}(\mathfrak{so}_{2n+1})$ via the Yang-Baxterization procedure developed by Ge \emph{et al.} \cite{CGeX CMP1991,GeWX1991}.
 (The construction of spectral-parameter-dependent braided $R$-matrices dates back to Bazhanov-Jimbo \cite{Bazhanov87,Jimbo 1986}, and was further developed in the context of integrable systems by Reshetikhin \cite{Res87}).
 Both of them satisfy the QYBE (\ref{spectral qybe}), as the braid group representation induced by the basic braided $R$-matrix admits the Birman-Wenzl-Murakami algebraic structure.

 Although the Yang-Baxterization produces several braided spectral-parameter-dependent $R$-matrices,
the requirement that the braided spectral $R$-matrix acts as an intertwiner for tensor products
of finite-dimensional representations imposes strong constraints, namely, obeying the exchange condition of scattering and evolution.
In particular, for generic spectral parameters, such intertwiners are unique up to scalar,
which singles out a distinguished spectral braided $R$-matrix, in the spirit of Jimbo’s construction \cite{Jimbo 1986}.

From a representation-theoretic viewpoint, this uniqueness reflects that
spectral-parameter-dependent intertwiners are naturally defined on minimal affinizations \cite{Chari 1995}
of finite-dimensional modules.
In our situation, we consider the minimal affinization of the first fundamental representation
of finite type. Among the spectral braided $R$-matrices obtained via Yang-Baxterization,
only one is compatible with the intertwining property
(see Propositions \ref{prop intert} and \ref{prop nonintert}).

 Thirdly, using the spectral parameter $R$-matrix we chosen, we derive the commutation relations of the generators in the $RLL$ realization stated in Theorem \ref{thm relations}.
We then construct a homomorphism $\Phi_{r,s}$ between the Drinfeld and $RLL$ realizations, see Definition \ref{def phirs}.

 Finaly, we observe that there exists a canonical surjective morphism $\pi'$ between the $RLL$ realizations of two-parameter and one-parameter case to make the diagram commute,
\[\begin{tikzcd}
	{\mathcal{U}_{r,s}\mathcal(\widehat {\mathfrak{so}_{2n+1}})} && {U(\widetilde{R}_{r,s}(z))} \\
	\\
	{\mathcal{U}_q\mathcal(\widehat {\mathfrak{so}_{2n+1}})} && {U(\widetilde{R}_q(z))}
	\arrow["{\Phi_{r,s}}", from=1-1, to=1-3]
	\arrow["\pi"', two heads, from=1-1, to=3-1]
	\arrow["{\pi'}", two heads, from=1-3, to=3-3]
	\arrow["{\Phi_q}", from=3-1, to=3-3]
\end{tikzcd}\]
and the morphism $\Phi_{r,s}$ can be specialized to the finite case established in the first step.
These two observations allow us to successfully prove the nested mutual theorem between the Drinfeld and the $RLL$ realizations to ensure the bijectivity of $\Phi_{r,s}$ in Theorem \ref{thm affine iso}.
In fact, this isomorphism can be specialized into the one-parameter isomorphism $\Phi_q$ previously established by Jing-Liu-Molev \cite{JingLM SIGMA 2020}, via the modulus of hyperbolic loci $rs=1$ and $\omega_i\omega'_i=1$ for $0\leq i\leq n$.
Indeed, our isomorphism $\Phi_{r,s}$ can be viewed as a lift of the one-parameter case $\Phi_q$.
For our two-parameter isomorphism $\Phi_{r,s}$, the surjectivity relies on the characterization of $\text{Ker}\,\pi'$.
The proof of injectivity is more subtle.
Appealing to the Beck inverse isomorphism in the two-parameter setting
\[
\Psi: U_{r,s}(\widehat{\mathfrak{so}_{2n+1}}) \xrightarrow{\sim} \mathcal{U}_{r,s}(\widehat{\mathfrak{so}_{2n+1}}),
\]
established in \cite{HuCMP 2008,HuZhang2014,Zhang phd 2007}, we are able to achieve the nested embedding of the Drinfeld realization within the RLL realization.
This way contributes an algebraic proof for the celebrated Ding-Frenkel Isomorphism Theorem.

\vspace{1em}

\textbf{1.4}\;
This paper in dealing with $(r, s)$-deformation businesses mainly contains following highlights.
First, when we try to give its FRT presentation, we discover and prove a fact which is different from the type $A$ case.
That is, for type $B$, in the upper triangular matrix $L^+$ (used in the $RLL$ formalism), the two regions concerning the anti-diagonal are distributed symmetrically with two quantum Lyndon bases of $U_{r,s}^+(\mathfrak{so}_{2n+1})$ defined by two kinds of different manners (see Definitions \ref{def Lynbas} \& \ref{def auto} \& Theorem \ref{thm Lyn}).
From Definition \ref{def Lynbas}, Proposition \ref{prop recurrence}, Remark \ref{rmk Lyn decomp} in subsection \ref{subsec intrinsic}, one sees that the precise expression of quantum Lyndon bases, especially the bracketing rule, is actually regulated by the $RLL$ relations intrinsically. That means via the $RLL$ defining approach,
we obtain {\it a criterion for defining the quantum Lyndon bases} via the word-formation regulated by the $RLL$ formalism.
The same phenomenon happens for the lower triangular matrix $L^-$.
In previous studies on one-parameter quantum algebras, the coefficients $\ell_{ij}^\pm$ in the matrix $L^\pm$ of $RLL$-formalism were all just treated as formal generators (No such descriptions can be found in the existing literature, e.g., \cite{Klimyk}).

Secondly, we provide an algebraic proof of the Ding-Frenkel Isomorphism Theorem,
replacing the Ding-Frenkel's expository analytical one (\cite{Ding CMP 1993}).
The four main steps can be found in subsection 1.3.

As a remark, the two-parameter $RLL$ realization cannot be parallelly derived from the one-parameter case.
On the one hand, since the one-parameter setting arises as a degeneration of the two-parameter case when both the parameters $(r, s)$ and the group-like elements $(\omega_i, \omega'_i)$ lie on the hyperbolic singularity loci $xy = 1$, deriving the $RLL$ realization of $U_{r,s}(\widehat{\mathfrak{so}_{2n+1}})$ requires recovering certain missing factors, such as powers of $rs$.
On the other hand, the inductive starting point of our reasoning is based on the $U_{r, s}(\widehat{\mathfrak{so}_7})$ (see subsection \ref{ssec b3}) rather than $U_q(\widehat{\mathfrak{o}_3})$\;(i.e. $n=1$) used in \cite{JingLM SIGMA 2020}.
(According to the classification of affine Lie algebras (\cite{Kac}), the family of $B_n^{(1)}$ exists only for $n\geq 3$.)
An advantage of \cite{JingLM SIGMA 2020} is the establishment of the Homomorphism Theorem by rank reduction
(This technical viewpoint remains valid in the two-parameter setting).
They used it to reduce the general $U_q(\widehat{\mathfrak{o}_{2n+1}})$
to the quantum algebra $U_q(\widehat{\mathfrak{o}_{3}})$.
However, what the exact isoclass of the latter is (which merely has one real simple root vector), remains unclear, and it does not belong to the family $U_q(B_n^{(1)})\; (n\geq 3)$.
To avoid ambiguity, we prefer to explicitly derive first the $RLL$ formalism of $U_{r, s}(\widehat{\mathfrak{so}_7})$ through extensive computational work, even though it involves considerably more complicated commutation relations among more quantum root vectors in its presentation.
Based on it and the Homomorphism Theorem by rank reduction, we can present the $RLL$ realiztion of $U_{r, s}(\widehat{\mathfrak{so}_{2n+1}})$.

\vspace{1em}

\textbf{1.5}\;
The outline of this paper is as follows.
In Section \ref{sec pre}, we recall the Drinfeld-Jimbo presentation of $U_{r, s}(\widehat{\mathfrak{so}_{2n+1}})$ with its Hopf structure, as well as the Drinfeld realization  $\mathcal{U}_{r,s}(\widehat{\frak {g}})$.
In Section \ref{sec basicR}, we present the vector representation $V$ and decompose $V^{\otimes2}$ canonically into the direct sum of three simple $U_{r,s}(\mathfrak{so}_{2n+1})$-modules.
From this, we formulate the basic braided $R$-matrix and obtain a representation of Birman-Wenzl-Murakami algebra.
We figure out the distribution rule of two quantum Lyndon bases in the $L$-matrix (see Theorem \ref{thm Lyn}).
Moreover, we find an interesting fact that the defining rules of the Lyndon bases should be regulated by the $RLL$ formalism (see the proof of Theorem \ref{thm Lyn}).
This fact highlights the significance of the triangular matrices $L^\pm$ even in the one-parameter setting.
As a consequence, this allows us to put forward an explicit solid new proof for establishing the isomorphism between the Drinfeld-Jimbo presentation and FRT presentation (see Theorem \ref{thm FRTISO}).
In Section \ref{sec specR}, we determine the spectral parameter dependent $R$-matrix $\widehat{R}(z)$ (Theorem \ref{thm spectral Rz1}),
which satisfies the intertwining property with respect to the minimal affinization of $U_{r, s}(\widehat{\mathfrak{so}_{2n+1}})$ (Propositions \ref{prop intert} and \ref{prop nonintert}),
and serves to define the $RLL$ realization of $U_{r, s}(\widehat{\mathfrak{so}_{2n+1}})$.
In Section \ref{sec $RLL$}, we calculate explicitly the commutation relations and obtain the $RLL$ realization.
Moreover, we establish an isomorphism between the Drinfeld realization and the $RLL$ realization.
In this way we intrinsically recover its Drinfeld realization, which is consistent with the one initially obtained in \cite{Zhang phd 2007}.
In the appendix, as a supplement, we provide an explicit verification of the metric condition $L^+C(L^+)^tC^{-1}=I$ for $U_{r,s}(\mathfrak{so}_{2n+1})$.
This verification confirms that the distribution law of the two quantum Lyndon bases in the $L^+$-matrix established in Section \ref{sec basicR} is compatible with the metric condition.

\section{Preliminaries}\label{sec pre}

This section is devoted to illustrating some basic definitions.

  Let $\mathbb{K}'=\mathbb{Q}(r,s)\subset\mathbb{C}$ be a ground field of rational functions in $r, s$, where $r, s$ are algebraically independent indeterminates.
  For convenience, we may consider the extension field $\mathbb{K}$  of $\mathbb{K}'$ such that $\mathbb{K}$ contains $\sqrt{r}, \sqrt{s}$.
  Let $\Phi$ be the root system of $\mathfrak{so}_{2n+1}$, with $\Pi$ a base of simple roots,
  which is a finite subset of a Euclidean space $\mathbb{R}^n$ with an inner product $(\;,\;)$.
  $\alpha_0=\delta-\theta$ is the $0$th simple root of $\widehat{\mathfrak{so}_{2n+1}}$, where $\theta$ is the highest root of $\mathfrak{so}_{2n+1}$,
  and $\delta$ is the prime imaginary root of $\widehat{\mathfrak{so}_{2n+1}}$.

  Let $\varepsilon_1, \cdots, \varepsilon_n$ denote an orthogonal basis of $\mathbb{R}^n$,
  then $\Pi=\{\alpha_i=\varepsilon_i-\varepsilon_{i+1}\mid 1\leq i<n\}\cup \{\alpha_n=\varepsilon_n\}$,
  $\Phi=\{\pm \varepsilon_i\pm \varepsilon_j\mid 1\leq i\neq j\leq n\}\cup \{\pm\varepsilon_i\mid 1\leq i\leq n\}$.
  In this case, set $r_i=r^{(\alpha_i, \alpha_i)}, s_i=s^{(\alpha_i, \alpha_i)}$,
  then $r_1=\cdots=r_{n-1}=r^2$, $r_n=r$, $s_1=\cdots=s_{n-1}=s^2$, $s_n=s$.
  We also set $J=\{1\leq j\leq n=\text{rank}\,\mathfrak{g}\}$, and $J_0=J\cup \{0\}$.

  Given two sets of symbols $W=\{\omega_0, \omega_1, \cdots, \omega_n\}, W'=\{\omega_0', \omega_1', \cdots, \omega_n'\}$,
  define the structure constant matrix $(\langle\omega_i', \omega_j\rangle)_{i, j\in J_0}$ of type $B^{(1)}_n$ by
  $$\left(
      \begin{array}{ccccccc}
        r^2s^{-2} & r^{-2}s^{-2} &r^{-2} &1  &\cdots &1 &r^2s^2\\
        r^2s^2 & r^2s^{-2} & r^{-2} & 1 & \cdots & 1 & 1 \\
        s^2 & s^2 & r^2s^{-2} &r^{-2} & \cdots & 1 & 1 \\
        \cdots & \cdots &\cdots &\cdots & \cdots & \cdots & \cdots \\
        1 & 1 & 1&\cdots & \cdots & r^2s^{-2} & r^{-2} \\
       r^{-2}s^{-2}& 1 & 1 & \cdots &1 & s^2 & rs^{-1} \\
      \end{array}
    \right).\leqno{(*)}
  $$

  In this paper, $\mathfrak{g}$ denotes a simple Lie algebra over $\mathbb{C}$, and $\hat{\mathfrak{g}}$ is its corresponding untwisted affine Lie algebra.
\begin{defi}[Drinfeld-Jimbo presentation]\cite{HuZhang2014, Zhang phd 2007}\label{type B def}
  The unital associative algebra $U_{r,s}(\widehat{\mathfrak{g}})$ over $\mathbb{K}$ is generated by the elements $e_j, f_j, \omega_j^{\pm}, (\omega'_j)^{\pm}, \gamma^{\pm \frac{1}{2}}, (\gamma')^{\pm \frac{1}{2}}, D^{\pm}, (D')^{\pm}$, where $j\in J_0$, and satisfies the following relations (where $c$ is the canonical central element of $\widehat{\mathfrak{g}}$):

  \medskip
  $(\textrm{B1})$
  $\gamma^{\pm\frac{1}{2}}, (\gamma')^{\pm\frac{1}{2}}$ are central with $\gamma = (\omega'_\delta)^{-1}$, $\gamma' = (\omega_\delta)^{-1}$, $\gamma\gamma' = (rs)^c$, such that $\omega_i \omega_i^{-1} = \omega'_i(\omega'_i)^{-1} = 1 = DD^{-1} = D'D'^{-1}$, and
  \begin{align*}
    [\omega_i^{\pm 1}, \omega_j^{\pm 1}] &= [\omega_i^{\pm1}, D^{\pm1}] = [(\omega'_j)^{\pm1}, D^{\pm1}] = [\omega_i^{\pm1}, (D')^{\pm1}] = 0 \\
    &= [\omega_i^{\pm 1}, (\omega'_j)^{\pm 1}] = [(\omega'_j)^{\pm1}, (D')^{\pm1}] = [(D')^{\pm1}, D^{\pm1}] = [(\omega'_i)^{\pm 1}, (\omega'_j)^{\pm 1}].
  \end{align*}

  $(\textrm{B2})$ For $i, j\in J_0$,
  \begin{align*}
    D e_i D^{-1} &= r_i^{\delta_{0i}} e_i, \qquad\qquad D f_i D^{-1} = r_i^{-\delta_{0i}} f_i, \\
    \omega_j e_i (\omega_j)^{-1} &= \langle \omega'_i, \omega_j \rangle e_i, \qquad\quad \omega_j f_i (\omega_j)^{-1} = \langle \omega'_i, \omega_j \rangle^{-1} f_i.
  \end{align*}

  $(\textrm{B3})$ For $i, j\in J_0$,
  \begin{align*}
    D' e_i (D')^{-1} &= s_i^{\delta_{0i}} e_i, \qquad\qquad D' f_i (D')^{-1} = s_i^{-\delta_{0i}} f_i, \\
    \omega'_j e_i (\omega'_j)^{-1} &= \langle \omega'_j, \omega_i \rangle^{-1} e_i, \qquad\quad \omega'_j f_i (\omega'_j)^{-1} = \langle \omega'_j, \omega_i \rangle f_i.
  \end{align*}

  $(\textrm{B4})$ For $i, j\in J_0$, we have
  $$[\,e_i, f_j\,] = \frac{\delta_{ij}}{r_i - s_i}(\omega_i - \omega'_i).$$

  $(\textrm{B5})$ For any $i\ne j$, we have the $(r,s)$-Serre relations
  \begin{gather*}
    \bigl(\operatorname{ad}_l e_i\bigr)^{1-a_{ij}}(e_j) = 0, \\
    \bigl(\operatorname{ad}_r f_i\bigr)^{1-a_{ij}}(f_j) = 0,
  \end{gather*}
  where the left-adjoint action $\operatorname{ad}_l e_i$ and the right-adjoint action $\operatorname{ad}_r f_i$ are defined as follows:
  $$
  \operatorname{ad}_l a (b) = \sum_{(a)} a_{(1)} b S(a_{(2)}), \quad
  \operatorname{ad}_r a (b) = \sum_{(a)} S(a_{(1)}) b a_{(2)}, \quad
  \forall\; a, b\in U_{r,s}(\widehat{\mathfrak{g}}),
  $$
  and the coproduct $\Delta(a) = \sum_{(a)} a_{(1)} \otimes a_{(2)}$ is given by
  \begin{gather*}
    \Delta(\omega_i^{\pm1}) = \omega_i^{\pm1} \otimes \omega_i^{\pm1}, \qquad
    \Delta((\omega'_i)^{\pm1}) = (\omega'_i)^{\pm1} \otimes (\omega'_i)^{\pm1}, \\
    \Delta(e_i) = e_i \otimes 1 + \omega_i \otimes e_i, \qquad \Delta(f_i) = 1 \otimes f_i + f_i \otimes \omega'_i, \\
    \varepsilon(\omega_i^{\pm}) = \varepsilon((\omega'_i)^{\pm1}) = 1, \qquad
    \varepsilon(e_i) = \varepsilon(f_i) = 0, \\
    S(\omega_i^{\pm1}) = \omega_i^{\mp1}, \qquad
    S((\omega'_i)^{\pm1}) = (\omega'_i)^{\mp1}, \\
    S(e_i) = -\omega_i^{-1}e_i, \qquad S(f_i) = -f_i (\omega'_i)^{-1}.
  \end{gather*}
\end{defi}

\begin{remark}\label{rmk finaff}
  If we restrict the generators on $\{e_j, f_j, \omega_j, \omega_j'\mid j\in J\}$, we will obtain the definition of $U_{r,s}(\frak{so}_{2n+1})$, see \cite{BGH 2006}.
\end{remark}

\begin{defi} \cite{HuZhang2014} \label{def affine}
 The unital associative algebra $\mathcal{U}_{r,s}(\widehat{\frak {g}})$
 over $\mathbb{K}$  is generated by the
Drinfeld generators $x_j^{\pm}(k)$, $a_j(\ell)$, $\omega_j^{\pm1}$,
${\omega'_j}^{\pm1}$, $\gamma^{\pm\frac{1}{2}}$,
${\gamma'}^{\,\pm\frac{1}2}$, $D^{\pm1}$, $D'^{\,\pm1}$ $(j\in J, \, k,\,k' \in \mathbb{Z}$, $\ell\in \mathbb{Z}\backslash
\{0\})$, subject to the following defining relations:

\medskip
\noindent $(\textrm{D1})$ \  $\gamma^{\pm\frac{1}{2}}$,
$\gamma'^{\,\pm\frac{1}{2}}$ are central such that
$\gamma\gamma'=(rs)^c $,\,
$\omega_i\,\omega_i^{-1}=\omega_i'\,\omega_i'^{\,-1}=1$,
and for $i,\,j\in J$, one has
\begin{equation*}
\begin{split}
[\,\omega_i^{\pm 1},\omega_j^{\,\pm 1}\,]&=[\,\omega_i^{\pm1},
D^{\pm1}\,]=[\,\omega_j'^{\,\pm1}, D^{\pm1}\,] =[\,\omega_i^{\pm1},
D'^{\pm1}\,]=0\\
&=[\,\omega_i^{\pm 1},\omega_j'^{\,\pm 1}\,]=[\,\omega_j'^{\,\pm1},
D'^{\pm1}\,]=[D'^{\,\pm1}, D^{\pm1}]=[\,\omega_i'^{\pm
1},\omega_j'^{\,\pm 1}\,].
\end{split}
\end{equation*}
$$[\,a_i(\ell),a_j(\ell')\,]
=\delta_{\ell+\ell',0}\frac{ (\gamma\gamma')^{\frac{|\ell|}{2}}(\langle
i,\,i\rangle^{\frac{\ell a_{ij}}{2}}- \langle i,\,i\rangle^{\frac{-\ell
a_{ij}}{2}})} {|\ell|(r_i-s_i)}
\cdot\frac{\gamma^{|\ell|}-\gamma'^{|\ell|}}{r_j-s_j}.
 \leqno(\textrm{D2})
$$
$$[\,a_i(\ell),~\omega_j^{{\pm }1}\,]=[\,a_i(\ell),~\omega_j'^{\,\pm1}\,]=0.\leqno(\textrm{D3})
$$
$$
\begin{array}{lll}
D\,x_i^{\pm}(k)\,D^{-1}=r^k\, x_i^{\pm}(k), \qquad\ \
D'\,x_i^{\pm}(k)\,D'^{\,-1}=s^k\, x_i^{\pm}(k),
\\
D\, a_i(\ell)\,D^{-1}=r^\ell\,a_i(\ell), \qquad\qquad D'\,
a_i(\ell)\,D'^{\,-1}=s^\ell\,a_i(\ell).
\end{array}\leqno{(\textrm{D4})}
$$
$$
\omega_i\,x_j^{\pm}(k)\, \omega_i^{-1} =  \langle \omega_j',
\omega_i\rangle^{\pm 1} x_j^{\pm}(k), \qquad \omega'_i\,x_j^{\pm}(k)\,
\omega_i'^{\,-1} =  \langle \omega'_i, \omega_j\rangle
^{\mp1}x_j^{\pm}(k).\leqno(\textrm{D5})
$$
$$
\begin{array}{lll}
[\,a_i(\ell),x_j^{\pm}(k)\,]=\pm\frac{(\gamma\gamma')^{\frac{\ell}{2}} (\langle
i,\,i\rangle^{\frac{\ell a_{ij}}{2}}-\langle i,\,i\rangle^{\frac{-\ell
a_{ij}}{2}})}
{\ell(r_i-s_i)}\gamma'^{\pm\frac{\ell}{2}}x_j^{\pm}(\ell{+}{k}),\quad
\textit{for} \quad \ell>0,
\end{array}\leqno{(\textrm{D$6_1$})}
$$
$$
\begin{array}{lll}
[\,a_i(\ell),x_j^{\pm}(k)\,]=\pm\frac{(\gamma\gamma')^{\frac{-\ell}{2}}(\langle
i,\,i\rangle^{\frac{\ell a_{ij}}{2}}-\langle i,\,i\rangle^{\frac{-\ell
a_{ij}}{2}})}
{\ell(r_i-s_i)}\gamma^{\pm\frac{\ell}{2}}x_j^{\pm}(\ell{+}k), \quad
\textit{for} \quad \ell<0.
\end{array}\leqno{(\textrm{D$6_2$})}
$$
$$
\begin{array}{lll}
x_i^{\pm}(k{+}1)\,x_j^{\pm}(k') - \langle j,i\rangle^{\pm1} x_j^{\pm}(k')\,x_i^{\pm}(k{+}1)\\
=-\Bigl(\langle j,i\rangle\langle
i,j\rangle^{-1}\Bigr)^{\pm\frac1{2}}\,\Bigl(x_j^{\pm}(k'{+}1)\,x_i^{\pm}(k)-\langle
i,j\rangle^{\pm1} x_i^{\pm}(k)\,x_j^{\pm}(k'{+}1)\Bigr).
\end{array}\leqno{(\textrm{D7})}
$$

$$
[\,x_i^{+}(k),~x_j^-(k')\,]=\frac{\delta_{ij}}{r_i-s_i}\Big(\gamma'^{-k}\,{\gamma}^{-\frac{k+k'}{2}}\,
\omega_i(k{+}k')-\gamma^{k'}\,\gamma'^{\frac{k+k'}{2}}\,\omega'_i(k{+}k')\Big),\leqno(\textrm{D8})
$$
where $\omega_i(m)$, $\omega'_i(-m)~(m\in \mathbb{Z}_{\geq 0})$ such that
$\omega_i(0)=\omega_i$ and  $\omega'_i(0)=\omega_i'$ are defined as below:

\begin{gather*}
\omega_i(z)=\sum\limits_{m=0}^{\infty}\omega_i(m) z^{-m}=\omega_i \exp \Big(
(r_i{-}s_i)\sum\limits_{\ell=1}^{\infty}
 a_i(\ell)z^{-\ell}\Big),\quad \bigl(\omega_i(-m)=0, \  \forall\ m>0\bigr); \\
\omega'_i(z)=\sum\limits_{m=0}^{\infty}\omega'_i(-m) z^{m}=\omega'_i \exp
\Big({-}(r_i{-}s_i)
\sum\limits_{\ell=1}^{\infty}a_i(-\ell)z^{\ell}\Big), \quad
\bigl(\omega'_i(m)=0, \ \forall \ m>0\bigr).
\end{gather*}
$$x_i^{\pm}(m)x_j^{\pm}(k)=\langle j,i\rangle^{\pm1}x_j^{\pm}(k)x_i^{\pm}(m),
\qquad\ \hbox{for} \quad a_{ij}=0,\leqno(\textrm{D$9_1$})$$
$$
\begin{array}{lll}
& Sym_{m_1,\cdots
m_{n}}\sum_{k=0}^{n=1-a_{ij}}(-1)^k(r_is_i)^{\pm\frac{k(k-1)}{2}}
\Big[{1-a_{ij}\atop  k}\Big]_{\pm{i}}x_i^{\pm}(m_1)\cdots x_i^{\pm}(m_k) x_j^{\pm}(\ell)\\
&\hskip1.8cm \times x_i^{\pm}(m_{k+1})\cdots x_i^{\pm}(m_{n})=0,
\quad\hbox{for} \quad a_{ij}\neq 0, \quad  1\leq i<j\leq n,
\end{array} \leqno{(\textrm{D$9_2$})}
$$
$$
\begin{array}{lll}
& Sym_{m_1,\cdots
m_{n}}\sum_{k=0}^{n=1-a_{ij}}(-1)^k(r_is_i)^{\mp\frac{k(k-1)}{2}}
\Big[{1-a_{ij}\atop  k}\Big]_{\mp{i}}x_i^{\pm}(m_1)\cdots x_i^{\pm}(m_k) x_j^{\pm}(\ell)\\
&\hskip1.8cm \times x_i^{\pm}(m_{k+1})\cdots x_i^{\pm}(m_{n})=0,
\quad\hbox{for } \quad a_{ij}\neq 0, \quad  1\leq j<i\leq n,
\end{array} \leqno{(\textrm{D$9_3$})}
$$
where $[m]_{\pm{i}}=\frac{r_i^{\pm m}-s_i^{\pm m}}{r_i-s_i}$, $[m]_{\pm i}!=[m]_{\pm i}\cdots[2]_{\pm i}[1]_{\pm i}$,
$\Bigl[{m\atop n}\Bigr]_{\pm i}=\frac{[m]_{\pm i}!}{[n]_{\pm i}![m-n]_{\pm i}!}$.
\end{defi}

In Section \ref{sec specR}, we will denote $x_i^\pm(k), a_i(\ell)$ by $x_{ik}^\pm, a_{i\ell}$, respectively.

\begin{remark}
  Definition \ref{def affine} is also true for $\mathcal{U}_{r,s}(\widehat{\frak {so}_{2n+1}})$.
\end{remark}

\section{Basic braided $R$-matrix of $U_{r, s}(\mathfrak{so}_{2n+1})$} \label{sec basicR}
To derive the Faddeev-Reshetikhin-Takhtajan realization, we need to determine the basic braided $R$-matrix.
In one-parameter setting, Jantzen gave a strategy \cite{Jantzen} to construct $R$-matrices from the module category of $U_q(\mathfrak{g})$.
Benkart and Witherspoon extended this result to $U_{r, s}(\mathfrak{sl}_n)$ \cite{Benkart baiscR}.
For the other classical types, Begeron-Gao-Hu gave a unified description (including the type $A$ case):
\begin{theorem}\cite{BGH 2007}
   Let $M, M'$ be $U_{r, s}(\mathfrak{g})$-modules in $\mathcal{O}$ where $\mathfrak{g}=\mathfrak{so}_{2n+1}, \mathfrak{so}_{2n} \;\text{or}\; \mathfrak{sp}_{2n}$.
   Then the map
   $$R_{M, M'}=\Theta\circ\tilde{f}\circ P: M'\otimes M\rightarrow M\otimes M'$$
   is an isomorphism of $U_{r, s}(\mathfrak{g})$-modules, where $P: M'\otimes M\rightarrow M \otimes M'$ is
   the flip map such that $P(m'\otimes m)=m\otimes m'$ for any $m\in M, m' \in M'$.
\end{theorem}

\begin{remark}
  One can prove, $R_{M, M'}$ satisfies the braid relation.
  That is, for any $U_{r, s}(\mathfrak{g})$-modules $M, M', M^{\prime\prime}$,
  we have $R_{12}\circ R_{23}\circ R_{12}=R_{23}\circ R_{12}\circ R_{23}$.
  If we take $M=M'=M''=V$, where $V$ is the vector representation of $U_{r, s}(\mathfrak{so}_{2n+1})$,
then $R_{V, V}$ is the desired basic braided $R$-matrix.
\end{remark}

\subsection{Vector representation $V=V(\varepsilon_1)$}
\begin{lemm}\label{lemm vect rep}
The vector representation $T_1$ of $U_{r, s}(\mathfrak{so}_{2n+1})$ is given by

$\text{\rm(I)}$ For simple root elements, we have
\begin{gather*}
T_1(e_i)=E_{i, i+1}-(rs)^{-1}E_{(i+1)', i'},\\
T_1(f_i)=E_{i+1, i}-(rs)^{-1}E_{i', (i+1)'},\\
T_1(e_n)=(r+s)^{\frac{1}{2}}\Bigl((rs)^{-\frac{1}{2}}E_{n, n+1}-r^{-1}E_{n+1, n'}\Bigr),\\
T_1(f_n)=(r+s)^{\frac{1}{2}}\Bigl((rs)^{-\frac{1}{2}}E_{n+1, n}-s^{-1}E_{n', n+1}\Bigr).
\end{gather*}

$\text{\rm(II)}$ For group-like elements, we have
\begin{align*}
T_1(\omega_i)=\;&r^2 E_{ii}+s^2E_{i+1, i+1}+s^{-2}E_{(i+1)', (i+1)'}+r^{-2}E_{i', i'}+\sum_{j\neq i, i+1,\atop i',\; (i+1)'} E_{jj},\\
T_1(\omega_i')=\;&s^2 E_{ii}+r^2E_{i+1, i+1}+r^{-2}E_{(i+1)', (i+1)'}+s^{-2}E_{i', i'}+\sum_{j\neq i, i+1,\atop i',\; (i+1)'} E_{jj},
\\
T_1(\omega_n)=\;&rs^{-1}E_{n, n}+E_{n+1, n+1} +r^{-1}sE_{n', n'}+(rs)^{-1}\sum_{1\le j\le n-1}E_{jj}+rs\sum_{(n-1)'\le j\le 1'}E_{jj},\\
T_1(\omega'_n)=\;&r^{-1}sE_{n, n}+E_{n+1, n+1} +rs^{-1}E_{n', n'}+(rs)^{-1}\sum_{1\le j\le n-1}E_{jj}+rs\sum_{(n-1)'\le j\le 1'}E_{jj},
\end{align*}
where $1\le i\le n-1$, $i'= 2n+2-i$ $(\text{\rm i.e.}$, $(n+1)'=n+1$, $1<2<\cdots<n<n+1<n'<\cdots<2'<1'\,)$, $\{\,v_1, v_2,\cdots,v_n, v_{n+1}, v_{n'},\cdots, v_{2'}, v_{1'}\,\}$ is the canonical basis of $U_{r,s}(\mathfrak{so}_{2n+1})$-module of the first fundamental weight $\varepsilon_1$ with $v_1$ as highest weight vector $($cf. \cite{BGH 2007}$)$.
\end{lemm}
\begin{proof}
  We first verify that $T_1$ preserves the defining relations of $U_{r, s}(\mathfrak{so}_{2n+1})$
(see Definition \ref{type B def}, where we restrict to the finite type).
 Clearly, $T_1(\omega_i)$ and $T_1(\omega'_j)\;$ (for $1\leq i, j\leq n)$ commute with each other.
Thus, $T_1$ preserves (B1).

 For (B2) and (B3), we take verifying $T_1(\omega_j)T_1(e_i)=\langle \omega_i', \omega_j\rangle T_1(e_i)T_1(\omega_j)$ as an example, since the other equations can be verified similarly.

 (i) When $1\leq i, j\leq n-1$: we have
   $$T_1(\omega_j)T_1(e_i)=\left\{
                          \begin{array}{ll}
                            r^2E_{i, i+1}-r^{-1}s^{-3}E_{(i+1)', i'}, & \hbox{$j=i$,} \\
                            s^2E_{i, i+1}-(rs)^{-1}E_{(i+1)', i'}, & \hbox{$j=i-1$,} \\
                            E_{i, i+1}-r^{-3}s^{-1}E_{(i+1)', i'}, & \hbox{$j=i+1$.}
                          \end{array}
                        \right.
$$
$$T_1(e_i)T_1(\omega_j)=\left\{
                          \begin{array}{ll}
                            s^2E_{i, i+1}-r^{-3}s^{-1}E_{(i+1)', i'}, & \hbox{$j=i$,} \\
                            E_{i, i+1}-r^{-1}s^{-3}E_{(i+1)', i'}, & \hbox{$j=i-1$,} \\
                            r^2E_{i, i+1}-(rs)^{-1}E_{(i+1)', i'}, & \hbox{$j=i+1$.}
                          \end{array}
                        \right.
$$
In this case, $T_1(\omega_j)T_1(e_i)=\langle \omega_i', \omega_j\rangle T_1(e_i)T_1(\omega_j)$ holds.

(ii) When $1\leq i\leq n-1, j=n$:
$$T_1(\omega_n)T_1(e_i)=\left\{
                          \begin{array}{ll}
                            (rs)^{-1}E_{i, i+1}-E_{2n-i+1, 2n-i+2}, & \hbox{$1\leq i\leq n-2 $;} \\
                            (rs)^{-1}E_{i, i+1}-r^{-2}E_{2n-i+1, 2n-i+2}, & \hbox{$i=n-1$.}
                          \end{array}
                        \right.
$$
$$T_1(e_i)T_1(\omega_n)=\left\{
                          \begin{array}{ll}
                            (rs)^{-1}E_{i, i+1}-E_{2n-i+1, 2n-i+2}, & \hbox{$1\leq i\leq n-2 $;} \\
                            rs^{-1}E_{i, i+1}-E_{2n-i+1, 2n-i+2}, & \hbox{$i=n-1$.}
                          \end{array}
                        \right.
$$
Then $T_1(\omega_n)T_1(e_i)=\langle \omega_i^\prime, \omega_n\rangle T_1(e_i)T_1(\omega_n)$ holds.

(iii) When $1\leq j\leq n-1, i=n$:
$$T_1(\omega_j)T_1(e_n)=\left\{
                          \begin{array}{ll}
                            (r+s)^{\frac{1}{2}}((rs)^{-\frac{1}{2}}E_{n, n+1}-r^{-1}E_{n+1, n+2}),
                           & \hbox{$1\leq j\leq n-2 $;} \\
                            (r+s)^{\frac{1}{2}}(r^{-\frac{1}{2}}s^{\frac{3}{2}}E_{n, n+1}-r^{-1}E_{n+1, n+2}),
                            & \hbox{$j=n-1$.}
                          \end{array}
                        \right.
$$
$$T_1(e_n)T_1(\omega_j)=\left\{
                          \begin{array}{ll}
                            (r+s)^{\frac{1}{2}}((rs)^{-\frac{1}{2}}E_{n, n+1}-r^{-1}E_{n+1, n+2}),
                           & \hbox{$1\leq j\leq n-2 $;} \\
                            (r+s)^{\frac{1}{2}}((rs)^{-\frac{1}{2}}E_{n, n+1}-r^{-1}s^{-2}E_{n+1, n+2}),
                            & \hbox{$j=n-1$.}
                          \end{array}
                        \right.
$$
Then $T_1(\omega_j)T_1(e_n)=\langle \omega_n^\prime, \omega_j\rangle T_1(e_n)T_1(\omega_j)$ holds.

(iv) When $i=j=n$:
$$T_1(\omega_n)T_1(e_n)=(r+s)^{\frac{1}{2}}(r^{\frac{1}{2}}s^{-\frac{3}{2}}E_{n, n+1}-r^{-1}E_{n+1, n+2}),$$
$$T_1(e_n)T_1(\omega_n)=(r+s)^{\frac{1}{2}}((rs)^{-\frac{1}{2}}E_{n, n+1}-r^{-2}sE_{n+1, n+2}).$$
Then $T_1(\omega_n)T_1(e_n)=\langle \omega_n^\prime, \omega_n\rangle T_1(e_n)T_1(\omega_n)$ also holds.

 An argument similar to the one used earlier shows $T_1$ also preserves (B4):
 $$T_1(e_i)T_1(f_i)-T_1(f_i)T_1(e_i)=\dfrac{T_1(\omega_i)-T_1(\omega_i')}{r_i-s_i}.$$
Serre relation (B5) is also obviously satisfied since each summand on the left-hand side of the equations is zero.

It remains to verify that the highest weight is the first fundamental weight.
In fact, since $T_1(e_i)v_1=0 \ (i=1, 2, \cdots, n)$, we know $v_1$ is the highest weight vector.
By calculating
$$T_1(\omega_i) v_1= \langle \omega_1^\prime\omega_2^\prime\cdots \omega_n^\prime, \omega_i\rangle v_1, \quad \quad
T_1(\omega_i^\prime)v_1= \langle \omega_i^\prime, \omega_1\omega_2\cdots \omega_n\rangle^{-1} v_1, \quad 1\leq i\leq n,$$
we conclude that $v_1$ corresponds to the highest weight $\alpha_1+\alpha_2+\cdots+\alpha_n=\varepsilon_1$,
which is the first fundamental weight of $\mathfrak{so}_{2n+1}$.
\end{proof}
To simplify our notation, we denote $T_1(e_i)v_j$ by $e_i.\, v_j$, and so on.

\subsection{Decomposition of $V^{\otimes2}$}\label{subsec decompostion}
To determine the explicit formula of $R_{V, V}$,
it is necessary to work out the effect of $R$ acting on $V^{\otimes2}$.
Firstly, we give the decomposition of $V^{\otimes2}$ as follows:

\begin{lemm}
For vector representation $V$, we have
  $$V^{\otimes2}=V(0)\bigoplus V(2\varepsilon_1)\bigoplus V(\varepsilon_1{+}\varepsilon_2),$$
  where $V(0)$, $V(2\varepsilon_1)$ and $V(\varepsilon_1{+}\varepsilon_2)$ denotes the simple module with the highest weight $0$, $2\varepsilon_1$ and $\varepsilon_1{+}\varepsilon_2$, respectively.
\end{lemm}

\begin{proof}
  This follows from the fact that the braided categorical equivalence between $\mathcal O^{r,s}$ and $\mathcal O^q$ established in \cite{HuPeiCA 2012}, as well as the full subcategory  $\mathcal (O^q)^{(f)}$ of finite-dimensional modules being semisimple \cite{Klimyk}.
\end{proof}

Next we describe its simple submodules explicitly.

\begin{lemm}\label{lemm decomp1}
The module $S^{o}(V^{\otimes2})$ generated by $\sum\limits_{i=1}^{1'}a_iv_{i'}\otimes v_i$ is simple, which is isomorphic to $V(0)$, where $a_i=(rs^{-1})^{\rho_i}$ and
$$\rho_i=\left\{
          \begin{array}{ll}
            \frac{2n+1}{2}-i, & \hbox{$i<i'$;} \\
            -\rho_{i'}, & \hbox{$i\geq i'$.}
          \end{array}
        \right.
$$
\end{lemm}
\begin{proof}
Note that the one-dimensional trivial submodule of $V^{\otimes 2}$ is of weight $0$, which is a combination of those vectors $v_{i'}\otimes v_i$ of weight $0$, denote it by
$\sum_{i=1}^{1'}b_iv_{i'}\otimes v_i$. Obviously, $\omega_i$ and $\omega_i'$ (for $1\leq i\leq n$) act as the identity.

  (1) When $1\leq i\leq n-1$: we have trivial actions as follows
  \begin{gather*}
  \begin{split}
  &e_i\,.\Bigl(\sum_{j=1}^{1'}b_jv_{j'}\otimes v_j\Bigr)\\
&=(e_i\otimes 1+\omega_i\otimes e_i)\Bigl(b_{(i+1)'}v_{i+1}\otimes v_{(i+1)'}+b_{i'}v_i\otimes v_{i'}+b_iv_{i'}\otimes v_i+b_{i+1}v_{(i+1)'}\otimes v_{i+1}\Bigr)  \\
  &=\Bigl(b_{(i+1)'}-rs^{-1} b_{i'}
  \Bigr)v_i\otimes v_{(i+1)'}+\Bigl(b_{i+1}s^{-2}-(rs)^{-1}b_i\Bigr)v_{(i+1)'}\otimes v_i\\
&=0;\\
  &f_i\,.\Bigl(\sum_{j=1}^{1'}b_jv_{j'}\otimes v_j\Bigr)\\
&=(1\otimes f_i+f_i\otimes\omega_i')\Bigl(b_iv_{i'}\otimes v_i+b_{(i+1)'}v_{i+1}\otimes v_{(i+1)'}+b_{i'}v_i\otimes v_{i'}+b_{i+1}v_{(i+1)'}\otimes v_{i+1}\Bigr)  \\
 & =\Bigl(s^{-2}b_{i'}-(rs)^{-1} b_{(i+1)'}\Bigr)v_{i+1}\otimes v_{i'}+(b_i-rs^{-1}b_{i+1})v_{i'}\otimes v_{i+1}\\
&=0.
\end{split}
\end{gather*}
Thus we get a geometric sequence: $b_i =rs^{-1}b_{i+1}$ and $b_{(i+1)'}=rs^{-1} b_{i'}$ $(1\le i<n)$ with the ordering $1<2<\cdots<n-1$, and $(n-1)'<\cdots<2'<1'$.

(2) When $i=n$:
\begin{gather*}
\begin{split}
   &e_n\,.\Bigl(\sum_{j=1}^{1'}b_jv_{j'}\otimes v_j\Bigr)  \\
&=(e_n\otimes 1+\omega_n\otimes e_n)\Bigl(b_{n+1}v_{n+1}\otimes v_{n+1}+b_nv_{n'}\otimes v_n+b_{n'}v_{n}\otimes v_{n'}\Bigr) \\
&=(r+s)^{\frac{1}{2}}\Bigl((r^{-\frac{1}{2}}s^{-\frac{1}{2}}b_{n+1}-s^{-1}b_{n'})v_n\otimes v_{n+1}
+(r^{-\frac{1}{2}}s^{-\frac{1}{2}}b_{n+1}-r^{-1}b_n)v_{n+1}\otimes v_n\Bigr)\\
&=0;\\
&f_n\,.\Bigl(\sum_{j=1}^{1'}b_jv_{j'}\otimes v_j\Bigr)  \\
&=(1\otimes f_n+f_n\otimes\omega_n') \Bigl(b_{n+1}v_{n+1}\otimes v_{n+1}+b_nv_{n'}\otimes v_n+b_{n'}v_{n}\otimes v_{n'}\Bigr) \\
&=(r+s)^{\frac{1}{2}}\Bigl((r^{\frac{1}{2}}s^{-\frac{3}{2}}b_{n'}-s^{-1}b_{n+1})v_{n+1}\otimes v_{n'}
+(r^{-\frac{1}{2}}s^{-\frac{1}{2}}b_n-s^{-1}b_{n+1})v_{n'}\otimes v_{n+1}\Bigr)\\
&=0.
\end{split}
\end{gather*}
From this,  we get $b_n=(rs^{-1})^{\frac{1}{2}}b_{n+1}$ and $b_{n+1}=(rs^{-1})^{\frac{1}{2}}b_{n'}$.

Without loss of generality, we may assume $b_{n+1}=1$.
Then we conclude that $b_i=(rs^{-1})^{\rho_i}=a_i$ for all $1\leq i\leq 1'$.
\end{proof}

\begin{defi}\label{metric matrix}
  We introduce the quantum metric matrix $C=(C_j^i)$ for $U_{r,s}(\mathfrak{so}_{2n+1})$ is
  $$C_j^i=\delta_{i j'}\bigl(r^{-1}s\bigr)^{\rho_i},$$
  where $i, j$ represent for the row and column index, respectively.
  Obviously, $C^2=I$, the identity matrix.
\end{defi}

\begin{lemm}\label{lemm decomp2}
The following elements span a simple submodule of $V^{\otimes2}$, denoted by $\mathcal{S}'(V^{\otimes2})$, which is isomorphic to $V(2\varepsilon_1):$\\
$\text{\rm (i)}$\;\;\,  $v_i\otimes v_i$, $1\le i\le n$ or $n'\le i\le 1'$,\vspace{0.15em}\\
$\text{\rm (ii)}$\;\,  $v_i\otimes v_{j}+r^{-1}s\,v_{j}\otimes v_i$, $1\le i\le n, j=n+1$ or $i=n+1, n'\le j\le 1',$\vspace{0.15em} \\
$\text{\rm (iii)}$\, $v_i\otimes v_j+s^2v_j\otimes v_i$, $1\le i\le n$, $i+1\le j\le n$ or $(i-1)'\le j\le 1'$,\vspace{0.15em}\\
$\text{\rm (iv)}$\, $v_i\otimes v_j+r^{-2}v_j\otimes v_i$, $1\le i\le n-1$, $n'\le j\le (i+1)'$ or $n'\le i\le 2'$, $i+1\le j\le 1'$,\vspace{0.15em}\\
$\text{\rm (v)}$\;\, $v_n\otimes v_{n'}+r^{-2}s^2v_{n'}\otimes v_n-\Bigl(r^{-\frac{3}{2}}s^{\frac{3}{2}}+r^{-\frac{1}{2}}s^{\frac{1}{2}}\Bigr)v_{n+1}\otimes v_{(n+1)'}$,\vspace{0.15em}\\
$\text{\rm (vi)}$\, $v_i\otimes v_{i'}+r^{-2}s^{2}v_{i'}\otimes v_i-r^{-1}s\Bigl(v_{i+1}\otimes v_{(i+1)'}
+v_{(i+1)'}\otimes v_{i+1}\Bigr)$, $1\le i\le n-1$,\\
where $v_1\otimes v_1$ is the highest weight vector with the highest weight $2 \varepsilon_1$.
\end{lemm}

\begin{proof}
   According to Definition 2.4 of \cite{BGH 2007}, $v_1\otimes v_1$ is the highest weight vector, corresponding to the highest weight $2 \varepsilon_1$.
   One can derive the rest generators by iteratively using the action of $e_i$ and $f_i$, $i=1,2,\cdots,n$.
   The detailed verification can be found in \cite{Xu thesis}.
\end{proof}

\begin{lemm}\label{lemm decomp3}
The following elements span a simple submodule of $V^{\otimes2}$, denoted by $\Lambda(V^{\otimes2})$, which is isomorphic to $V(\varepsilon_1{+}\varepsilon_2):$\\
$\text{\rm (i)}$\;\;\,  $v_i\otimes v_j-r^2 v_j\otimes v_i$, $1\le i\le n$, $i+1\le j\le n$ or $(i-1)'\le j\le 1'$,\vspace{0.15em}\\
$\text{\rm (ii)}$\;\,  $v_i\otimes v_{j}-r^{-1}s\,v_{j}\otimes v_i$, $1\le i\le n, j=n+1$ or $i=n+1, n'\le j\le 1',$ \vspace{0.15em}\\
$\text{\rm (iii)}$\, $v_i\otimes v_j-s^{-2}v_j\otimes v_i$, $1\le i\le n-1$, $n'\le j\le (i+1)'$ or $n'\le i\le 2'$, $i+1\le j\le 1'$,\vspace{0.15em}\\
$\text{\rm (iv)}$\; $v_n\otimes v_{n'}-v_{n'}\otimes v_n-\Bigl(r^{\frac{1}{2}}s^{-\frac{1}{2}}+r^{-\frac{1}{2}}s^{\frac{1}{2}}\Bigr)v_{n+1}\otimes v_{(n+1)'}$,\vspace{0.15em}\\
$\text{\rm (v)}$\;\, $v_i\otimes v_{i'}-v_{i'}\otimes v_i-r^{-1}s\,v_{i+1}\otimes v_{(i+1)'}
+rs^{-1}v_{(i+1)'}\otimes v_{i+1}$, $1\le i\le n-1$,\vspace{0.15em}\\
where the highest weight vector is $v_1\otimes v_2-r^2v_2\otimes v_1$, with respect to the highest weight $\varepsilon_1+\varepsilon_2$.
\end{lemm}

\begin{proof}
  One can prove this lemma similarly to the preceding lemma.
\end{proof}

Correspondingly, we easily get the following
\begin{lemm}\label{lemma minimal polynomial}
  The minimal polynomial of $R=R_{V, V}$  on $V^{\otimes2}$ is $$(t-r^{-1}s)(t+rs^{-1})(t-r^{2n}s^{-2n}).$$
\end{lemm}

\begin{proof}
By the foregoing Lemmas,
$\mathcal{S}^0(V^{\otimes2})$, $S^{\prime}(V^{\otimes2})$ and $\Lambda(V^{\otimes2})$
are cyclic modules generated by their highest weight vectors.
By definition of $R_{V, V}$, we can calculate that
\begin{align*}
  R\Bigl(v_1\otimes v_2-r^2v_2\otimes v_1\Bigr) &=\Theta\circ\tilde{f}\Bigl(v_2\otimes v_1-r^2v_1\otimes v_2\Bigr) \\
  &=\langle\omega'_{\alpha_1+\cdots+\alpha_n}, \omega_{\alpha_2+\cdots+\alpha_n}\rangle^{-1}\Theta(v_2\otimes v_1) \\
   &\qquad-r^2\langle \omega'_{\alpha_2+\dots+\alpha_n}, \omega_{\alpha_1+\cdots+\alpha_n} \rangle^{-1}\Theta(v_1\otimes v_2)\\
   &=(rs)(1\otimes1)(v_2\otimes v_1)-rs^{-1}\Bigl[1\otimes1+(s^2-r^2)f_1\otimes e_1\Bigr](v_1\otimes v_2)\\
   &=-rs^{-1}\Bigl(v_1\otimes v_2-r^2v_2\otimes v_1\Bigr);
\end{align*}
Similarly,
$$R\Bigl(v_1\otimes v_1\Bigr)=r^{-1}s\,v_1\otimes v_1.$$
One can also prove
$$R\Bigl(\sum_{i=1}^{2n+1}a_iv_{i'}\otimes v_i\Bigr)=r^{2n}s^{-2n}\Bigl(\sum_{i=1}^{2n+1}a_iv_{i'}\otimes v_i\Bigr)$$
by comparing the coefficient of $v_1\otimes v_{1'}$ on both sides.
Indeed, the left's $v_1\otimes v_{1'}$ only comes from the term $R(a_1 v_{1'}\otimes v_1)$.
By direct calculations,
\begin{align*}
  R\Bigl(a_1 v_{1'}\otimes v_1\Bigr) & =\; r^{\frac{2n-1}{2}}s^{-\frac{2n-1}{2}}\Theta\circ\tilde{f}\Bigl(v_{1}\otimes v_{1'}\Bigr) \\
   &=\; r^{\frac{2n+1}{2}}s^{-\frac{2n+1}{2}}\Theta\circ(v_{1}\otimes v_{1'})\\
   &=\; r^{\frac{2n+1}{2}}s^{-\frac{2n+1}{2}}v_1\otimes v_{1'}+ \textit{Other Terms}.
\end{align*}
We remark that $v_1\otimes v_{1'}$ does not appear in `` Other Terms''.
Then $R_{\mathcal{S}'(V^{\otimes2})}$, $R_{\mathcal{S}^0(V^{\otimes2})}$ and $R_{\Lambda(V^{\otimes2})}$ have the corresponding eigenvalues $r^{-1}s$, $r^{2n}s^{-2n}$ and $-rs^{-1}$, respectively.
This completes the proof.
\end{proof}

\subsection{Formula of the basic braided $R$-matrix}
Now we can establish the explicit formula of the basic braided $R$-matrix.
\begin{theorem}\label{main theorem}
  The formula of $R=R_{V, V}$ is
   \begin{align*}
R=r^{-1}s&\sum_{i \atop i\neq i'}E_{ii}\otimes E_{ii}+r^{-1}s^{-1}\Bigl(\sum_{1\leq i\leq n \atop i+1\leq j\le n}E_{ij}\otimes E_{ji}+\sum_{2\leq i\leq n \atop (i-1)'\leq j\leq 1'}E_{ij}\otimes E_{j i}\\
+&\sum_{1\leq i\leq n-1 \atop n'\leq j\leq (i+1)'}E_{ji}\otimes E_{ij}+\sum_{n'\leq i\leq 2'\atop i+1\leq j\leq 1'}E_{ji}\otimes E_{ij}\Bigr)+(rs)\Bigl(\sum_{1\leq i\leq n\atop i+1\leq j\leq n}E_{ji}\otimes E_{ij} \\
+&\sum_{n'\leq i\leq 2'\atop i+1\leq j\leq 1'}E_{ij}\otimes E_{ji}+\sum_{2\leq i\leq n\atop (i-1)'\leq j\leq 1'}E_{ji}\otimes E_{ij}+\sum_{1\leq i\leq n-1\atop n'\leq j\leq (i+1)'}E_{i j}\otimes E_{j i}\Bigr)\\
+&\sum_{i\atop i\neq i'} E_{i, n+1}\otimes E_{n+1,i}+\sum_{j\atop j\neq j'}E_{n+1, j}\otimes E_{j, n+1}   +rs^{-1}\sum_{i\atop i\neq i'}E_{i' i}\otimes E_{ii'}\\
+&\Bigl(r^{-1}s-rs^{-1}\Bigr)\Bigl\{\sum_{i,j\atop i> j}E_{ii}\otimes E_{jj}-\sum_{i,j\atop i> j}\Bigl(r^{-1}s\Bigr)^{(\rho_i-\rho_j)}E_{ij'}\otimes E_{i' j}\Bigr\}\\
+&E_{n+1, n+1}\otimes E_{n+1, n+1},
\end{align*}
where $\rho_i:=\left\{
            \begin{array}{ll}
              \frac{2n+1}{2}-i, & \hbox{if $i<n+1,$} \\
              -\rho_{i'}, & \hbox{if $i\geq n+1$.}
            \end{array}
          \right.
$
\end{theorem}

To prove this Theorem,
it suffices to show that the effect of $R$ acting on $V^{\otimes2}$ is equivalent to that of the minimal polynomial given in Lemma\;\ref{lemma minimal polynomial}.
We present the explicit verification in the Lemmas \ref{Lemm verifyV0} and \ref{Lemma verifyV+-}.

\begin{lemm}\label{Lemm verifyV0}
  $R$ acts on $\mathcal{S}^0(V^{\otimes2})$ as scalar multiplication, with eigenvalue $r^{2n}s^{-2n}$,
  that is to say:
   $$R\Bigl(\sum_{i=1}^{2n+1}a_iv_{i'}\otimes v_i\Bigr)=r^{2n}s^{-2n}\Bigl(\sum_{i=1}^{2n+1}a_iv_{i'}\otimes v_i\Bigr).$$
\end{lemm}
\begin{proof}
  (1) First, we calculate the effect of $R$ acting on each summand on the left-hand side of the equation.

 (i) Assume $1\leq i\leq n$. Since
  $$R=rs^{-1}E_{ii'}\otimes E_{i' i}+\Bigl(r^{-1}s-rs^{-1}\Bigr)\Bigl[E_{i' i'} \otimes E_{ii}
-\sum_{j>i}(r^{-1}s)^{(\rho_j-\rho_i)}E_{ji'}\otimes E_{j' i}\Bigr]+\cdots$$ $(\text{we ignore those items acting as zero})$, we have
\begin{gather*}
\begin{split}
   & R\;\Bigl(a_iv_{i'}\otimes v_i\Bigr) \\
  &=(rs^{-1})^{\frac{2n+3-2i}{2}}\Bigl\{rs^{-1}v_i\otimes v_{i'}{+}(r^{-1}s{-}rs^{-1})\Bigl[v_{i'}\otimes v_i{-}\sum_{j>i}(r^{-1}s)^{(\rho_j-\rho_i)} v_j\otimes v_{j'}\Bigr]\Bigr\}\\
  &=(rs^{-1})^{\frac{2n+3-2i}{2}}v_i\otimes v_{i'}-\sum_{i<j\leq n}\Bigl(rs^{-1}\Bigr)^{\frac{2n+1-2i}{2}}(r^{-1}s{-}rs^{-1})
  \Bigl\{(r^{-1}s)^{i-j}v_j\otimes v_{j'} \\
   &\quad+\Bigl[1{-}(r^{-1}s)^{-2n+2i-1}\Bigr]v_{i'}{\otimes} v_i{-}(r^{-1}s)^{i-n-\frac{1}{2}}v_{n+1}{\otimes} v_{n+1}{-}\sum_{j\geq n' \atop j\neq i'}(r^{-1}s)^{i-j+1}v_j{\otimes} v_{j'}\Bigr\}.
  \end{split}
\end{gather*}

(ii) Next we consider the case of $i=n+1$, \. Notice that
$$R=-(r^{-1}s-rs^{-1})\sum_{i>n+1}(r^{-1}s)^{-i+n+\frac{3}{2}}E_{i, n+1}\otimes E_{i', n+1}+E_{n+1, n+1}\otimes E_{n+1, n+1}+\cdots$$
we obtain that
$$R(v_{n+1}\otimes v_{n+1})=-(r^{-1}s-rs^{-1})\sum_{i>n+1}(r^{-1}s)^{-i+n+\frac{3}{2}}v_i\otimes v_{i'}+v_{n+1}\otimes v_{n+1}.$$

(iii) For $n'\leq i\leq 1'$. Since
$$R=rs^{-1}E_{ii'}\otimes E_{i' i}-(r^{-1}s-rs^{-1})\sum_{j>i}(r^{-1}s)^{\rho_j-\rho_i}E_{ji'}\otimes E_{j' i}+\cdots$$
we conclude that
$$R(a_iv_{i'}\otimes v_i)=r^{\frac{2n+3-2i}{2}}s^{-\frac{2n+3-2i}{2}}(rs^{-1}v_i\otimes v_{i'}
-(r^{-1}s-rs^{-1})\sum_{j>i}(r^{-1}s)^{i-j}v_j\otimes v_{j'}).$$

(2) Then we  assume $$R\Bigl(\sum_{i=1}^{2n+1}a_iv_{i'}\otimes v_i\Bigr)=\sum_{k=1}^{2n+1}b_k(v_{k'}\otimes v_k)$$
  It suffices to show that $b_k=r^{2n}s^{-2n}a_k \ (1\leq k\leq 2n+1)$,
where $b_k$ can be calculated from (1):

(i) When $1\leq k\leq n:$
\begin{gather*}
\begin{split}
   b_k&=(r^{-1}s-rs^{-1})\Bigl[(rs^{-1})^{\frac{2n+1-2k}{2}}(1-(r^{-1}s)^{-2n+2k-1})\\
  &\quad-(rs^{-1})^{\frac{2n+1-2k}{2}}(1-(r^{-1}s)^{-2n+2k-1})-\sum_{i=1}^{n}(rs^{-1})^{\frac{2n+1-2i}{2}}(r^{-1}s)^{i-k'+1} \\
   &\quad-\sum_{n+2\leq i<k'}(rs^{-1})^{\frac{2n+3-2i}{2}}(r^{-1}s)^{i-k'}-(r^{-1}s)^{-n+k-\frac{1}{2}}\Bigr]\\
   &=(rs^{-1})^{3n-k+\frac{1}{2}}=(rs^{-1})^{2n} a_k.
\end{split}
\end{gather*}

 (ii) For the case of $k=n+1$:
\begin{align*}
  b_{n+1}\; &= -\sum_{1\leq i\leq n} r^{\frac{2n+1-2i}{2}}s^{-\frac{2n+1-2i}{2}}(r^{-1}s-rs^{-1})(r^{-1}s)^{i-n-\frac{1}{2}}+1 \\
  &=r^{2n}s^{-2n}= r^{2n}s^{-2n}a_{n+1}.
\end{align*}

(iii) If $n'\leq k\leq 1'$:
\begin{align*}   b_k\;&=r^{-\frac{2n+1-2k}{2}}s^{\frac{2n+1-2k}{2}}-\sum_{i<k'}r^{\frac{2n+1-2i}{2}}s^{-\frac{2n+1-2i}{2}}(r^{-1}s-rs^{-1})(r^{-1}s)^{-2n+i-2+k}\\
&= r^{3n-k+\frac{3}{2}}s^{-(3n-k+\frac{3}{2})}=r^{2n}s^{-2n} (r^{n+\frac{3}{2}-k}s^{-(n+\frac{3}{2}-k)})= r^{2n}s^{-2n} a_k.
\end{align*}

Thus we complete our proof.
\end{proof}

\begin{lemm}\label{Lemma verifyV+-}
  $R$ acts on $\mathcal{S}'(V^{\otimes2})$ and $\Lambda(V^{\otimes2})$ as scalar multiplication,
  with eigenvalue $r^{-1}s$ and $-rs^{-1}$, respectively.
\end{lemm}

\begin{proof}
It suffices to verify that $R$ acts as scalar multiplication  on the generators of $\mathcal{S}'(V^{\otimes2})$ and $\Lambda(V^{\otimes2})$, where these generators are derived in Lemmas \ref{lemm decomp2} and \ref{lemm decomp3}.
Noticing that
\begin{gather*}
\begin{split}
  R&=rs^{-1}E_{n, n'}\otimes E_{n', n}+rs^{-1}E_{n', n}\otimes E_{n, n'}+E_{n+1, n+1}\otimes E_{n+1, n+1}{-}(r^{-1}s{-}rs^{-1})\\
    &\cdot\Bigl\{\sum_{i>n+1}(r^{-1}s)^{n-i+\frac{3}{2}}E_{i, n+1}\otimes E_{i', n+1}{+}(rs^{-1})^\frac{1}{2}E_{n+1, n'}\otimes E_{n+1, n}{-}E_{n', n'}\otimes E_{n, n}\\
    &+\sum_{i>n+2}(r^{-1}s)^{n-i+2}E_{i, n}\otimes E_{i', n'}+\sum_{i>n+1}(r^{-1}s)^{n-i+1}E_{i, n'}\otimes E_{i', n}\Bigr\}+\cdots,
\end{split}
\end{gather*}
we have
\begin{gather*}
\begin{split}
   & R\Bigl[v_n\otimes v_{n'}+r^{-2}s^2v_{n'}\otimes v_n-(r^{-\frac{3}{2}}s^{\frac{3}{2}}+r^{-\frac{1}{2}}s^{\frac{1}{2}})v_{n+1}\otimes v_{(n+1)'}\Bigr] \\
  &\quad=r^{-1}s\Bigl[v_n\otimes v_{n'}+r^{-2}s^2v_{n'}\otimes v_n-(r^{-\frac{3}{2}}s^{\frac{3}{2}}+r^{-\frac{1}{2}}s^{\frac{1}{2}})v_{n+1}\otimes v_{(n+1)'}\Bigr]; \\
  & R\Bigl[v_n\otimes v_{n'}-v_{n'}\otimes v_n-(r^{\frac{1}{2}}s^{-\frac{1}{2}}+r^{-\frac{1}{2}}s^{\frac{1}{2}})v_{n+1}\otimes v_{(n+1)'}\Bigr] \\
&\quad=-rs^{-1}\Bigl[v_n\otimes v_{n'}-v_{n'}\otimes v_n-(r^{\frac{1}{2}}s^{-\frac{1}{2}}+r^{-\frac{1}{2}}s^{\frac{1}{2}})v_{n+1}\otimes v_{(n+1)'}\Bigr].
\end{split}
\end{gather*}

The effect on the other generators can be checked similarly.
\end{proof}

Assume $\hat{R}=P\circ R$, where $P$ is the standard flip map.
Then $\hat{R}$ satisfies the crossing symmetry relations,
which is analogous to that in the one-parameter case (see \cite{FRT Russ 1989}).
\begin{prop} \label{prop crossing symmetry finite}
  We have the following crossing symmetry relations:
  \begin{equation}\label{equ crossing relation finite}
    \hat{R}=C_1(\hat{R}^{t_1})^{-1}(C_1)^{-1}=C_2(\hat{R}^{-1})^{t_2}(C_2)^{-1},
  \end{equation}
  where $C_1=C\otimes I$, $C_2=I\otimes C$, $I$ denotes the identity matrix, and $t_i\,(i=1,2)$ denotes the standard matrix transpose with respect to the $i$-th tensor factor.
\end{prop}

\subsection{Representation of Birman-Wenzl-Murakami algebra}\label{ssec BMW}
We will establish a representation of Birman-Wenzl-Murakami algebra via the basic braided $R$-matrix in this subsection.

Firstly, we can derive the inverse of the basic braided $R$-matrix as follows:
\begin{lemm}\label{lemm Rinverse}
   \begin{align*}
R^{-1}=rs^{-1}&\sum_{i \atop i\neq i'}E_{ii}\otimes E_{ii}+r^{-1}s^{-1}\Bigl(\sum_{1\leq i\leq n \atop i+1\leq j\le n}E_{ij}\otimes E_{ji}+\sum_{2\leq i\leq n \atop (i-1)'\leq j\leq 1'}E_{ij}\otimes E_{j i}\\
+&\sum_{1\leq i\leq n-1 \atop n'\leq j\leq (i+1)'}E_{ji}\otimes E_{ij}+\sum_{n'\leq i\leq 2'\atop i+1\leq j\leq 1'}E_{ji}\otimes E_{ij}\Bigr)+(rs)\Bigl(\sum_{1\leq i\leq n\atop i+1\leq j\leq n}E_{ji}\otimes E_{ij} \\
+&\sum_{n'\leq i\leq 2'\atop i+1\leq j\leq 1'}E_{ij}\otimes E_{ji}+\sum_{2\leq i\leq n\atop (i-1)'\leq j\leq 1'}E_{ji}\otimes E_{ij}+\sum_{1\leq i\leq n-1\atop n'\leq j\leq (i+1)'}E_{i j}\otimes E_{j i}\Bigr)\\
+&\sum_{i\atop i\neq i'} E_{i, n+1}\otimes E_{n+1,i}+\sum_{j\atop j\neq j'}E_{n+1, j}\otimes E_{j, n+1}   +r^{-1}s\sum_{i\atop i\neq i'}E_{i' i}\otimes E_{ii'}\\
+&\Bigl(rs^{-1}-r^{-1}s\Bigr)\Bigl\{\sum_{i,j\atop i<j}E_{ii}\otimes E_{jj}-\sum_{i,j\atop i< j}\Bigl(r^{-1}s\Bigr)^{(\rho_i-\rho_j)}E_{ij'}\otimes E_{i' j}\Bigr\}\\
+&E_{n+1, n+1}\otimes E_{n+1, n+1},
\end{align*}
$\quad$where $\rho_i$ is defined in Theorem \ref{main theorem}.

\end{lemm}

\begin{proof}
    By Lemma \ref{lemma minimal polynomial}, $R$ has the spectral decomposition:
    \begin{equation}\label{equ specdec}
      R=r^{-1}sP_+-rs^{-1}P_-+(rs^{-1})^{2n}P_0,
    \end{equation}
    where
    \begin{align*}
      P_+ &=\;\frac{R^2-\Bigl((rs^{-1})^{2n}-rs^{-1}\Bigr)R-(rs^{-1})^{2n+1}I}{(r^{-1}s+rs^{-1})\Bigl(r^{-1}s-(rs^{-1})^{2n}\Bigr)}, \\
      P_- &=\;\frac{R^2-\Bigl(r^{-1}s+(rs^{-1})^{2n}\Bigr)R+(rs^{-1})^{2n-1}I}{(rs^{-1}+r^{-1}s)\Bigl(rs^{-1}+(rs^{-1})^{2n}\Bigr)},\\
      P_0&=\;\frac{R^2-(r^{-1}s-rs^{-1})R-I}{\Bigl((rs^{-1})^{2n}-r^{-1}s\Bigr)\Bigl(rs^{-1}+(rs^{-1})^{2n}\Bigr)}.
    \end{align*}
   Here $P_+, P_-,P_0$ are the projections onto $S^{\prime}(V^{\otimes2})$,\;$\Lambda(V^{\otimes2})$ and $\mathcal{S}^0(V^{\otimes2})$, respectively.
   Using (\ref{equ specdec}), one can easily verify that
  \begin{equation}\label{equ specdec2}
      R^{-1}=rs^{-1} P_+-r^{-1}sP_-+(rs^{-1})^{-2n}P_0,
    \end{equation}
and thus
\begin{equation}\label{}
  R^{-1}=P\circ R(r^{-1}, s^{-1})\circ P,
\end{equation}
   where $R(r^{-1}, s^{-1})$ denotes the matrix $R$ with $r$ and $s$ replaced by $r^{-1}$ and $s^{-1}$, respectively.
\end{proof}

It can be easily checked:
\begin{coro}\label{cor R R-1}
 The basic braided $R$-matrix and its inverse satisfy the Kauffman skein relation:
  $$R-R^{-1}=(r^{-1}s-rs^{-1})\left(I-K\right),$$
  with $K:=\sum_{i,j=1}^{2n+1}(r^{-1}s)^{\rho_i-\rho_j}E_{ij'}\otimes E_{i'j}$.
\end{coro}

\begin{coro}\label{cor Kquasi}
  The matrix $K$ is a quasi-idempotent element. To be specific,
  $$K^2=\Bigl(rs^{-1}-r^{-1}s-(r^{-1}s)^{2n}+(rs^{-1})^{2n}\Bigr)(rs^{-1}-r^{-1}s)^{-1}K.$$
\end{coro}

\begin{proof}
   Considering the action of $K$ on $V^{\otimes 2}$, we find that $K$ is a complex multiple of $P_0$.
   Indeed,
   \begin{align*}
K&\Bigl(\sum_{i=1}^{2n+1}(rs^{-1})^{\rho_i}v_{i'}\otimes v_i\Bigr)
     =\sum_{i=1}^{2n+1}(rs^{-1})^{2\rho_i}\Bigl(\sum_{j=1}^{2n+1}(rs^{-1})^{\rho_j}v_{j'}\otimes v_j\Bigr)\\
     &\quad=\Bigl(rs^{-1}-r^{-1}s-(r^{-1}s)^{2n}+(rs^{-1})^{2n}\Bigr)(rs^{-1}-r^{-1}s)^{-1} \Bigl(\sum_{j=1}^{2n+1}(rs^{-1})^{\rho_j}v_{j'}\otimes v_j\Bigr),
   \end{align*}
   and $K$ acts on $S^{\prime}(V^{\otimes2})$ and $\Lambda(V^{\otimes2})$ as zero obviously.
   Using the spectral decompositions (\ref{equ specdec}) and (\ref{equ specdec2}), we conclude that
   \begin{equation}\label{equ KP}
     KR^\pm=(rs^{-1})^{\pm2n}K.
   \end{equation}
   Then we get the desired equation via Corollary \ref{cor R R-1} (by multiplying both sides by $K$).
\end{proof}

For convenience, we write
 $$R=\sum R_{ik}^{jl}E_{ij}\otimes E_{kl},\quad K=\sum K_{ik}^{jl}E_{ij}\otimes E_{kl},$$
where $R_{ik}^{jl}$ and $K_{ik}^{jl}$ are the coefficients calculated in Theorem \ref{main theorem} and Corollary \ref{cor R R-1}, respectively.
\begin{coro}\label{cor Kij}
  We have the following identities:
\begin{equation}\label{}
    K^{ij}_{kl}=(r^{-1}s)^{\rho_i+\rho_k}\delta_{ij'}\delta_{kl'}, \quad 1\leq i,j,k,l \leq 2n+1;
\end{equation}
\begin{equation}\label{}
  \sum_{j=1}^{2n+1}R_{ij}^{ij}(r^{-1}s)^{2\rho_j}=(r^{-1}s)^{2n}, \quad 1\leq i\leq 2n+1.
\end{equation}
\end{coro}

\begin{proof}
These identities are directly from Theorem \ref{main theorem}, Corollary \ref{cor R R-1} and the definition of $\rho_i$.
\end{proof}

\begin{coro}\label{cor KRcross}
  $$K_{i,i+1}R_{i+1,i+2}^\pm K_{i,i+1}=(r^{-1}s)^{\pm2n} K_{i,i+1}, \quad 1\leq i\leq t-1$$
  where $K_{i, i+1}$ and $R_{i, i+1}^\pm$ are the operators on $V^{\otimes t}$ acting as $K$ and $R^\pm$, respectively,
  in the $i$-th and $(i+1)$-th tensor factors and as the identity elsewhere.
\end{coro}

\begin{proof}
We only prove the positive case since the negative one can be proved similarly.
By direct calculations, we have
  \begin{align*}
     (K&\otimes I)(I\otimes R)(K\otimes I) \\
    &\quad=\sum_{a,c,d,f}\Bigl((r^{-1}s)^{\rho_a+\rho_{c'}}E_{ac'}\otimes E_{a'c}\otimes I\Bigr)\Bigl(R_{cd}^{cd}\otimes E_{cc}\otimes E_{dd}\Bigr) \\
     &\hspace{10em}\cdot \Bigl( (r^{-1}s)^{\rho_f+\rho_{c'}}E_{c'f}\otimes E_{cf'}\otimes I\Bigr)\\
    &\quad=\sum_{a,c,d,f}(r^{-1}s)^{\rho_a+\rho_f+2\rho_{c'}}R_{cd}^{cd}\;E_{af}\otimes E_{a'f'}\otimes E_{dd}  \\
    &\quad=\sum_{a,d,f}(r^{-1}s)^{\rho_a+\rho_f}\Bigl(\sum_{c'} R_{d'c'}^{d'c'}(r^{-1}s)^{2\rho_{c'}}\Bigr)\;E_{af}\otimes E_{a'f'}\otimes E_{dd}  \\
    &\quad=(r^{-1}s)^{2n} \sum_{a,d,f}(r^{-1}s)^{\rho_a+\rho_f}\;E_{af}\otimes E_{a'f'}\otimes E_{dd}  \\
   &\quad=(r^{-1}s)^{2n} \sum_{a,f}(r^{-1}s)^{\rho_a+\rho_f}\;E_{af}\otimes E_{a'f'}\otimes I\\
    &\quad=(r^{-1}s)^{2n}(K\otimes I),
  \end{align*}
  where $a,c,d,f$ range from $1$ to $2n+1$.
  The first and fourth equalities are derived from Corollary \ref{cor Kij}.
  The third one follows from the fact that $R_{cd}^{cd}=R_{d'c'}^{d'c'}$.
\end{proof}

\begin{defi}\cite{BW Trans89,Klimyk, Wenzl CMP1990}
  Let $p, q\in \mathbb{C}\setminus\{0\}$, $q^2\neq 1$, and $r\in \mathbb{N}$, $r\geq 2$.
  The Birman-Wenzl-Murakami algebra $\textrm{BWM}_t(p,q)$ is the complex unital algebra with invertible generators
  $g_1, g_2, \cdots, g_{t-1}$ subject to the following relations:
  \begin{align}\label{}
     g_ig_{i+1}g_i&=g_{i+1}g_ig_{i+1},  \label{braid 1}\\
g_ig_j&=g_jg_i, \quad \mid i-j\mid \geq 2,  \label{braid 2}\\
e_ig_i&=p^{-1}e_i, \label{eigi} \\
e_ig_{i-1}^{\pm 1} e_i&=p^{\pm 1} e_i, \label{eigi-1}
  \end{align}
  where $e_i \;(1\leq i\leq t-1)$ is defined by the equation:
\begin{equation}\label{equ ei}
  (q-q^{-1})(1-e_i)=g_i-g_i^{-1}.
\end{equation}
\end{defi}

\begin{coro}\label{cor BWM}
 $e_i$ is a quasi-idempotent element and $g_i$ satisfies a cubic equation.
  To be specific,
  \begin{align}
e_i^2&=\Bigl(1+(p-p^{-1})(q-q^{-1})\Bigr)e_i, \label{equ ei ide} \\
    &(g_i-q)(g_i+q^{-1})(g_i-p^{-1})=0.\label{equ cubic}
  \end{align}
\end{coro}

\begin{proof}
  These standard results can be obtained from (\ref{eigi}) and (\ref{equ ei}).
\end{proof}

\begin{prop}\label{prop BWrep} The Birman-Wenzl-Murakami algebra
   $\textrm{BWM}_t(r^{-2n}s^{2n}, r^{-1}s)$
  admits a representation $\pi$ on the vector space $V^{\otimes t}$,
  such that $\pi(g_i)=R_{i,i+1}\;(1\leq i\leq t-1)$, where $V$ is the natural representation of $U_{r, s}(\mathfrak{so}_{2n+1})$.
\end{prop}

\begin{proof}
 We define $\pi(e_i)=K_{i,i+1}$.
 According to Corollaries \ref{cor R R-1} and \ref{cor Kquasi},
this definition is compatible with (\ref{equ ei}) and (\ref{equ ei ide}).
 It suffices to verify that $R_{i,i+1}$, $K_{j,j+1}\; (1\leq i, j\leq t-1)$ satisfy the defining relations of $\textrm{BWM}_t(r^{-2n}s^{2n}, r^{-1}s)$. Specifically:

$\bullet$  Relations (\ref{braid 1}) and (\ref{braid 2}) are clearly satisfied by the definition of the braided $R$-matrix.

$\bullet$  Relation (\ref{eigi}) follows from (\ref{equ KP}).

$\bullet$ Finally, Relation (\ref{eigi-1}) is derived using Corollary \ref{cor KRcross}.
\end{proof}

\subsection{Regulated quantum Lyndon bases by the $RLL$-formalism}\label{subsec intrinsic}

As in Hu-Wang \cite{HuWangJGP 2010}, a quantum Lyndon basis $\{\,\mathcal E_\alpha\mid \alpha\in\Phi^+\}$ has been defined as the quantum root vectors in $U^+$ for each $\alpha\in\Phi^+$ via the defining rule
\begin{equation}\label{Lyndon 1}
  \mathcal E_\gamma:=[\mathcal E_\alpha,\mathcal E_\beta]_{\langle\omega_{\beta}',\omega_{\alpha}\rangle}
=\mathcal E_\alpha\mathcal E_\beta-\langle\omega_{\beta}',\omega_{\alpha}\rangle\mathcal E_\beta\mathcal E_\alpha
\end{equation}
   for $\alpha, \gamma, \beta\in \Phi^+$ such that $\alpha<\gamma<\beta$ in the convex ordering, and $\gamma=\alpha+\beta$.
   Meanwhile, another quantum Lyndon basis is given by the defining rule
\begin{equation}\label{Lyndon 2}
  \mathcal E'_\gamma:=[\mathcal E'_\alpha,\mathcal E'_\beta]_{\langle \omega_{\alpha}',\omega_{\beta}\rangle^{-1}}
   =\mathcal E'_\alpha\mathcal E'_\beta-\langle\omega_{\alpha}',\omega_{\beta}\rangle^{-1}\mathcal E'_\beta\mathcal E'_\alpha.
\end{equation}

  This subsection is devoted to demonstrating how the two different Lyndon bases are distributed in the $L$-matrix.
  We can give an intrinsic realization (see Theorem \ref{thm Lyn}),
  which is determined only by the basic braided $R$-matrix given in Theorem \ref{main theorem} and the $RLL$-formalism (\ref{$RLL$}) \& (\ref{metric condition}).
Using (\ref{Lyndon 1}) and (\ref{Lyndon 2}), we derive the following definition,
which is a refined definition form regulated by the $RLL$ formalism out of the argumentation of Theorem \ref{thm Lyn}.

\begin{defi}\label{def Lynbas}
The Lyndon bases of $U_{r, s}(\mathfrak{so}_{2n+1})$ is given as follows:
\begin{align}\label{}
&\mathcal{E}_{i,i+1}=\mathcal E_{\alpha_{i}}=e_i, &  &1\leq i\le n,  \label{Lyn 20}\\
&\mathcal{E}_{i,j+1}=\mathcal E_{\alpha_{i,j+1}}=e_i\mathcal{E}_{i+1,j+1}-s^2\mathcal{E}_{i+1,j+1} e_i,&
	&1\leq i< j < n, \label{Lyn 21} \\
&\mathcal{E}_{i,n+1}=\mathcal E_{\alpha_{i,n+1}}=\mathcal{E}_{i, n}e_n-s^2e_n\mathcal{E}_{i, n},& & 1\le i< n,\label{Lyn 22}\\
&\mathcal{E}_{i,n'}=\mathcal E_{\beta_{i,n}}=\mathcal{E}_{i,n+1} e_n-(rs)e_n\mathcal{E}_{i,n+1},&
	&1\leq i \leq n-1, \label{Lyn 23} \\
&\mathcal{E}_{i,j'}=\mathcal E_{\beta_{i,j}}=\mathcal{E}_{i,(j+1)'} e_{j}-r^{-2}e_{j}\mathcal{E}_{i,(j+1)'},&
	&1\leq i<j\le n-1. \label{Lyn 24}
\end{align}
Another Lyndon bases (compared to Definition 2.3 \cite{HuWangJGP 2010}) can be defined in the following way:
\begin{align}\label{}
&\mathcal{E}'_{i,i+1}=\mathcal E'_{\alpha_{i}}=e_i,& &  1\leq i< n, \label{Lyn 10} \\
&\mathcal{E}'_{i,j+1}=\mathcal E'_{\alpha_{i,j+1}}=e_i\mathcal{E}'_{i+1,j+1}-r^2\mathcal{E}'_{i+1,j+1}e_i,&
 &1\leq i< j < n, \label{Lyn 11} \\
&\mathcal{E}'_{i,n+1}=\mathcal E'_{\alpha_{i,n+1}}=\mathcal{E}'_{i, n}e_n-r^2e_n\mathcal{E}'_{i, n}, & & 1\le i< n,\label{Lyn 12}\\
&\mathcal{E}'_{i,n'}=\mathcal E'_{\beta_{i,n}}=\mathcal{E}'_{i,n+1}e_n-(rs)e_n\mathcal{E}'_{i,n+1},&&
1\leq i \leq n-1, \label{Lyn 13}\\
&\mathcal{E}'_{i,j'}=\mathcal E'_{\beta_{i,j}}=\mathcal{E}'_{i,(j+1)'}e_{j}
-s^{-2}e_{j}\mathcal{E}'_{i,(j+1)'},&&
 1\leq i<j\le  n-1, \label{Lyn 14}
\end{align}
where $\alpha_{i,j+1}=\alpha_i+\cdots+\alpha_{j}=\epsilon_i-\epsilon_{j+1}$, for $1\le i\le j< n$;  $\alpha_{i,n+1}:=\alpha_i+\alpha_{i+1}+\cdots+\alpha_{n}
=\epsilon_{i}$, for $1\le i\le n$;
$\beta_{ij}=\alpha_i+\alpha_{i+1}+\cdots+2\alpha_{n}
+\alpha_{n-1}+\cdots+\alpha_{j}=\epsilon_i+\epsilon_j$, for $1\le i<j\le n$.
That is to say, the following positive roots of $\mathfrak{so}_{2n+1}$
\begin{equation*}
\begin{split}
\alpha_{12} \prec  \alpha_{13} \prec \cdots \prec \alpha_{1n}
\prec &\;\epsilon_1 \prec \beta_{1n} \prec \cdots \prec
 \beta_{13} \prec \beta_{12} \\
\alpha_{23} \prec \cdots \prec \alpha_{2n}\prec &\;\epsilon_2\prec
\beta_{2n}\prec \cdots\prec
 \beta_{23}  \\
&\cdots \\
\alpha_{n{-}1,n}\prec  &\;\epsilon_{n{-}1}\prec\beta_{n{-}1,n} \\
&\;\epsilon_n
\end{split}
\end{equation*}
are arranged as the convex ordering (cf. \cite{Beck1, HuWangJGP 2010,Kha Pac02} and references therein).
\end{defi}

\begin{defi}\label{def auto}
  Define a $\mathbb{Q}$-algebra automorphism $\zeta$ of $U_{r,s}(so_{2n+1})$ by $\zeta(r)=s$, $\zeta(s)=r$, and
  $$\zeta(e_i)=e_i, \quad \zeta(f_i)=f_i, \quad \zeta(\omega_i)=\omega_i', \quad \zeta(\omega_i')=\omega_i,
  \qquad 1\leq i\leq n.$$
\end{defi}

\noindent Obviously $\zeta^2=1$.
Then the two different Lyndon bases can be connected with $\zeta$:
$$\zeta(\mathcal{E}_\gamma)=\mathcal{E}'_\gamma, \quad \gamma\in \Phi.$$

We can also rewrite $\mathcal{E}'_{\gamma}$ in the form of $\mathcal{E}_\beta$,
where $\alpha, \beta\in \Phi$, see Lemma \ref{lemm sub eij e'ij} of the Appendix.

\begin{defi}\label{def FRT}\cite{FRT Russ 1989}
   $U(\widehat{R})$ is an associative algebra with unit, where $\widehat{R}=P\circ R$, $P$ is the usual flip map.
   It has generators $\ell_{ij}^+, \ell_{ji}^-, 1\leq i\leq j\leq 2n+1$.
   Let $L^{\pm}=(\ell^{\pm}_{ij}), 1\leq i, j\leq 2n+1$, with $\ell_{ij}^+=\ell_{ji}^-=0$,
   and $\ell_{ii}^-\ell_{ii}^+=\ell_{ii}^+\ell_{ii}^-$ for $1\leq j<i\leq 2n+1$.
   The defining relations are given in matrix form as follows:
   \begin{equation}\label{$RLL$}
     \widehat{R}L_1^{\pm}L_2^{\pm}=L_2^{\pm}L_1^{\pm}\widehat{R}, \quad \widehat{R}L_1^{+}L_2^{-}=L_2^{-}L_1^{+}\widehat{R},
   \end{equation}
   \begin{equation}\label{metric condition}
     L^\pm C(L^\pm )^t C^{-1}=I,
   \end{equation}
   where $L_1^{\pm}=L^{\pm}\otimes I, L_2^{\pm}=I\otimes L^{\pm}$, $t$ denotes the matrix transposition,
    and $C$ is the quantum metric matrix, see Remark \ref{metric matrix}.
\end{defi}

 The following theorem is the main result of this subsection.
 Notice that the anti-diagonal of $L^+$ divides the matrix into two parts: the upper part $L^{(\text{p})}$ and the lower part $L^{(\text{d})}$.

\begin{theorem} \label{thm Lyn}
In the upper triangular matrix $L^+:$
  $$L^+=
  \left(
    \begin{array}{cccccccc}
      \ell_{11}^+ &\color{blue} \ell_{12}^+ & \cdots & \color{blue}\ell_{1, n+1}^+ &\color{blue}\ell_{1, n'}^+ &\cdots & \color{blue}\ell_{12'}^+  &\color{JungleGreen} \ell_{11'}^+\\
       & \ell_{22}^+ & \cdots & \color{blue}\ell_{2,n+1}^+ &\color{blue}\ell_{2, n'}^+ & \cdots & \color{JungleGreen}\ell_{22'}^+ & \color{red}\ell_{21'}^+ \\
       &\qquad \ddots & \vdots & \vdots & \vdots & \ddots & \vdots & \vdots\\
       &&\ell_{n,n}^+ &\color{blue}\ell_{n,n+1}^+ &\color{JungleGreen}\ell_{n,n'}^+& \cdots & \color{red}\ell_{n,2'}^+& \color{red}\ell_{n,1'}^+ \\
       &  &  &\ell_{n+1,n+1}^+ &\color{red} \ell_{n+1, n'}^+ & \cdots & \color{red}\ell_{n+1, 2'}^+ & \color{red}\ell_{n+1,1'}^+ \\
       &&&&&&&\\
       &  &  &  & \ell_{n'n'}^+ & \cdots & \color{red}\ell_{n'2'}^+ & \color{red}\ell_{n'1'}^+\\
       &&&&&\ddots&\vdots&\vdots\\
       &  &  &  & &  & \ell_{2' 2'}^+ & \color{red}\ell_{2' 1'}^+ \\
       &  &  &  & & &  & \ell_{1' 1'}^+
    \end{array}
    \right),$$
two quantum Lyndon bases (labeled in red and blue respectively) of $U_{r, s}^+(\mathfrak{so}_{2n+1})$ are symmetrically distributed with respect to the anti-diagonal (labeled in green, the explicit definitions of $\ell_{ii'}^+$ and $\theta_{ii'}^+$ can be seen in (\ref{equ thetaii'})). From the $RLL$ relations, the following correspondences hold (up to scalars):
 $$\left(
\begin{array}{cccccccc}
    (\omega_{\epsilon_1}')^{-1} & \color{blue}\mathcal{E}_{12}(\omega_{\epsilon_1}')^{-1}  &  \cdots &  \color{blue}\mathcal{E}_{1,n+1}(\omega_{\epsilon_1}')^{-1} &\color{blue}\mathcal{E}_{1n'}(\omega_{\epsilon_1}')^{-1} & \cdots & \color{blue}\mathcal{E}_{12'}(\omega_{\epsilon_1}')^{-1} & \color{JungleGreen}\theta_{11'}^+(\omega_{\epsilon_1}')^{-1} \\
    &&&&&&&\\
     & (\omega_{\epsilon_2}')^{-1} & \cdots & \color{blue}\mathcal{E}_{2,n+1}(\omega_{\epsilon_2}')^{-1} &\color{blue}\mathcal{E}_{2n'}(\omega_{\epsilon_2}')^{-1} &\cdots & \color{JungleGreen}\theta_{22'}^+(\omega_{\epsilon_2}')^{-1}& \color{red}\mathcal{E}'_{12'}(\omega_{\epsilon_2}')^{-1} \\
     &\qquad \ddots & \vdots  &  \vdots  & \vdots  & \ddots & \vdots & \vdots\\
     &  &   (\omega_{\epsilon_n}')^{-1} & \color{blue}\mathcal{E}_{n,n+1}(\omega_{\epsilon_n}')^{-1} & \color{JungleGreen}\theta_{nn'}^+(\omega_{\epsilon_n}')^{-1} & \cdots &\color{red}\mathcal{E}'_{2n'}(\omega_{\epsilon_n}')^{-1}&    \color{red}\mathcal{E}'_{1n'}(\omega_{\epsilon_n}')^{-1} \\
     &&&&&&&\\
     &  &      & 1 &\color{red} \mathcal{E}'_{n,n+1}& \cdots  & \color{red}\mathcal{E}'_{2,n+1} & \color{red}\mathcal{E}'_{1,n+1} \\
     &&&&&&&\\
     &  &      &  &\omega_{\epsilon_n}'  & \cdots & \color{red}\mathcal{E}'_{2n}\omega_{\epsilon_n}'  & \color{red}\mathcal{E}'_{1n}\omega_{\epsilon_n}' \\
     &  &      &  &  & \ddots & \vdots & \vdots  \\
     &  &      &  &  &  & \omega_{\epsilon_2}' & \color{red}\mathcal{E}'_{12}\omega_{\epsilon_2}' \\
     &&&&&&&\\
     &  &      &  &  &  &  & \omega_{\epsilon_1}'\\
  \end{array}
\right)$$
 Similarly, for the lower triangular matrix $L^-$, analogous correspondences hold in $U_{r,s}^-(\mathfrak{so}_{2n+1})$.
\end{theorem}
\begin{proof}
 (I)
To begin with, we consider the upper part $L^{+(\text{p})}$, and first fix some $k$ in the range $1 \leq k \leq n-2$ to analyze its constituents $\ell_{kj}^+$ where $k < j \leq (k+1)'$.

From the $RLL$ relation, we have
$$\widehat{R}L_1^+L_2^+(v_{k+1}\otimes v_{k+2})=L_2^+L_1^+\widehat{R}(v_{k+1}\otimes v_{k+2}).$$
Observe the coefficients of $v_k\otimes v_{k+1}$ on both sides.
The coefficient of the left-hand side is equal to
\begin{align*}
   & \widehat{R}\Bigl(\ell_{k, k+1}^+\ell_{k+1, k+2}^+v_k\otimes v_{k+1}
   +\ell_{k+1, k+1}^+\ell_{k, k+2}^+v_{k+1}\otimes v_k\Bigr)\\
  &\quad = \Bigl[(rs)\ell_{k, k+1}^+\ell_{k+1, k+2}^+
   +(r^{-1}s{-}rs^{-1})\ell_{k+1, k+1}^+\ell_{k, k+2}^+\Bigr]v_k\otimes v_{k+1},
\end{align*}
while that of the right-hand side is
$rs\,\ell_{k+1, k+2}^+\ell_{k, k+1}^+v_k\otimes v_{k+1}. $
Therefore we have
\begin{equation*}\label{}
  \ell_{k, k+2}^+=\frac{rs}{rs^{-1}-r^{-1}s}(\ell_{k+1, k+1}^+)^{-1}\Bigl[\ell_{k, k+1}^+
\ell_{k+1, k+2}^+-\ell_{k+1, k+2}^+\ell_{k, k+1}^+\Bigr].
\end{equation*}

Assume that $\ell_{k+1, k+1}^+= (\omega_{\epsilon_{k+1}}')^{-1}$,
  $\ell^+_{k, k+1}=B_{k, k+1}^+e_k(\omega'_{\epsilon_k})^{-1}$
and $\ell^+_{k+1, k+2}=B_{k+1, k+2}^+e_{k+1}(\omega'_{\epsilon_{k+1}})^{-1}$,
where $B_{k, k+1}^+$ and $B_{k+1, k+2}^+$ are some scalars (note that $B_{k+1, k+1}^+=1$).
Then
\begin{align*}
  \ell_{k, k+2}^+ & =\frac{rs B_{k, k+1}^+ B_{k+1, k+2}^+}{rs^{-1}-r^{-1}s}\omega_{\epsilon_{k+1}}'\Bigl[e_k(\omega'_{\epsilon_k})^{-1}e_{k+1}(\omega'_{\epsilon_{k+1}})^{-1}
  -e_{k+1}(\omega'_{\epsilon_{k+1}})^{-1}e_k(\omega'_{\epsilon_k})^{-1}\Bigr] \\
   &=\frac{B_{k, k+1}^+ B_{k+1, k+2}^+}{r^2-s^2}\Bigl[e_ke_{k+1}-s^2 e_{k+1}e_k\Bigr](\omega'_{\epsilon_k})^{-1}\\
  &:=B_{k, k+2}^+\mathcal{E}_{k, k+2}(\omega'_{\epsilon_k})^{-1}.
\end{align*}

Using induction, we have
\begin{equation}\label{equ ell kj}
  \ell_{k, j}^+=B_{k, j}^+\mathcal{E}_{k, j}(\omega'_{\epsilon_k})^{-1}=\frac{B_{k, j-1}^+ B_{j-1, j}^+}{r^2-s^2}\mathcal{E}_{k, j}(\omega'_{\epsilon_k})^{-1}, \quad k+1\leq j\leq n.
\end{equation}

Using induction and the recursive relation
\begin{equation*}\label{}
  \ell_{k, n+1}^+=\frac{1}{rs^{-1}-r^{-1}s}(\ell_{n, n}^+)^{-1}\Bigl[(rs)\ell_{k, n}^+ \ell_{n, n+1}^+
-\ell_{n, n+1}^+\ell_{k, n}^+\Bigr],
\end{equation*}
we conclude that
\begin{align}\label{ell k, n+1}
  \ell_{k, n+1}^+ & =\frac{B_{k, n}^+ B_{n, n+1}^+}{rs^{-1}-r^{-1}s}\;\omega'_n\Bigl[(rs)\mathcal{E}_{k, n}(\omega'_{\epsilon_k})^{-1}e_n(\omega'_n)^{-1}-e_n(\omega'_n)^{-1}\mathcal{E}_{k, n}(\omega'_{\epsilon_k})^{-1}\Bigr] \nonumber\\
  &=\frac{B_{k, n}^+ B_{n, n+1}^+}{r^2-s^2}\Bigl[\mathcal{E}_{k, n}e_n-s^2e_n\mathcal{E}_{k, n}\Bigr](\omega'_{\epsilon_k})^{-1}\\
  &:=B_{k, n+1}^+\mathcal{E}_{k, n+1}(\omega'_{\epsilon_k})^{-1}.\nonumber
\end{align}

Therefore, we derive that the half of $k$-row elements of the upper part under the anti-diagonal of $L^+$ are as follows:
\begin{equation}\label{norbasis1}
  \ell_{k, j}^+=B_{k, j}^+\,\mathcal{E}_{k, j}(\omega'_{\epsilon_k})^{-1}, \quad k+1\leq j\leq n+1.
\end{equation}

To achieve the remaining elements $\ell_{k, j'}^+$ for $k<j\leq n$, we calculate step by step. From
\begin{equation*}\label{}
  \ell_{k, n'}^+=\frac{1}{rs^{-1}-r^{-1}s}(\ell_{n+1, n+1}^+)^{-1}\Bigl[\ell_{k, n+1}^+\ell_{n+1, n'}^+-\ell_{n+1, n'}^+\ell_{k, n+1}^+\Bigr],
\end{equation*}
we find that (Here we have used the fact that $\ell_{n+1, n+1}^+=1$.)
\begin{align}\label{equ ell k n'}
  \ell_{k, n'}^+ &=\frac{B_{k, n+1}^+ B_{n+1, n'}^+}{rs^{-1}-r^{-1}s}\Bigl[\mathcal{E}_{k, n+1}(\omega'_{\epsilon_k})^{-1}e_n-e_n\mathcal{E}_{k, n+1}(\omega'_{\epsilon_k})^{-1}\Bigr]  \nonumber\\
  &=\frac{B_{k, n+1}^+ B_{n+1, n'}^+}{r^2-s^2}\Bigl[\mathcal{E}_{k, n+1}e_n-(rs)e_n\mathcal{E}_{k, n+1}\Bigr](\omega'_{\epsilon_k})^{-1}\\
  &:= B_{k, n'}^+ \mathcal{E}_{k, n'}(\omega'_{\epsilon_k})^{-1}. \nonumber
\end{align}

By the same token,
\begin{align*}
  \ell_{k, (n-1)'}^+ &=\frac{1}{r^2-s^2}(\ell_{n', n'}^+)^{-1}\Bigl[\ell_{k, n'}^+\ell_{n', (n-1)'}^+-\ell_{n', (n-1)'}^+\ell_{k, n'}^+\Bigr]\\
   &=\frac{B_{k, n'}^+ B_{n', (n-1)'}^+}{r^2-s^2}(\omega'_n)^{-1}\Bigl[\mathcal{E}_{k, n'}(\omega'_{\epsilon_k})^{-1}e_{n-1}\omega'_n
   -e_{n-1}\omega'_n\mathcal{E}_{k, n'}(\omega'_{\epsilon_k})^{-1}\Bigr] \\
   &=\frac{r^2s^2 B_{k, n'}^+ B_{n', (n-1)'}^+}{r^2-s^2}\Bigl[\mathcal{E}_{k, n'}e_{n-1}-r^{-2}e_{n-1}\mathcal{E}_{k, n'}\Bigr](\omega'_{\epsilon_k})^{-1} \\
  &:= B_{k, (n-1)'}^+\mathcal{E}_{k, (n-1)'}(\omega'_{\epsilon_k})^{-1}.
\end{align*}

Using induction and the recursive relation:
\begin{equation*}\label{}
  \ell_{k, j'}^+=\frac{1}{r^2-s^2}(\ell_{(j+1)',(j+1)'}^+)^{-1}\Bigl[\ell_{k, (j+1)'}^+ \ell_{(j+1)', j'}^+
-\ell_{(j+1)', j'}^+\ell_{k, (j+1)'}^+\Bigr],
\end{equation*}
we conclude that
\begin{align}\label{equ ell kj'}
   \ell_{k, j'}^+&=\frac{B_{k, (j+1)'}^+ B_{(j+1)', j'}^+}{r^2-s^2}(\omega'_{\epsilon_{j+1}})^{-1}\Bigl[\mathcal{E}_{k, (j+1)'}(\omega_{\epsilon_k}')^{-1}e_j\omega_{\epsilon_{j+1}}' - e_j\omega_{\epsilon_{j+1}}'\mathcal{E}_{k, (j+1)'}(\omega_{\epsilon_k}')^{-1}\Bigr] \nonumber \\
   & =\frac{r^2s^2B_{k, (j+1)'}^+ B_{(j+1)', j'}^+}{r^2-s^2}\Bigl[r^{-2}e_j\mathcal{E}_{k, (j+1)'}-\mathcal{E}_{k, (j+1)'}e_j\Bigr](\omega_{\epsilon_k}')^{-1}\\
   &:=B_{k, j'}^+\mathcal{E}_{k, j'}(\omega_{\epsilon_k}')^{-1}. \nonumber
\end{align}

 Therefore,
\begin{equation}\label{norbasis2}
  \ell_{k, s'}^+=B_{k, s'}^+\mathcal{E}_{k, s'}(\omega_{\epsilon_k}')^{-1}, \qquad k<s\leq n-1.
\end{equation}

Now we consider the case $k=n-1$.
Assume that $\ell_{n-1, n}^+= B_{n-1, n}^+e_{n-1}^+(\omega'_{\epsilon_{n-1}})^{-1}$
and $\ell_{n, n+1}^+=B_{n,n+1}^+e_n (\omega'_{\epsilon_{n}})^{-1}$, where $B_{n-1,n}^+$ and $B_{n,n+1}^+$ are some scalars.
From the $RLL$ relation,
we conclude that
\begin{align}\label{ell n-1, n+1}
  \ell_{n-1, n+1}^+&=\;\frac{1}{rs^{-1}-r^{-1}s}(\ell_{n, n}^+)^{-1}\Bigl[rs\ell_{n-1, n}^+
\ell_{n, n+1}^+-\ell_{n, n+1}^+\ell_{n-1, n}^+\Bigr] \nonumber\\
&\;=\frac{B_{n-1, n}^+B_{n,n+1}^+}{rs^{-1}-r^{-1}s} \omega'_n\Bigl[rs e_{n-1}(\omega'_{\epsilon_{n-1}})^{-1}e_n(\omega'_{\epsilon_{n}})^{-1}
-e_n(\omega'_{\epsilon_{n}})^{-1}e_{n-1}(\omega'_{\epsilon_{n-1}})^{-1}\Bigr]\\
&\;=\frac{B_{n-1, n}^+B_{n,n+1}^+}{r^2-s^2}\Bigl[e_{n-1}e_n-s^2 e_n e_{n-1}\Bigr](\omega'_{\epsilon_{n-1}})^{-1} \nonumber \\
&\;:=B_{n-1, n+1}^+ \mathcal{E}_{n-1, n+1}(\omega'_{\epsilon_{n-1}})^{-1}. \nonumber
\end{align}

Next, assume that $\ell_{n+1, n'}^+=B_{n+1, n'}^+e_n$.
From the $RLL$ relation, we obtain that
\begin{align*}\label{}
  \ell_{n-1, n'}^+&=\;\frac{1}{rs^{-1}-r^{-1}s}(\ell_{n+1, n+1}^+)^{-1}\Bigl[\ell_{n-1, n+1}^+
\ell_{n+1, n'}^+-\ell_{n+1, n'}^+\ell_{n-1, n+1}^+\Bigr]\\
&\;=\frac{B_{n-1, n}^+B_{n,n+1}^+}{rs^{-1}-r^{-1}s} \Bigl[\mathcal{E}_{n-1, n+1}(\omega'_{\epsilon_{n-1}})^{-1}e_n
-e_n\mathcal{E}_{n-1, n+1}(\omega'_{\epsilon_{n-1}})^{-1}\Bigr]\\
&\;=\frac{B_{n-1, n+1}^+B_{n+1, n'}^+}{r^2-s^2}\Bigl[\mathcal{E}_{n-1, n+1} e_n-rs e_n\mathcal{E}_{n-1, n+1}\Bigr](\omega'_{\epsilon_{n-1}})^{-1} \\
&\;:=B_{n-1, n'}^+ \mathcal{E}_{n-1, n'}(\omega'_{\epsilon_{n-1}})^{-1}.
\end{align*}

For $k=n$, there is only one simple root vector $\ell_{n, n+1}^+=B_{n, n+1}^+e_n$.
So we have finished the upper part $L^{+(\text{p})}$.

(II) Now we consider the constituents $\ell_{jk'}^+$'s for $k<j\le (k+1)'$ of the lower part $L^{+(\text{d})}$.
We also firstly fix some $k$ such that $1\leq k \leq n-2$.

From the $RLL$ relation, we have
$$\ell_{(k+2)', k'}^+=\frac{1}{s^2-r^2}(\ell_{(k+1)', (k+1)'}^+)^{-1}\Bigl[\ell_{(k+1)', k'}^+\ell_{(k+2)',(k+1)'}^+-\ell_{(k+2)',(k+1)'}^+\ell_{(k+1)', k'}^+
\Bigr].$$
Assume that $\ell^+_{(k+1)',k'}=B^+_{(k+1)',k'}e_k\omega_{\epsilon_{k+1}}'$
and $\ell_{(k+2)',(k+1)'}^+=B^+_{(k+2)',(k+1)'}e_{k+1}\omega_{\epsilon_{k+2}}'$,
then we have
\begin{align*}
  \ell_{(k+2)', k'}^+ & =\frac{B^+_{(k+2)',(k+1)'}B^+_{(k+1)',k'}}{s^2-r^2}(\omega_{\epsilon_{k+1}}')^{-1}\Bigl[e_k\omega_{\epsilon_{k+1}}' e_{k+1}\omega_{\epsilon_{k+2}}'
  - e_{k+1}\omega_{\epsilon_{k+2}}' e_k\omega_{\epsilon_{k+1}}'\Bigr] \\
   &=\frac{s^2 B^+_{(k+2)', (k+1)'} B^+_{(k+1)',k'}}{s^2-r^2}
   \Bigl[e_ke_{k+1}-r^2e_{k+1}e_k\Bigr]\omega_{\epsilon_{k+2}}'\\
  &:=B^+_{(k+2)', k'} \mathcal{E}'_{k, k+2}\omega_{\epsilon_{k+2}}'.
\end{align*}

Using induction, we have
\begin{equation}\label{equ ell j'k'}
  \ell_{j', k'}^+=\;\frac{s^2B_{j', (k+1)'}^+ B_{(k+1)', k'}^+}{s^2-r^2}\mathcal{E}'_{k, j}\omega'_{\epsilon_j}, \quad k+1\leq j\leq n.
\end{equation}

Using induction and the recursive relation:
$$\ell_{n+1, k'}^+=\frac{1}{r^{-1}s-rs^{-1}}(\ell_{n', n'}^+)^{-1}\Bigl[r^{-1}s^{-1}\ell_{n', k'}^+ \ell_{n+1, n'}^+-\ell_{n+1, n'}^+\ell_{n', k'}^+\Bigr],$$
we conclude that
\begin{align}\label{equ ell n+1, k'}
  \ell_{n+1, k'}^+ & =\frac{B_{n+1, n'} B_{n', k'}}{r^{-1}s-rs^{-1}}(\omega_n')^{-1}\Bigl[r^{-1}s^{-1}\mathcal{E}'_{k, n}\omega_n' e_n- e_n\mathcal{E}'_{k, n}\omega_n'\Bigr] \nonumber\\
  &=\frac{s^2B_{n+1, n'} B_{n', k'}}{s^2-r^2}(\mathcal{E}'_{k, n}e_n-r^2e_n\mathcal{E}'_{k, n})\\
  &:=B^+_{n+1, k'} \mathcal{E}'_{k, n+1}. \nonumber
\end{align}

Therefore, we derive that the half of $k'$-column elements of the lower part under the anti-diagonal of triangular matrix $L$ are as follows
\begin{equation}\label{}
  \ell_{(k+j)', k'}^+=\left\{
                                    \begin{array}{ll}
                                      B_{(k+j)', k'}^+ \mathcal{E}'_{k, k+j}\omega_{\epsilon_{k+j}}', & \hbox{$1\leq j\leq n-k$;} \\
                                      B_{n+1, k'}^+ \mathcal{E}'_{k, n+1}, & \hbox{$j=n-k+1$.}
                                    \end{array}
                                  \right.
\end{equation}

To achieve the remaining elements $\ell_{jk'}^+$ for $k<j\le n$, we do it case by case.
From
$$\ell_{n, k'}^+=\frac{1}{r^{-1}s-rs^{-1}}(\ell^+_{n+1, n+1})^{-1}\Bigl[\ell_{n+1, k'}^+\ell_{n, n+1}^+-\ell_{n, n+1}^+ \ell_{n+1, k'}^+\Bigr],$$
we find that (Here we have used the fact that $\ell_{n+1, n+1}^+=1$.)
\begin{align}\label{equ ell n k'}
  \ell_{n, k'}^+ &=\frac{B_{n, n+1}^+B_{n+1, k'}^+}{r^{-1}s-rs^{-1}}\Bigl[\mathcal{E}'_{k, n+1}e_n(\omega_n')^{-1}-e_n(\omega_n')^{-1}\mathcal{E}'_{k, n+1}\Bigr] \nonumber \\
  &=\frac{B_{n, n+1}^+B_{n+1, k'}^+}{r^{-1}s-rs^{-1}}\Bigl[\mathcal{E}'_{k, n+1}e_n-rse_n\mathcal{E}'_{k, n+1}\Bigr](\omega_n')^{-1}\\
  &:=B_{n, k'}^+ \mathcal{E}'_{k, n'}(\omega_n')^{-1}. \nonumber
\end{align}

Similarly, we have
\begin{align}\label{equ ell n-1, k'}
  \ell_{n-1, k'}^+ &=\frac{rs}{r^{-1}s-rs^{-1}}(\ell_{n, n}^+)^{-1}\Bigl[\ell_{n, k'}^+\ell_{n-1, n}^+-\ell_{n-1, n}^+ \ell_{n, k'}^+\Bigr] \nonumber\\
   &=\frac{rs B_{n-1, n}^+ B_{n, k'}^+}{r^{-1}s-rs^{-1}}\omega_n'\Bigl[\mathcal{E}'_{k, n'}(\omega_n')^{-1}e_{n-1}(\omega_{\epsilon_{n-1}}')^{-1}-e_{n-1}(\omega_{\epsilon_{n-1}}')^{-1}\mathcal{E}'_{k, n'}(\omega_n')^{-1}\Bigr] \\
   &=\frac{r^{-1}s B_{n-1, n}^+B_{n, k'}^+}{r^{-1}s-rs^{-1}}\Bigl[\mathcal{E}'_{k, n'}e_{n-1}-s^{-2}e_{n-1}\mathcal{E}'_{k, n'}\Bigr](\omega_{\epsilon_{n-1}}')^{-1} \nonumber\\
  &:=B_{n-1, k'}^+\mathcal{E}'_{k, (n-1)'}(\omega_{\epsilon_{n-1}}')^{-1}. \nonumber
\end{align}

Using induction on $j \;(j>k)$ and the recursive relation:
$$\ell_{j-1, k'}^+=\frac{rs}{r^{-1}s-rs^{-1}}(\ell_{j,j}^+)^{-1}\Bigl[\ell_{j, k'}^+ \ell_{j-1, j}^+-\ell_{j-1, j}^+\ell_{j, k'}^+\Bigr],$$
we conclude that
\begin{align}\label{equ ell j-1, k'}
   \ell_{j-1, k'}^+&=\frac{rs B_{j-1, j}^+B_{j, k'}^+}{r^{-1}s-rs^{-1}}\omega_{\epsilon_j}'\Bigl[\mathcal{E}'_{k, j'}(\omega_{\epsilon_j}')^{-1}e_{j-1}(\omega_{\epsilon_{j-1}}')^{-1}-e_{j-1}(\omega_{\epsilon_{j-1}}')^{-1}\mathcal{E}'_{k, j'}(\omega_{\epsilon_j}')^{-1}\Bigr]  \nonumber\\
   & =\frac{r^{-1}s B_{j-1, j}^+B_{j, k'}^+}{r^{-1}s-rs^{-1}}\Bigl[\mathcal{E}'_{k, j'}e_{j-1}-s^{-2}e_{j-1}\mathcal{E}'_{k, j'}\Bigr](\omega_{\epsilon_{j-1}}')^{-1}\\
   &:=B_{j-1, k'}^+ \mathcal{E}'_{k, (j-1)'}(\omega_{\epsilon_{j-1}}')^{-1}. \nonumber
\end{align}

Therefore,
\begin{equation}\label{}
  \ell_{jk'}^+=B_{j, k'}^+\mathcal{E}'_{k, j'}(\omega_{\epsilon_j}')^{-1}, \qquad k<j\leq n-1.
\end{equation}

Now we consider the case $k=n-1$.
Assume that $\ell_{n', (n-1)'}^+= B_{n', (n-1)'}^+e_{n-1}\omega'_n$ and $\ell_{n+1, n'}^+=B_{n+1, n'}^+e_n$.
From the $RLL$ relation, we have
$$\ell_{n+1, (n-1)'}^+=\frac{1}{r^{-1}s-rs^{-1}}(\ell_{n', n'})^{-1}\Bigl[r^{-1}s^{-1}\ell_{n', (n-1)'}^+ \ell_{n+1, n'}^+-\ell_{n+1, n'}^+\ell_{n', (n-1)'}^+\Bigr].$$
Therefore we conclude that
\begin{align} \label{equ ell n+1 n-1'}
   \ell_{n+1, (n-1)'}^+&=\frac{B_{n+1, n'}^+B_{n', (n-1)'}^+}{r^{-1}s-rs^{-1}}(\omega'_n)^{-1}\Bigl[r^{-1}s^{-1}e_{n-1}(\omega'_n)^{-1}e_n-e_ne_{n-1}(\omega'_n)^{-1}\Bigr] \nonumber \\
   & =\frac{s^2 B_{j-1, j}^+B_{j, k'}^+}{s^2-r^2}\Bigl[e_{n-1}e_n-r^2e_ne_{n-1}\Bigr]\\
   &:=B_{n+1, (n-1)'}^+ \mathcal{E}'_{n-1, n+1}. \nonumber
\end{align}

Next, assume that $\ell_{n, n+1}^+=B_{n, n+1}^+e_n(\omega'_n)^{-1}$.
From the $RLL$ relation, we obtain that
\begin{align*}\label{}
  \ell_{n, (n-1)'}^+&=\;\frac{1}{r^{-1}s-rs^{-1}}(\ell_{n+1, n+1}^+)^{-1}\Bigl[\ell_{n+1, (n-1)'}^+
\ell_{n, n+1}^+-\ell_{n, n+1}^+\ell_{n+1, (n-1)'}^+\Bigr]\\
&\;=\frac{B_{n, n+1}^+B_{n+1, (n-1)'}^+}{r^{-1}s-rs^{-1}} \Bigl[\mathcal{E}'_{n-1, n+1}e_n(\omega'_n)^{-1}
-e_n(\omega'_n)^{-1}\mathcal{E}'_{n-1, n+1}\Bigr]\\
&\;=\frac{B_{n, n+1}^+B_{n+1, (n-1)'}^+}{r^{-1}s-rs^{-1}}\Bigl[\mathcal{E}'_{n-1, n+1} e_n-rs e_n\mathcal{E}'_{n-1, n+1}\Bigr](\omega'_n)^{-1} \\
&\;:=B_{n, (n-1)'}^+ \mathcal{E}'_{n-1, n'}(\omega'_n)^{-1}.
\end{align*}

For $k=n$, there is only one simple root vector $\ell_{n+1, n'}^+=B_{n+1, n'}^+e_n$.
So we have finished the upper part $L^{+(\text{d})}$.

(III)
For those elements $\ell_{ii'}^+$ (where $1\leq i< n$), we derive their formulas from
\begin{align}\label{equ thetaii'}
  \ell_{ii'}^+ &= \frac{1}{r^2 - s^2} (\ell_{(i+1)',(i+1)'}^+)^{-1}\Bigl[\ell_{i, (i+1)'}^+ \ell_{(i+1)', i'}^+ - \ell_{(i+1)', i'}^+\ell_{i, (i+1)'}^+\Bigr],  \nonumber \\
   &= \frac{B_{i, (i+1)'}^+ B_{(i+1)', i'}^+}{r^2 - s^2}(\omega'_{\epsilon_{i+1}})^{-1}\Bigl[\mathcal{E}_{i, (i+1)'}(\omega'_{\epsilon_i})^{-1}e_i\omega'_{\epsilon_{i+1}} - e_i\omega'_{\epsilon_{i+1}}\mathcal{E}_{i, (i+1)'}(\omega'_{\epsilon_i})^{-1}\Bigr] \\
   &= \frac{r^4s^2 B_{i, (i+1)'}^+ B_{(i+1)', i'}^+ }{r^2 - s^2} \Bigl(\mathcal{E}_{i, (i+1)'}e_i - r^{-4} e_i\mathcal{E}_{i, (i+1)'}\Bigr)(\omega'_{\epsilon_i})^{-1}, \nonumber\\
  &:= B_{i, i'}^+\theta_{ii'}^+(\omega'_{\epsilon_i})^{-1}. \nonumber
\end{align}
For $\ell_{n, n'}^+$, we have
\begin{align}\label{equ theta nn'}
  \ell_{nn'}^+ &= \frac{1}{rs^{-1} - r^{-1}s} (\ell_{n+1, n+1}^+)^{-1}\Bigl[\ell_{n, n+1}^+ \ell_{n+1, n'}^+ - \ell_{n+1, n'}^+\ell_{n, n+1}^+\Bigr],  \nonumber \\
   &= \frac{B_{n, n+1}^+ B_{n+1, n'}^+}{rs^{-1} - r^{-1}s}\Bigl[e_n(\omega'_n)^{-1}e_n - e_ne_n(\omega'_n)^{-1}\Bigr] \\
   &= \frac{rs(rs^{-1} - 1)}{r^2 - s^2} B_{n, n+1}^+ B_{n+1, n'}^+ e_n^2(\omega'_n)^{-1} \nonumber \\
  &:= B_{n, n'}^+\theta_{nn'}^+(\omega'_n)^{-1}. \nonumber
\end{align}
Clearly, $\theta_{ii'}^+$ for $1\leq i\leq n$ (with $\theta_{ii'}^+$ for $1\leq i<n$ from \eqref{equ thetaii'} and $\theta_{nn'}^+$ from  \eqref{equ theta nn'}) are not quantum root vectors of $U_{r, s}^+(\mathfrak{so}_{2n+1})$, since $2(\alpha_i+\alpha_{i+1}+\cdots+\alpha_n)$ is not a root of $\mathfrak{so}_{2n+1}$.

(IV) For the lower triangular matrix $L^-$, we can also proceed similarly to determine its corresponding quantum root vectors in $U_{r,s}^-(\mathfrak{so}_{2n+1})$, along with the coefficients $B_{j, i}^-$ for $1\leq i\leq j\leq n$.
Notably, similar to the anti-diagonal elements $\ell_{ii'}^+ \; (1\leq i\leq n)$ of the upper triangular part $L^+$, the anti-diagonal elements $\ell_{i'i}^- \; (1\leq i\leq n)$ of $L^-$ do not correspond to quantum root vectors.

 (V) From the preceding calculations, if we specify explicit values for $B_{i, i+1}^+$ and $B_{i+1, i}^-$ (for $1 \leq i \leq 2'$), we can recursively determine all the values of $B_{i, j}^+$ and $B_{j, i}^-$ for $1 \leq i < j \leq 1'$.

To further constrain these coefficients, we use the $RLL$ relation for $1 \leq i \leq 2':$
$$\widehat{R}L_1^+L_2^-(v_{i+1}\otimes v_i) = L_2^-L_1^+\widehat{R}(v_{i+1}\otimes v_i).$$

By comparing the relevant coefficients of $v_i\otimes v_{i+1}$ on both sides,
we derive the following relations, splitting into two cases based on the range of $i$:
\begin{align*}
  rs\ell_{i, i+1}^+ \ell_{i+1, i}^- - r^{-1}s^{-1}\ell_{i+1, i}^-\ell_{i, i+1}^+ & = (r^{-1}s - rs^{-1}) \left(\ell_{i+1, i+1}^- \ell_{ii}^+ - \ell_{i+1, i+1}^+\ell_{i, i}^-\right), \quad 1 \leq i < n; \\
  \ell_{n, n+1}^+ \ell_{n+1, n}^- - \ell_{n+1, n}^-\ell_{n, n+1}^+ & = (r^{-1}s - rs^{-1}) \left(\ell_{n+1, n+1}^- \ell_{nn}^+ - \ell_{n+1, n+1}^+\ell_{n, n}^-\right), \quad n\leq i\leq n+1; \\
  r^{-1}s^{-1}\ell_{i, i+1}^+ \ell_{i+1, i}^- - rs\ell_{i+1, i}^-\ell_{i, i+1}^+ & = (r^{-1}s - rs^{-1}) \left(\ell_{i+1, i+1}^- \ell_{ii}^+ - \ell_{i+1, i+1}^+\ell_{i, i}^-\right), \quad n' \leq i < 2'.
\end{align*}
For the first case $1\leq i<n$, we substitute $\ell_{i, i+1}^+=B_{i, i+1}^+ e_i (\omega'_{\epsilon_i})^{-1}$ and $\ell_{i+1, i}^-=B_{i+1, i}^-(\omega_{\epsilon_i})^{-1}f_i$, $\ell_{i i}^+=(\omega'_{\epsilon_i})^{-1}$ and $\ell_{ii}^-=(\omega_{\epsilon_i})^{-1}$ into the first equation, we conclude that
$B_{i, i+1}^+ B_{i+1, i}^-=-(r^2-s^2)^2.$
Without loss of generality, we may assume that
\begin{equation}\label{equ Bi i+1}
  B_{i, i+1}^+=-B_{i+1, i}^-=r^2-s^2.
\end{equation}
Similarly, we can assume that
\begin{equation}\label{equ Bn n+1}
  B_{n, n+1}^+=-B_{n+1, n}^-=(rs)^{-\frac{1}{2}}(r+s)^{\frac{1}{2}}(r-s),
\end{equation}
\begin{equation}\label{equ Bn+1, n'}
  B_{n+1, n'}^+=-B_{n', n+1}^-=-r^{-1}(r+s)^{\frac{1}{2}}(r-s),
\end{equation}
and
\begin{equation}\label{equ B(i+1)' i'}
  B_{(i+1)', i'}^+=-B_{i', (i+1)'}^-=-r^{-1}s^{-1}(r^2-s^2).
\end{equation}

(VI) We also verify that this distribution law is indeed compatible with the metric condition (\ref{metric condition})
 in the Appendix.
\end{proof}

\begin{remark}\label{rmk Lyn decomp}
 For the non-unique decomposition $\gamma=\alpha+\beta=\alpha'+\beta'$ such that $\alpha, \alpha'<\gamma<\beta, \beta'$,
we have two quantum Lie bracketing rules:
$$\mathcal{E}_\gamma=[\mathcal{E}_{\alpha'}, \mathcal{E}_{\beta'}]_{\langle \omega'_{\alpha'}, \omega_{\beta'}\rangle^{-1}}=[\mathcal{E}_{\alpha}, \mathcal{E}_{\beta}]_{\langle \omega'_{\alpha}, \omega_{\beta}\rangle^{-1}}.$$
We have seen that the $RLL$ formalism regulates the ways of bracketing, and our algorithm (see Theorem \ref{thm Lyn} and its proof) provides one of such ways.
This is exactly the reason why we added (\ref{Lyn 22}).
Compared with the original Definition in \cite{HuWangJGP 2010},
this expression is something newly added.
A similar situation also applies to another Lyndon base $\mathcal{E}'_\gamma$ with adding (\ref{Lyn 12}).
\end{remark}

The next results reveal that our algorithm is independent of the ways of bracketing.
Indeed, for the entry $\ell_{ij}^+$ of $L^+$, which corresponds to $\mathcal{E}_{ij}(\omega'_{\epsilon_i})^{-1}$,
we can obtain this entry not only from the quantum Lie bracket of  $\ell_{i, j-1}^+\; \text{and}\; \ell_{j-1, j}^+$ (this is exactly our algorithm in Theorem \ref{thm Lyn}) , but also from that of $\ell_{ik}^+ \; \text{and}\; \ell_{kj}^+$, where $i< k< j$. The same situation holds for  $\ell_{ij'}^+$.

\begin{lemm}\label{lemm relaL}
  We have the following identities in the $L^+$-matrix,
  \begin{align}
    \ell_{ik}^+ &=\frac{rs}{rs^{-1}-r^{-1}s} (\ell_{jj}^+)^{-1}\Bigl[\ell_{ij}^+\ell_{jk}^+-\ell_{jk}^+\ell_{ij}^+\Bigr],\quad i<j<k<n+1;  \label{equ ell ik}\\
     \ell_{i, n+1}^+&=\frac{1}{rs^{-1}-r^{-1}s} (\ell_{jj}^+)^{-1}\Bigl[(rs)\ell_{ij}^+\ell_{j, n+1}^+-\ell_{j, n+1}^+\ell_{ij}^+\Bigr], \quad i<j<n+1; \\
     \ell_{ik'}^+&=\frac{1}{rs^{-1}-r^{-1}s} (\ell_{jj}^+)^{-1}\Bigl[(rs)\ell_{ij}^+\ell_{jk'}^+-(rs)^{-1}\ell_{jk'}^+\ell_{ij}^+\Bigr], \quad i<j<k\leq n; \\
     \ell_{k'i'}^+ &=\frac{(rs)^{-1}}{rs^{-1}-r^{-1}s} (\ell_{j'j'}^+)^{-1}\Bigl[\ell_{k'j'}^+\ell_{j'i'}^+-\ell_{j'i'}^+\ell_{k'j'}^+\Bigr],\quad i<j<k<n+1;\\
     \ell_{n+1, i'}^+&=\frac{1}{rs^{-1}-r^{-1}s} (\ell_{j'j'}^+)^{-1}\Bigl[\ell_{n+1,j'}^+\ell_{j'i'}^+-(rs)^{-1}\ell_{j' i'}^+\ell_{n+1,j'}^+\Bigr], \quad i<j<n+1; \\
     \ell_{ki'}^+&=\frac{1}{rs^{-1}-r^{-1}s} (\ell_{j'j'}^+)^{-1}\Bigl[(rs)\ell_{kj'}^+\ell_{j'i'}^+-(rs)^{-1}\ell_{j'i'}^+\ell_{kj'}^+\Bigr], \quad i<j<k\leq n;
  \end{align}
\end{lemm}
\begin{proof}
  These identities can be proved similarly as those in the proof of Theorem \ref{thm Lyn}.
\end{proof}

\begin{prop}\label{prop recurrence}
  The recurrence relations obtained in Lemma \ref{lemm relaL} turn to the following relations of Lyndon bases:
  \begin{align}
    \mathcal{E}_{ik} &=\mathcal{E}_{ij}\mathcal{E}_{jk}-s^2\mathcal{E}_{jk}\mathcal{E}_{ij}, \quad i<j<k\leq n+1; \label{equ eik} \\
    \mathcal{E}_{ik'}&= \mathcal{E}_{ij}\mathcal{E}_{jk'}-s^2\mathcal{E}_{jk'}\mathcal{E}_{ij}, \quad i<j<k\leq n;\\
    \mathcal{E}'_{ik} &=\mathcal{E}'_{ij}\mathcal{E}'_{jk}-r^2\mathcal{E}'_{jk}\mathcal{E}'_{ij}, \quad i<j<k\leq n+1;  \\
    \mathcal{E}'_{ik'}&= \mathcal{E}'_{ij}\mathcal{E}'_{jk'}-r^2\mathcal{E}'_{jk'}\mathcal{E}'_{ij}, \quad i<j<k\leq n;
  \end{align}
\end{prop}

\begin{proof}
   Here we only show the derivation of (\ref{equ eik}) since the others can be obtained similarly.
   First, we assume that $i<j<k\leq n$.
  From (\ref{equ ell ik}) and Theorem \ref{thm Lyn} , we calculate that
        \begin{align*}
     \ell_{ik}^+ & =\frac{rs}{rs^{-1}-r^{-1}s} (\ell_{jj}^+)^{-1}\bigl(\ell_{ij}^+\ell_{jk}^+-\ell_{jk}^+\ell_{ij}^+\bigr), \\
      &=\frac{rsB_{ij}^+B_{jk}^+}{rs^{-1}-r^{-1}s}\omega'_{\epsilon_j}\bigl[\mathcal{E}_{ij} (\omega'_{\epsilon_i})^{-1}\mathcal{E}_{jk}(\omega'_{\epsilon_j})^{-1}-\mathcal{E}_{j k}(\omega'_{\epsilon_j})^{-1}\mathcal{E}_{ij} (\omega'_{\epsilon_i})^{-1}\bigr]\\
      &=\frac{B_{ij}^+B_{jk}^+}{r^2-s^2}\bigl[\mathcal{E}_{ij}\mathcal{E}_{jk}-s^2\mathcal{E}_{jk}\mathcal{E}_{ij}\bigr](\omega'_{\varepsilon_i})^{-1};
   \end{align*}
   In particular, letting $j=i+1$, we have
   \begin{align*}
   \ell_{ik}^+ &=\frac{B_{i, i+1}^+B_{i+1, k}^+}{r^2-s^2}\bigl[e_i\mathcal{E}_{i+1, k}-s^2\mathcal{E}_{i+1, k}e_i\bigr](\omega'_{\epsilon_i})^{-1} \\
     &= \frac{B_{i, i+1}^+B_{i+1, k}^+}{r^2-s^2} \mathcal{E}_{ik}(\omega'_{\epsilon_i})^{-1}.
   \end{align*}
   It follows from Theorem \ref{thm Lyn} that $B_{i, i+1}^+B_{i+1, k}^+=B_{ij}^+B_{jk}^+=(r^2-s^2)^2$ for any $i<j<k\leq n$.
   Thus we must have $\mathcal{E}_{ik}=\mathcal{E}_{ij}\mathcal{E}_{jk}-s^2\mathcal{E}_{jk}\mathcal{E}_{ij}.$

   Finally, we need to verify that $\mathcal{E}_{i, n+1} =\mathcal{E}_{ij}\mathcal{E}_{j, n+1}-s^2\mathcal{E}_{j, n+1}\mathcal{E}_{ij}$ for $i<j<n+1$.
   From the $RLL$ relation, we have
   \begin{align*}
     \ell_{i, n+1}^+&=\;\frac{1}{rs^{-1}-r^{-1}s}(\ell_{jj}^+)^{-1}\Bigl[rs\ell_{ij}^+\ell_{j, n+1}^+-\ell_{j, n+1}^+\ell_{ij}^+\Bigr]  \\
    &=\; \frac{B_{ij}^+ B_{j, n+1}^+}{rs^{-1}-r^{-1}s} (\omega'_{\epsilon_j})^{-1}\Bigl[rs \mathcal{E}_{ij} (\omega'_{\epsilon_i})^{-1}\mathcal{E}_{j, n+1} (\omega'_{\epsilon_j})^{-1}-\mathcal{E}_{j, n+1} (\omega'_{\epsilon_j})^{-1}\mathcal{E}_{ij} (\omega'_{\epsilon_i})^{-1}\Bigr]
   \end{align*}
   for $i<j<n+1$. In particular, this holds for $j=n$.
   It follows from Theorem \ref{thm Lyn} that $B_{in}^+B_{n, n+1}^+=B_{ij}^+B_{j, n+1}^+=(rs)^{-\frac{1}{2}}(r+s)^{\frac{3}{2}}(r-s)^2$ for any $i<j<n+1$.
   Then we have the following identity
   \begin{align*}
      &rs\mathcal{E}_{i, n}(\omega'_{\epsilon_i})^{-1}e_n(\omega'_n)^{-1}-e_n(\omega'_n)^{-1}\mathcal{E}_{i, n}(\omega'_{\epsilon_i})^{-1} \\
      =\; &\frac{\omega'_{\epsilon_j}}{\omega'_n}\Bigl[ rs\mathcal{E}_{i, j}(\omega'_{\epsilon_i})^{-1}\mathcal{E}_{j, n+1}(\omega'_j)^{-1}-\mathcal{E}_{j, n+1}(\omega'_j)^{-1}\mathcal{E}_{i, j}(\omega'_{\epsilon_i})^{-1}\Bigr].
   \end{align*}
   We can obtain the desired identity $\mathcal{E}_{i, n+1} =\mathcal{E}_{ij}\mathcal{E}_{j, n+1}-s^2\mathcal{E}_{j, n+1}\mathcal{E}_{ij}$ through some basic commutator calculations.
\end{proof}

\begin{remark}
  These identities were originally established in Lemma 3.1 \cite{HuWangJGP 2010}.
  Here we recover these via the $RLL$ formalism.
\end{remark}

\subsection{An isomorphism between Drinfeld-Jimbo presentation and FRT presentation}
With  Theorem \ref{thm Lyn} in hand, we are able to contribute an intuitive new proof for establishing an isomorphism between the Drinfeld-Jimbo presentation
and FRT presentation of $U_{r, s}(\mathfrak{so}_{2n+1})$, which replaces the Ding-Frenkel's expository analytical proof in one-parameter setup (see Theorem 2.1 in \cite{Ding CMP 1993}).

\begin{theorem}\label{thm FRTISO}
  There is an algebra isomorphism  between $U_{r,s}(\mathfrak{so}_{2n+1})$ and $U(\widehat{R}_{r,s})$,
where we use the subscripts $r, s$ to emphasize that the $R$-matrix depend on the parameters $r, s$.
\end{theorem}

\begin{proof}
(I) Define a map
$\phi_{r, s}: U_{r, s}\bigl(\mathfrak{so}_{2n+1}\bigr)\longrightarrow U(\widehat{R}_{r, s})$ as follows
\begin{equation*}
\begin{aligned}
e_j&\mapsto \frac{1}{r^2-s^2}\ell_{j,j+1}^+\Bigl(\ell_{jj}^+\Bigr)^{-1}:=\tilde e_j, &
f_j&\mapsto -\frac{1}{r^2-s^2}\Bigl(\ell_{jj}^-\Bigr)^{-1}\ell_{j+1,j}^-:=\tilde f_j, \\
e_n&\mapsto c^{-1}\ell_{n,n+1}^+\Bigl(\ell_{nn}^+\Bigr)^{-1}:=\tilde e_n,&
f_n&\mapsto -c^{-1}\Bigl(\ell_{nn}^-\Bigr)^{-1}\ell_{n+1,n}^-:=\tilde f_n, \\
\omega_j&\mapsto \Bigl(\ell_{jj}^-\Bigr)^{-1}\ell_{j+1,j+1}^-:=\tilde\omega_j, &
\omega_j'&\mapsto \Bigl(\ell_{jj}^+\Bigr)^{-1}\ell_{j+1,j+1}^+:=\tilde\omega_j',\\
\omega_n&\mapsto \Bigl(\ell_{nn}^-\Bigr)^{-1}:=\tilde\omega_n,  &
\omega_n'&\mapsto \Bigl(\ell_{nn}^+\Bigr)^{-1}:=\tilde\omega_n', \\
\end{aligned}
\end{equation*}
where $\epsilon_i=\alpha_i+\alpha_{i+1}+\cdots+\alpha_n, 1\leq i\leq n$, $1\leq j\leq n-1$,
$c=(rs)^{-\frac{1}{2}}(r+s)^{\frac{1}{2}}(r-s)$.

From Theorem \ref{thm Lyn}, we see that the generators of $U(\widehat{R}_{r, s})$ come from those generators of two PBW quantum Lyndon bases of $U_{r, s}(\mathfrak{so}_{2n+1})$.
 Obviously,  $\phi_{r, s}$ is not only surjective, but injective as well (the next step (II) of our proof, also shows $U(\widehat{R}_{r, s})$ contains a copy of $U_{r, s}(\mathfrak{so}_{2n+1})$. The reasoning is the same as that of Theorem 8.33 \cite{Klimyk}).

(II) It suffices to show that $\phi_{r, s}$ preserves the defining relations of $U_{r, s}(\mathfrak{so}_{2n+1})$.
Here we restrict Definition \ref{type B def} of $U_{r, s}(\widehat{\mathfrak{so}_{2n+1}})$ to the finite type: $U_{r, s}\bigl(\mathfrak{so}_{2n+1}\bigr)$.
We use induction on $n$ and start with $n=2$.

Firstly, from (\ref{$RLL$}), we conclude that $\ell_{ii}^\pm, \ell_{jj}^\pm\ (1\leq i, j \leq 5)$ commutate with each other.
This means $\phi_{r, s}$ preserves (B1) in Definition \ref{type B def}.

From $\widehat{R}_{r, s}L_1^{+}L_2^{+}=L_2^{+}L_1^{+}\widehat{R}_{r, s}$, we derive the following relations:
\begin{equation}\label{equ ellii++}
  \ell_{11}^+\ell_{12}^+=r^2\ell_{12}^+\ell_{11}^+, \quad \ell_{11}^+\ell_{23}^+=(rs)^{-1}\ell_{23}^+\ell_{11}^+, \quad
  \ell_{22}^+\ell_{12}^+=s^2\ell_{12}^+\ell_{22}^+, \quad \ell_{22}^+\ell_{23}^+=rs^{-1}\ell_{23}^+\ell_{22}^+,
\end{equation}
From $\widehat{R}_{r, s}L_1^{+}L_2^{-}=L_2^{-}L_1^{+}\widehat{R}_{r, s}$, we have
\begin{equation}\label{equ ellii+-}
   \ell_{11}^+\ell_{21}^-=r^{-2}\ell_{21}^-\ell_{11}^+, \quad \ell_{11}^+\ell_{32}^-=(rs)\, \ell_{32}^-\ell_{11}^+, \quad
   \ell_{22}^+\ell_{21}^-=s^{-2}\ell_{21}^-\ell_{22}^+, \quad \ell_{22}^+\ell_{32}^-=r^{-1}s\, \ell_{32}^-\ell_{22}^+.
\end{equation}
Here we take the derivation of $\ell_{11}^+\ell_{12}^+=r^2\ell_{12}^+\ell_{11}^+$ as an example to illustrate.
Starting from $\widehat{R}_{r, s}L_1^{+}L_2^{+}(v_1\otimes v_2)=L_2^{+}L_1^{+}\widehat{R}_{r, s}(v_1\otimes v_2)$,
we can obtain the desired equation by comparing the coefficients of $v_1\otimes v_1$ on both sides.
(\ref{equ ellii++}) and (\ref{equ ellii+-}) yield that $\phi_{r, s}$ preserves (B3) in Definition \ref{type B def}.

Similarly, one can also prove $\phi_{r, s}$ preserves (B2) from $\widehat{R}_{r, s}L_1^{-}L_2^{-}=L_2^{-}L_1^{-}\widehat{R}_{r, s}$
and $\widehat{R}_{r, s}L_1^{+}L_2^{-}=L_2^{-}L_1^{+}\widehat{R}_{r, s}$.

From $\widehat{R}_{r, s}L_1^{+}L_2^{-}=L_2^{-}L_1^{+}\widehat{R}_{r, s}$, we also have
\begin{equation}\label{}
  rs\ell_{12}^+\ell_{21}^--(rs)^{-1}\ell_{21}^-\ell_{12}^+=(r^{-1}s-rs^{-1})(\ell_{22}^-\ell_{11}^+-\ell_{22}^+\ell_{11}^-),
\end{equation}
\begin{equation}\label{}
  \ell_{23}^+\ell_{32}^--\ell_{32}^-\ell_{23}^+=(r^{-1}s-rs^{-1})(\ell_{33}^-\ell_{22}^+-\ell_{33}^+\ell_{22}^-).
\end{equation}
These imply that $\phi_{r, s}$ preserves (B4) in Definition \ref{type B def}.

Also, we derive the following equations from $\widehat{R}_{r, s}L_1^{+}L_2^{+}=L_2^{+}L_1^{+}\widehat{R}_{r, s}$:
\begin{equation}\label{12 23}
(rs)\,\ell_{12}^+\ell_{23}^++(r^{-1}s-rs^{-1})\ell_{22}^+\ell_{13}^+=\ell_{23}^+\ell_{12}^+,
\end{equation}
\begin{equation}\label{12 13}
  \ell_{12}^+\ell_{13}^+=rs^{-1}\ell_{13}^+\ell_{12}^+,
\end{equation}
\begin{equation}\label{22 23}
  \ell_{22}^+\ell_{23}^+=rs^{-1}\ell_{23}^+\ell_{22}^+,
\end{equation}
\begin{equation}\label{14 22}
  \ell_{14}^+\ell_{22}^+=r^{-2}\ell_{22}^+\ell_{14}^+,
\end{equation}
\begin{equation}\label{13 22}
  \ell_{13}^+\ell_{22}^+=(rs)^{-1}\ell_{22}^+\ell_{13}^+.
\end{equation}
\begin{equation}\label{13 23}
  (rs)\,\ell_{13}^+\ell_{23}^+-(r^{-1}s-rs^{-1})r^{-\frac{1}{2}}s^{\frac{1}{2}}\ell_{22}^+\ell_{14}^+=\ell_{23}^+\ell_{13}^+.
\end{equation}

By (\ref{12 23}) and (\ref{12 13}), we have
\begin{equation}\label{rel Serre1}
  (rs)\,\ell_{12}^{+2}\ell_{23}^++rs^{-3}\ell_{23}^+\ell_{12}^{+2}=(1+r^2s^{-2})\ell_{12}^+\ell_{23}^+\ell_{12}^+.
\end{equation}

 With (\ref{12 23}), (\ref{22 23}), (\ref{14 22}), (\ref{13 22}) and (\ref{13 23}), we conclude that
\begin{equation}\label{rel Serre2}
  \ell_{23}^{+3}\ell_{12}^+=rs^3(r^{-2}+r^{-1}s^{-1}+s^{-2})\ell_{23}^{+2}\ell_{12}^+\ell_{23}^+
-rs^5(r^{-2}+r^{-1}s^{-1}+s^{-2})\ell_{22}^+\ell_{12}^+\ell_{23}^{+2}+s^6\ell_{12}^+\ell_{23}^{+3}.
\end{equation}
(\ref{rel Serre1}) and (\ref{rel Serre2}) show that $\phi_{r, s}$ preserves Serre relations (B5) in Definition \ref{type B def}.

This completes the proof for $n=2$.

Using induction and restricting the generating relations (\ref{$RLL$}) to $E_{ij}\otimes E_{kl}, 2\leq i,j,k,l\leq 2n+1$,
we can get all commutation relations except those between $\ell_{11}^{\pm}, \ell_{12}^+, \ell_{21}^-$ and $\ell_{ii}^{\pm},\ell_{ij}^\pm$.
However, those can be derived easily by repeating the above computations.
Therefore, we complete our proof for general $n$.
\end{proof}

\smallskip
\section{Spectral parameter-dependent $R$-matrices}\label{sec specR}

\subsection{The Yang-Baxterization and Gauss decomposition}
In Ge-Wu-Xue's work \cite{GeWX1991}, they introduced a method to construct the spectral-parameter $R$-matrix from the basic braided $R$-matrix.
This method is called `the Yang-Baxterization' (also named affinization).
For those basic braided $R$-matrices which have three eigenvalues, they gave two different affinizations:
\begin{equation}\label{Rz 1}
  R(x)=\lambda_1 x(x-1) S^{-1}+\Bigl(1+\frac{\lambda_1}{\lambda_2}+\frac{\lambda_1}{\lambda_3}+\frac{\lambda_2}{\lambda_3}\Bigr)xI -\lambda_3^{-1}(x-1)S,
\end{equation}
\begin{equation}\label{Rz 2}
  R(x)=\lambda_1 x(x-1)S^{-1}+\Bigl(1+\frac{\lambda_1}{\lambda_2}+\frac{\lambda_1}{\lambda_3}+\frac{\lambda_1^2}{\lambda_2\lambda_3}\Bigr)xI
-\frac{\lambda_1}{\lambda_2\lambda_3}(x-1)S,
\end{equation}
where $S$ is a given basic braided $R$-matrix,
and $\lambda_i\; (i=1, 2, 3)$ are the eigenvalues of $S$.
Let $S^{-1}=R,\; z=x^{-1}$, then $\lambda_1=rs^{-1}, \lambda_2=-r^{-1}s \;\text{and}\; \lambda_3=(r^{-1}s)^{2n}$
from Lemma \ref{lemma minimal polynomial}.
We also have the formula of $S=R^{-1}$ in Lemma \ref{lemm Rinverse}.
They both satisfy the spectral braided QYBE:
$$R_{12}(x)R_{23}(xy)R_{12}(y)=R_{23}(y)R_{12}(xy)R_{23}(x).$$

To get the following spectral parameter-dependent $\widehat{R}_{r,s}(z):=P\circ R(z)$, where $P$ is the flip map, we adopt (\ref{Rz 1}) rather than (\ref{Rz 2}). The reason is that the latter doesn't have the intertwining property for minimal affinization of $U_{r, s}(\widehat{\mathfrak{so}_{2n+1}})$ (See subsection \ref{sec vecAffine}).

\begin{theorem}\label{thm spectral Rz1}
  The spectral parameter-dependent $\widehat{R}_{r,s}(z)$ is given by
\begin{align*}
\widehat{R}_{r,s}(z)=&\;\sum_{i \atop i\neq i'}E_{ii}\otimes E_{ii}+\frac{z-1}{r^2z-s^2}\Bigl\{\Bigl(\sum_{1\leq i\leq n \atop i+1\leq j\le n}
E_{jj}\otimes E_{ii}+\sum_{2\leq i\leq n \atop (i-1)'\leq j\leq 1'}E_{jj}\otimes E_{ii}\\
&+\sum_{1\leq i\leq n-1 \atop n'\leq j\leq (i+1)'}E_{ii}\otimes E_{jj}+\sum_{n'\leq i\leq 2'\atop i+1\leq j\leq 1'}E_{ii}\otimes E_{jj}\Bigr)+r^2s^2\Bigl(\sum_{1\leq i\leq n\atop i+1\leq j\leq n}E_{ii}\otimes E_{jj}\\
&+\sum_{n'\leq i\leq 2'\atop i+1\leq j\leq 1'}E_{jj}\otimes E_{ii}+\sum_{2\leq i\leq n\atop (i-1)'\leq j\leq 1'}E_{ii}\otimes E_{jj}+\sum_{1\leq i\leq n-1\atop n'\leq j\leq (i+1)'}E_{j j}\otimes E_{i i}\Bigr)\\
&+rs\Bigl(\sum_{i\atop i\neq i'} E_{n+1, n+1}\otimes E_{ii}+\sum_{j\atop j\neq j'}E_{j j}\otimes E_{n+1, n+1}\Bigr)\Bigr\}\\
&+\frac{r^2-s^2}{r^2z-s^2}\Bigl\{\sum_{i<j \atop i' \neq j}E_{ij}\otimes E_{ji}+z\sum_{i>j \atop i' \neq j }E_{ij}\otimes E_{ji}\Bigr\}\\
&+\frac{1}{(z-(r^{-1}s)^{2n-1})(r^2z-s^2)}\sum_{i,j=1}^{2n+1}d_{ij}(z, 1)E_{i' j'}\otimes E_{i j},
\end{align*}
where $d_{ij}(z, 1)=\left\{
                      \begin{array}{ll}
                        (s^2-r^2)z\Bigl\{(z-1)(r^{-1}s)^{\rho_i-\rho_j}-\delta_{i, j'}[z-(r^{-1}s)^{2n-1}]\Bigr\}, & \hbox{$i<j;$} \\
                        (s^2-r^2)\Bigl\{(z-1)(r^{-1}s)^{2n-1+\rho_i-\rho_j}-\delta_{i, j'}[z-(r^{-1}s)^{2n-1}]\Bigr\}, & \hbox{$i>j;$} \\
                        s^2(z-1)[z-(r^{-1}s)^{2n-3}], & \hbox{$i=j\neq i';$} \\
                       rs(z-1)[z-(r^{-1}s)^{2n-1}]+(r^2-s^2)z[1-(r^{-1}s)^{2n-1}], & \hbox{$i=j=i'.$}
                      \end{array}
                    \right.$

In particular, when $rs=1$, i.e., $r=q^{\frac{1}{2}}=s^{-1}$, $\widehat{R}_{r,s}(z)$ degenerates into $P\circ\widehat{R}_q(z)\circ P$,
where the $R_q(z)$ is given in \cite{JingLM SIGMA 2020}, $P$ is the usual flip map.
\end{theorem}

\begin{proof}
  According to \cite{CGeX CMP1991,GeWX1991}, it suffices to verify that the basic $R$-matrix admits the Birman-Wenzl-Murakami algebraic structure. Indeed, this is true from Proposition \ref{prop BWrep}.
\end{proof}

\begin{remark}
  We have the following identities by direct verifications:
\begin{equation*}\label{R1}
  R_{r,s}(1)=1,
\end{equation*}
\begin{equation*}\label{unitary condition}
  R_{r,s}(z)R_{r,s}(z^{-1})=R_{r,s}(z^{-1})R_{r,s}(z)=1,
\end{equation*}
\begin{equation*}\label{Rz iden}
  R_{r,s}(z)=P\circ R_{r,s}(z^{-1},r^{-1},s^{-1})\circ P.
\end{equation*}
\end{remark}


\begin{remark}\label{rmk relation}
The relationships between the spectral parameter-dependent $R$-matrix $\widehat{R}_{r,s}(z)$ and the basic one (and its inverse) are:
$$\widehat{R}_{r,s}(0)=r^{-1}s\,\widehat{R}_{r, s}, \quad \lim\limits_{z\to \infty} \widehat{R}_{r,s}(z)=rs^{-1}\widehat{R}_{r, s}^{-1}.$$
\end{remark}
\begin{defi}\cite{FRT 1989,ResLMP 1990}
 The algebra $\mathcal{U}(\widehat R_{r, s}(z))$ is an associative algebra over $\mathbb{K}$ with generators $\ell^{\pm}_{ij}[\mp m]$, $m\in\mathbb{Z}_+\setminus \{0\}$ and $\ell^+_{kl}[0]$, $\ell^-_{lk}[0]$, $1\le l\le k\le n$ and the central group-like elements $\gamma^{\frac{1}{2}}=r^{\frac{c}{2}}$ and  $\gamma'^{\frac{1}{2}}=s^{\frac{c}{2}}$ ($c$ is the canonical central elements of $\widehat{\mathfrak{g}}$).
Let $\ell^{\pm}_{ij}(z)=\sum\limits_{m=0}^{\infty}\ell^{\pm}_{ij}[\mp m]z^{\pm m}$, where $\ell^+_{kl}[0]=\ell^-_{lk}[0]=0$, for $1\le l<k\le n$. Let $L^{\pm}(z)=\sum\limits_{i,j=1}^{n}E_{ij}\otimes \ell^{\pm}_{ij}(z)$.
Then the relations are given by the following conditions and matrix equations on $\text{End}(V^{\otimes 2})\otimes \mathcal{U}(\widehat{R}_{r, s}(z))$:
\begin{gather}\label{}
  \ell_{ii}^+[0], \ell_{ii}^-[0]\text{ are invertible and } \ell_{ii}^+[0]\ell_{ii}^-[0]=\ell_{ii}^-[0]\ell_{ii}^+[0],
   \label{equ 5.1} \\
  \widehat{R}_{r, s}\Bigl(\frac{z}{w}\Bigr)L_1^\pm(z)L_2^\pm(w)=L_2^\pm(w)L_1^\pm(z)\widehat{R}_{r, s}\Bigl(\frac{z}{w}\Bigr),\label{equ 5.2}\\
  \widehat{R}_{r, s}\Bigl(\frac{z_+}{w_-}\Bigr)L_1^+(z)L_2^-(w)=L_2^-(w)L_1^+(z)\widehat{R}_{r, s}\Bigl(\frac{z_-}{w_+}\Bigr),\label{equ 5.3} \\
  L^\pm(z)CL^\pm(z\xi)^tC^{-1}=1, \label{equ metric}
\end{gather}
where $z_+=zr^{\frac{c}{2}}$ and $z_-=zs^{\frac{c}{2}}$.
Here (\ref{equ 5.2}) is expanded in the direction of either $\frac{z}{w}$ or $\frac{w}{z}$, and (\ref{equ 5.3}) is expanded in the direction of $\frac{z}{w}$.
In (\ref{equ metric}),  $C$ is the metric matrix defined in Definition \ref{metric matrix}, $t$ denotes the standard matrix transposition, and $\xi=(r^{-1}s)^{2n-1}$.
\end{defi}

\begin{remark}
  From Equation (\ref{equ 5.3}) and the unitary condition of $\widehat{R}$-matrix (\ref{unitary condition}), we have
\begin{equation}\label{equ 5.4}
  \widehat{R}_{r, s}\Bigl(\frac{z_\pm}{w_\mp}\Bigr)L_1^\pm(z)L_2^\mp(w)=L_2^\mp(w)L_1^\pm(z)\widehat{R}_{r, s}\Bigl(\frac{z_\mp}{w_\pm}\Bigr).
\end{equation}
\end{remark}

\begin{remark}\label{rmk matrixexp}
  Here we present the specific matrix expression formulas for (\ref{equ 5.2}).
$$L^\pm(z)=\left(
           \begin{array}{ccccc}
             \ell_{11}^\pm(z) & \ell_{12}^\pm(z) & \cdots   & \ell_{1, 2n+1}^\pm(z) \\
             \ell_{21}^\pm(z) & \ell_{22}^\pm(z) & \ddots   & \vdots \\
             \vdots & \ddots & \ddots   & \ell_{2n, 2n+1}^\pm(z) \\
             \ell_{2n+1, 1}^\pm(z) & \cdots &   \ell_{2n+1, 2n}^\pm(z) & \ell_{2n+1, 2n+1}^\pm(z) \\
           \end{array}
         \right)_{(2n+1)\times (2n+1)},$$
then for the generators $L_1^\pm(z), L_2^\pm(z), \widehat{R}_{r, s}(z)$, we have that
$$L_1^\pm(z)=\left(
               \begin{array}{ccc}
                 \ell_{11}^\pm(z)I_{2n+1} & \cdots & \ell_{1,2n+1}^\pm(z)I_{2n+1}  \\
                 \vdots & \cdots & \vdots \\
                 \ell_{2n+1, 1}^\pm(z)I_{2n+1} & \cdots &\ell_{2n+1, 2n+1}^\pm(z)I_{2n+1}  \\
               \end{array}
             \right)_{(2n+1)^2\times (2n+1)^2},$$
$$L_2^\pm(z)=\left(
               \begin{array}{cccc}
                 L^\pm(z) & 0 & \cdots & 0\\
                 0 & L^\pm(z) & \ddots & \vdots \\
                 \vdots & \ddots & \ddots & 0 \\
                 0 & \cdots & 0 & L^\pm(z) \\
               \end{array}
             \right)_{(2n+1)^2\times (2n+1)^2},$$
$$\widehat{R}_{r,s}(z)=\left(
               \begin{array}{ccc}
                 B_{11}(z) & \cdots & B_{1, 2n+1}(z) \\
                 \vdots & \cdots & \vdots \\
                 B_{2n+1, 1}(z) & \cdots & B_{2n+1, 2n+1}(z) \\
               \end{array}
             \right)_{(2n+1)^2\times (2n+1)^2},$$

$\bullet$ $B_{\ell\ell}(z)$ is a diagonal matrix, and $a_{\ell j}(z)$ is the coefficient of element $E_{\ell\ell}\otimes E_{jj}$ in $\widehat{R}_{r, s}(z)$.

$\bullet$ $ B_{ij}(z)=b_{ij}(z)E_{ji}+c_{i' j'}(z)E_{i' j'}$,
where $b_{ij}(z)$ is the coefficient of element $E_{ij}\otimes E_{ji}$ in $\widehat{R}_{r, s}(z)$,
and $c_{ij}$ is the coefficient of element $E_{i' j'}\otimes E_{ij}$ in $\widehat{R}_{r, s}(z)$.
Assume
$$\widehat{R}_{r, s}\Bigl(\frac{z}{w}\Bigr)L_1^\pm(z)L_2^\pm(w)=\left(
                                             \begin{array}{ccc}
                                               M_{11} & \cdots & M_{1, 2n+1} \\
                                               \vdots & \cdots & \vdots \\
                                               M_{2n+1,1} & \cdots & M_{2n+1,2n+1} \\
                                             \end{array}
                                           \right)_{(2n+1)^2\times (2n+1)^2},
$$
$$L_2^\pm(w)L_1^\pm(z)\widehat{R}_{r, s}\Bigl(\frac{z}{w}\Bigr)=\left(
                                            \begin{array}{ccc}
                                               M_{11}' & \cdots & M_{1, 2n+1}' \\
                                               \vdots & \cdots & \vdots \\
                                               M_{2n+1,1}' & \cdots & M_{2n+1,2n+1}' \\
                                            \end{array}
                                          \right)_{(2n+1)^2\times (2n+1)^2},
$$
where $M_{ij}, M'_{ij}$ are square matrices of order $2n+1$.

We have $M_{ij}=M_{ij}'$, where $M_{ij}=$
\begin{equation*}\label{5.9}
\left(
           \begin{array}{ccccccc}
             a_{i1}(\frac{z}{w})\ell_{ij}^\pm(z) &  & b_{i1}(\frac{z}{w})\ell_{1j}^\pm(z) &  &  &  &\\
              & \ddots & \vdots &  &  &  &  \\
              &  & a_{ii}(\frac{z}{w})\ell_{ij}^\pm(z) &  &  &  &  \\
              &  & \vdots & \ddots &  &  &  \\
             c_{i' 1}(\frac{z}{w})\ell_{1' j}^\pm(z) & \cdots & c_{i' i}(\frac{z}{w})\ell_{i' j}^\pm(z) & \cdots & c_{i' i'}(\frac{z}{w})\ell_{ij}^\pm(z) & \cdots & c_{i' 1'}(\frac{z}{w})\ell_{1j}^\pm(z) \\
              &  & \vdots &  &  & \ddots &  \\
              &  & b_{i 1'}(\frac{z}{w})\ell_{1' j}^\pm(z) &  &  &  & a_{i 1'}(\frac{z}{w})\ell_{ij}^\pm(z) \\
           \end{array}
         \right)L^\pm(w),
\end{equation*}
with non-zero elements only on the diagonal, the $i'$-th row, and the $i$-th column;
and $M_{ij}'=$
\begin{equation*}\label{5.11}
L^\pm(w)\left(
           \begin{array}{ccccccc}
             a_{j1}(\frac{z}{w})\ell_{ij}^\pm(z) &  & &  &c_{1j'}(\frac{z}{w})\ell_{i1'}^\pm(z) &  &  \\
              & \ddots &  &  &\vdots  &  &  \\
             b_{1, j}(\frac{z}{w})\ell_{i1}^\pm(z)& \cdots & a_{jj}(\frac{z}{w})\ell_{ij}^\pm(z) &\cdots  & c_{jj' }(\frac{z}{w})\ell_{i j'}^\pm(z) &\cdots  & b_{1' j}(\frac{z}{w})\ell_{i 1'}^\pm(z) \\
              &  &  & \ddots &\vdots  &  &  \\
              &  &  &  & c_{j' j'}(\frac{z}{w})\ell_{ij}^\pm(z)&& \\
              &  & &  &\vdots  & \ddots &  \\
              &  &  &  & c_{1' j'}(\frac{z}{w})\ell_{i1}^\pm(z) &  & a_{j 1'}(\frac{z}{w})\ell_{ij}^\pm(z) \\
           \end{array}
         \right),
\end{equation*}
with non-zero elements only on the diagonal, the $j$-th row, and the $j'$-th column.
\end{remark}

\begin{defi}\label{def quasidet}\cite{GGRW}
 Let $X$ = $(x_{ij})^{n}_{i,j=1}$ be a sequence matrix over a ring with identity.
 Denote by $X^{ij}$ the submatrix obtained from $X$ by deleting the $i$-th row and $j$-th column.
 Suppose that the matrix $X^{ij}$ is invertible.
 The $(i,j)$-th quasi-determinant $|X|_{ij}$ of $X$ is defined by
\begin{equation*}
|X|_{ij}=\left|\begin{array}{rrrrr}x_{11}&\cdots&x_{1j}&\cdots&x_{1n}\\ \qquad &\cdots&\qquad &\cdots&\qquad \\x_{i1}&\cdots&\boxed{x_{ij}}&\cdots&x_{in}\\\qquad &\cdots&\qquad &\cdots&\qquad \\x_{n1}&\cdots&x_{nj}&\cdots&x_{nn}\end{array}\right|=x_{ij}-r^j_i(X^{ij})^{-1}c^i_j,
\end{equation*}
where $r^j_i$ is the row matrix obtained from the $i$-th row of $X$ by deleting the element $x_{ij}$, and $c^i_j$ is the column matrix obtained from the $j$-th column of $X$ by deleting the element $x_{ij}$.
\end{defi}

According to \cite{GGRW}, we have the unique Gaussian decomposition of $L^\pm(z)$:
\begin{prop}\label{Prop Gauss decomp}
 $$ L^\pm(z)=F^\pm(z)K^\pm(z)E^\pm(z),$$
where
\begin{equation*}\label{5.15}
 F^\pm(z)=\left(
            \begin{array}{cccc}
              1 &  &  &  \\
              f_{21}^\pm(z) & \ddots &  &  \\
              \vdots & \ddots & \ddots &  \\
              f_{2n+1,1}^\pm(z) & \cdots & f_{2n+1,2n}^\pm(z) & 1 \\
            \end{array}
          \right),
\end{equation*}
\vspace{1em}
\begin{equation*}\label{5.16}
  E^\pm(z)=\left(
             \begin{array}{cccc}
               1 & e_{12}^\pm(z) & \cdots & e_{1,2n+1}^\pm(z) \\
                & \ddots & \ddots & \vdots \\
                &  & \ddots & e_{2n, 2n+1}^\pm(z) \\
                &  &  & 1 \\
             \end{array}
           \right),
\end{equation*}
and
\begin{equation*}\label{5.17}
  K^\pm(z)=\left(
             \begin{array}{cccc}
               k_1^\pm(z) &  &  &  \\
                & \ddots &  &  \\
                &  & \ddots & \\
                &  &  & k_{2n+1}^\pm(z) \\
             \end{array}
           \right).
\end{equation*}
Their entries are found by the quasi-determinant formulas:
\begin{equation}\label{equ quasidet k_i}
k^{\pm}_{i}(z)=\left|\begin{array}{cccc}\ell^{\pm}_{11}(z)&\cdots&\ell^{\pm}_{1,i-1}(z)&\ell^{\pm}_{1i}(z)\\ \vdots &\quad&\vdots &\vdots \\ \ell^{\pm}_{i1}(z)&\cdots&\ell^{\pm}_{i,i-1}(z)&\boxed{\ell^{\pm}_{ii}(z)}\end{array}\right|
\end{equation}
for $1\le i\le 2n+1$, $k^{\pm}_{i}(z)=\sum\limits_{m=0}^\infty k^{\pm}_{i}(\mp m)z^{\pm m}$.
\begin{equation}\label{equ quasidet eij}
e^{\pm}_{ij}(z)=k^{\pm}_i(z)^{-1}\left|\begin{array}{cccc}\ell^{\pm}_{11}(z)&\cdots&\ell^{\pm}_{1,i-1}(z)&\ell^{\pm}_{1j}(z)\\ \vdots &\quad&\vdots &\vdots \\ \ell^{\pm}_{i1}(z)&\cdots&\ell^{\pm}_{i,i-1}(z)&\boxed{\ell^{\pm}_{ij}(z)}\end{array}\right|
\end{equation}
for $1\le i<j\le 2n+1$, $e^{\pm}_{ij}(z)=\sum\limits_{m=0}^\infty e^{\pm}_{ij}(\mp m)z^{\pm m}$.
\begin{equation}\label{equ quasidet fji}
f^{\pm}_{ji}(z)=\left|\begin{array}{cccc}\ell^{\pm}_{11}(z)&\cdots&\ell^{\pm}_{1,i-1}(z)&\ell^{\pm}_{1i}(z)\\ \vdots &\quad&\vdots &\vdots \\ \ell^{\pm}_{j1}(z)&\cdots&\ell^{\pm}_{j,i-1}(z)&\boxed{\ell^{\pm}_{ji}(z)}\end{array}\right|
k^{\pm}_i(z)^{-1}
\end{equation}
for $1\le i<j\le 2n+1$, $f^{\pm}_{ji}(z)=\sum\limits_{m=0}^\infty f^{\pm}_{ji}(\mp m)z^{\pm m}$.
\end{prop}

Since $\ell^+_{kl}[0]=\ell^-_{lk}[0]=0$, for $1\le l<k\le n$, then we have
\begin{prop}\label{prop quasidet z=0}
The Gaussian generators are related to the FRT generators as follows:
  $$k_i^\pm(0)=\ell_{ii}^\pm[0], \quad e_{ij}^\pm(0)=(\ell_{ii}^\pm[0])^{-1}\ell_{ij}^\pm[0], \quad f_{ji}^\pm(0)=\ell_{ji}^\pm[0](\ell_{ii}^\pm[0])^{-1}.$$
\end{prop}

\begin{proof}
  From Definition \ref{def quasidet} and Equation (\ref{equ quasidet k_i}),
  by comparing the constant terms on both sides, we obtain
  \begin{align*}
   k_i^\pm(0) & = \ell_{ii}^\pm[0] -
(\ell_{i1}^\pm[0], \ell_{i2}^\pm[0], \dots, \ell_{i, i-1}^\pm[0])
\begin{vmatrix}
\ell_{11}^\pm[0] & \cdots & \ell_{1, i-1}^\pm[0] \\
\vdots & \ddots & \vdots \\
\ell_{i-1, 1}^\pm[0] & \cdots & \ell_{i-1, i-1}^\pm[0]
\end{vmatrix}^{-1}\\
 &\hspace{5em} \cdot (\ell_{1i}^\pm[0], \ell_{2i}^\pm[0], \dots, \ell_{i-1, i}^\pm[0])^t \\
    &= \ell_{ii}^\pm[0],
  \end{align*}
  where $t$ denotes the transposition map,
and the second equality holds since $\ell^+_{kl}[0] = \ell^-_{lk}[0] = 0$ for $1 \le l < k \le n$.
It follows that $k_i^\pm(0) = \ell_{ii}^\pm[0]$.

By (\ref{equ 5.1}), $\ell_{ii}^\pm[0]$ is invertible.
Using Definition \ref{def quasidet} and Equation (\ref{equ quasidet eij}), we obtain
\begin{align*}
e_{ij}^\pm(0) &= (\ell_{ii}^\pm[0])^{-1} \cdot
\begin{vmatrix}
\ell^{\pm}_{11}[0] & \cdots & \ell^{\pm}_{1,i-1}[0] & \ell^{\pm}_{1j}[0] \\
\vdots & & \vdots & \vdots \\
\ell^{\pm}_{i1}[0] & \cdots & \ell^{\pm}_{i,i-1}[0] & \boxed{\ell^{\pm}_{ij}[0]}
\end{vmatrix}  \\
&= (\ell_{ii}^\pm[0])^{-1}\ell_{ij}^\pm[0],
\end{align*}
where the second equality holds by analogy with the previous verifications.
The identity $f_{ji}^\pm(0)=\ell_{ji}^\pm[0](\ell_{ii}^\pm[0])^{-1}$ can be proved analogously by using (\ref{equ quasidet fji}).
\end{proof}

\subsection{Unique spectral parameter $R$-matrix satisfying the intertwining property}\label{sec vecAffine}

In this subsection, we first give the minimal affinization of $U_{r,s}(\widehat {\frak{so}_{2n+1}})$.
Then we can obtain the unique spectral parameter $R$-matrix satisfying the intertwining property (see Proposition \ref{prop intert}) with respect to the minimal affinization \cite{Chari 1995}.

\begin{lemm}\label{lemm vectrep aff}
The vector representation $T$ of $\mathcal{U}_{r,s}\mathcal(\widehat {\frak{so}_{2n+1}})$ is given by:

$\text{\rm(I)}$ For real root vectors of level $k\in\mathbb{Z}$, we have
  \begin{gather*}
T(x_{ik}^+)=(rs^{-1})^{ik}E_{i, i+1}-(rs)^{-1}(rs^{-1})^{(2n-1-i)k}E_{(i+1)', i'},\\
T(x_{ik}^-)=(rs^{-1})^{ik}E_{i+1, i}-(rs)^{-1}(rs^{-1})^{(2n-1-i)k}E_{i', (i+1)'},
\end{gather*}
and
\begin{gather*}
T(x_{nk}^+)=(r+s)^{\frac{1}{2}}\Bigl((rs)^{-\frac{1}{2}}(rs^{-1})^{nk}E_{n, n+1}-r^{-1}(rs^{-1})^{(n-1)k}
E_{n+1, n'}\Bigr),\\
T(x_{nk}^-)=(r+s)^{\frac{1}{2}}\Bigl((rs)^{-\frac{1}{2}}(rs^{-1})^{nk}E_{n+1, n}-s^{-1}(rs^{-1})^{(n-1)k}
E_{n',n+1}\Bigr),
\end{gather*}
with $1\leq i\leq n-1$.

$\text{\rm(II)}$ For imaginary root vectors of level $\ell\neq 0$, we have
\begin{align*}
  T(a_{i\ell})&=-\frac{(rs^{-1})^\ell-(rs^{-1})^{-\ell}}{\ell(r^2-s^2)} \Bigl[(rs^{-1})^{(i+1)\ell} E_{i+1,i+1}-(rs^{-1})^{(i-1)\ell}E_{ii} \\
   &\qquad+(rs^{-1})^{(2n-i)\ell} E_{i',i'}-(rs^{-1})^{(2n-i-2)\ell}E_{(i+1)',(i+1)'}\Bigr],\\
   T(a_{n\ell})&=-\frac{(rs^{-1})^\ell-(rs^{-1})^{-\ell}}{\ell(r-s)}\Bigl[-(rs^{-1})^{(n-1)\ell} E_{n,n}+\Bigl((rs^{-1})^{n\ell}-(rs^{-1})^{(n-1)\ell}\Bigr)E_{n+1,n+1} \\
   &\qquad+(rs^{-1})^{n\ell} E_{n',n'}\Bigr],
\end{align*}
with $1\leq i\leq n-1$.

$\text{\rm(III)}$ The value on group-like elements $\omega_i, \omega_i'$ are as the same as those in Lemma \ref{lemm vect rep}, the finite case.

$\text{\rm(IV)}$ $$T(\gamma)=T(\gamma')=1.$$
\end{lemm}

\begin{remark}
  The vector-representation of $U_{r,s}\mathcal(\widehat {\frak{so}_{2n+1}})$ can degenerate to the one-parameter case when
  one takes $r=q, s=q^{-1}$.
  The one-parameter case was originally established in \cite{JingLM SIGMA 2020}.
  However, there were some small typos since the basic representation should be derived from the vector representation after taking $k=0$.
\end{remark}

\begin{proof}
We need to verify that the representation satisfies the defining relations in Definition \ref{def affine}.
In our case $c=0$. Here prove
$$\Bigl[T(a_{im}), T(x_{ik}^+)\Bigr]=\frac{(rs^{-1})^{2m}-(rs^{-1})^{-2m}}{m(r^2-s^2)}T(x_{i,m+k}^+),\quad 1\leq i\leq n-1,$$
as an example, since the other relations can be verified similarly.
Indeed,
\begin{align*}
  T(a_{im}x_{ik}^+)=\; & \frac{(rs^{-1})^m-(rs^{-1})^{-m}}{m(r^2-s^2)}\Bigl[(rs^{-1})^{i(m+k)-m}E_{i,i+1} \\
   &-(rs)^{-1}(rs^{-1})^{(2n-i-1)(m+k)-m}E_{(i+1)',i'}\Bigr] \\
   T(x_{ik}^+a_{im})=\;& \frac{(rs^{-1})^m-(rs^{-1})^{-m}}{m(r^2-s^2)}\Bigl[-(rs^{-1})^{i(m+k)+m}E_{i,i+1} \\
   &+(rs)^{-1}(rs^{-1})^{(2n-i-1)(m+k)+m}E_{(i+1)',i'}\Bigr]
\end{align*}
Then
\begin{align*}
  \Bigl[T(a_{im}), T(x_{ik}^+)\Bigr]=\;& \frac{(rs^{-1})^{m}-(rs^{-1})^{-m}}{m(r^2-s^2)}\Bigl[(rs^{-1})^{i(m+k)}\Bigl((rs^{-1})^{m}+(rs^{-1})^{-m}\Bigr)E_{i,i+1}\\
   &-(rs^{-1})^{(2n-i-1)(m+k)}\Bigl((rs^{-1})^{m}+(rs^{-1})^{-m}\Bigr)(rs^{-1})E_{(i+1)',i'}\Bigr]\\
   =\;&\frac{(rs^{-1})^{2m}-(rs^{-1})^{-2m}}{m(r^2-s^2)}T(x_{i,m+k}^+),
\end{align*}
the desired equation.
\end{proof}

The next Proposition helps us to derive spectral-parameter $R$-matrices through finite-dimensional representations of $U_{r, s}(\widehat{\mathfrak{g}})$, which was firstly established in \cite{Chari 1995} for $U_q(\widehat{\mathfrak{g}})$.
\begin{prop}\label{prop auto z}
  For any $z\in \mathbb{C}^\times$, there exist a family of Hopf algebra automorphisms $\tau_z$ of $U_{r, s}(\widehat{\mathfrak{g}})$ defined by
$$\tau_z(x_{ik}^\pm)=z^k x_{ik}^\pm, \quad \tau_z(a_{i\ell})=z^{\ell}a_{i\ell},$$
and $\tau_z$ is the identity map on the other generators.
\end{prop}

Pulling back the vector representation $T$ by $\tau_z$ gives the following minimal affinization $T_z$:

\begin{lemm}\label{lemm evamod}
  Let $z\in \mathbb{C}^\times, V=V(z)$.
   The $($type $1)$ minimal affinization $$T_z: U_{r,s}\mathcal(\widehat {\frak{so}_{2n+1}})\longrightarrow End(V)$$
   can be defined as follows:
 Its restriction to $U_{r, s}(\mathfrak{so}_{2n+1})$ is defined as the same as those in  Lemma \ref{lemm vect rep}, and
\begin{align*}
&T_z(\omega_0)=\;\Bigl( s^2 E_{11}+r^{-2}E_{22}+(rs)^{-2}\sum_{i=3}^{n}E_{ii}+E_{n+1, n+1}\\
&\qquad\qquad+s^{-2}E_{1'1'}+r^2E_{2'2'}+(rs)^2\sum_{j=n'}^{3'}E_{jj}\Bigr),\\
&T_z(\omega_0')=\;\Bigl( r^2 E_{11}+s^{-2}E_{22}+(rs)^{-2}\sum_{i=3}^{n}E_{ii}+E_{n+1, n+1}\\
&\qquad\qquad+r^{-2}E_{1'1'}+s^2E_{2'2'}+(rs)^2\sum_{j=n'}^{3'}E_{jj}\Bigr),\\
&T_z(e_0)=\; -r^{-\frac{1}{2}}s^{-\frac{3}{2}}z\Bigl(E_{1',2}-(rs)E_{2',1}\Bigr),\\
&T_z(f_0)=\;-r^{-\frac{3}{2}}s^{-\frac{1}{2}}z^{-1}\Bigl(E_{2,1'}-(rs)E_{1,2'}\Bigr),\\
&T_z(\gamma)=\;T_z(\gamma')\;=1,
\end{align*}
as well as
$$T_z(D)(z^k v)=\;r_0^k v, \qquad T_z(D')(z^k v )=\;s_0^k v, \quad k\in \mathbb{Z},\; v\in V.$$
\end{lemm}

\begin{proof}
The assignments are based on the Drinfeld Isomorphism Theorem (See \cite{HuCMP 2008,Zhang phd 2007}).
It suffices to prove these are compatible with Definition \ref{type B def}.
Indeed, we have
$$T_z(\omega_0)=T(\gamma^{\prime-1}\omega_\theta^{-1}), \quad T_z(\omega'_0)=T(\gamma^{-1}\omega_\theta^{\prime-1}),$$
\noindent where $\theta$ is the highest positive root of $\mathfrak{so}_{2n+1}$.
And
$$T_z(e_0)=z\cdot T\Bigl(a\gamma^{\prime-1}x_\theta^-(1)\omega_\theta^{-1}\Bigr),\quad
T_z(f_0)=z^{-1}\cdot T\Bigl(a\gamma^{-1}\omega_\theta^{\prime-1}x_\theta^+(-1)\Bigr),$$
where $a=(rs)^{n-2}(r+s)^{-1}$,
and $x_\theta^+(-1), x_\theta^-(1)$ are the quantum affine root vectors as follows (\cite{HuCMP 2008, HuZhang2014, Zhang phd 2007}):
For $\alpha=\alpha_{i_1}+\cdots+\alpha_{i_n}:=\alpha_{i_1,i_2,\cdots,i_n}\in\Phi^+$, the non-trivial quantum affine root vectors $x_\alpha^\pm(k)$ of level $k$ is defined by:
$$x_\alpha^+(k):=[\cdots[x_{i_1}^+(k),x_{i_2}^+(0)]_{\langle\omega'_{i_1},\omega_{i_2}\rangle^{-1}},\cdots, x_{i_n}^+(0)]_{\langle\omega'_{\alpha_{i_1,\cdots,i_{n-1}}},\;\omega_{i_n}\rangle^{-1}},$$
$$x_\alpha^-(k):=[x_{i_n}^-(0),\cdots,[x_{i_2}^-(0),x_{i_1}^-(k)]_{\langle\omega'_{i_2},\omega_{i_1}\rangle}\cdots]_{\langle\omega'_{i_n},\;\omega_{\alpha_{i_1,\cdots,i_{n-1}}}\rangle}.$$
Using part (I) of Lemma \ref{lemm vectrep aff}, we can get the explicit formula of $T_z(e_0)$ and $T_z(f_0)$.
It is easy to verify that
$$\Bigl[T_z(e_0), T_z(f_0)\Bigr]=\frac{T_z(\omega_0)-T_z(\omega'_0)}{r^2-s^2},$$
\begin{align*}
  T_z(\omega_ie_0)=(r^2s^{-2})^{\delta_{0i}}T_z(e_0\omega_i),\qquad & T_z(\omega_if_0)=(r^{-2}s^2)^{\delta_{0i}}T_z(f_0\omega_i),  \\
  T_z(\omega'_ie_0)=(r^{-2}s^2)^{\delta_{0i}}T_z(e_0\omega'_i),\qquad & T_z(\omega'_if_0)=(r^2s^{-2})^{\delta_{0i}}T_z(f_0\omega'_i).
\end{align*}
Combining these verifications with the basic representation given in Lemma \ref{lemm vect rep}, we can complete the proof.
\end{proof}

\begin{prop}\label{prop intert}
  Using the explicit formula of $\widehat{R}(\frac{z}{w})$ in Theorem \ref{thm spectral Rz1} ,
we can get the intertwining operator:
$$R\Bigl(\frac{z}{w}\Bigr): V(z)\otimes V(w)\rightarrow V(w)\otimes V(z),$$
where $R(\frac{z}{w})=P\circ \widehat{R}(\frac{z}{w}).$
Specifically, it satisfies:
\begin{equation}\label{equ intert}
  R\Bigl(\frac{z}{w}\Bigr)(T_z\otimes T_w)(\Delta(x))=(T_w\otimes T_z)(\Delta(x))R\Bigl(\frac{z}{w}\Bigr), \quad x\in U_{r,s}\mathcal(\widehat {\frak{so}_{2n+1}}).
\end{equation}
\end{prop}

The explicit formula of intertwining operator in one-parameter case was firstly proposed by Jimbo in \cite{Jimbo 1986}.

\begin{proof}
 To verify the validity of this formula in the two-parameter case,
one first notices that this formula holds immediately for $x\in U_{r,s}\mathcal(\frak{so}_{2n+1})$,
as $S$ and $S^{-1}$ are braided homomorphisms (it should be noted that $R(\frac{z}{w})$ is the linear combination of $S$, $I$ and $S^{-1}$, see (\ref{Rz 1})).
Thus, it suffices to verify for $x=e_0$ (the case for $f_0, \omega_0$ and $\omega'_0$ can be proved similarly).
Specifically, we need to verify:
\begin{equation}\label{equ intert2}
  R\Bigl(\frac{z}{w}\Bigr)(T_z\otimes T_w)(\Delta(e_0))(v_i\otimes v_j)=(T_w\otimes T_z)(\Delta(e_0))R\Bigl(\frac{z}{w}\Bigr)(v_i\otimes v_j),
\end{equation}
where $v_i\in V(z)$ and $v_j\in V(\omega)$.

We conduct a case-by-case discussion based on the different values of $i$ and $j$.
From  Lemma \ref{lemm evamod}, we know that $T_z(e_0)=az\Bigl(E_{1',2}-rsE_{2',1}\Bigr)$, where $a=-r^{-\frac{1}{2}}s^{-\frac{3}{2}}$.

(I) We assume that $i, j\notin \{1, 2\}$. Then we know the left-hand side (LHS) is obviously equal to $0$.
For the right-hand side (RHS):

$\bullet$ If $i\neq j'$,  RHS is also 0.

$\bullet$ If \(i = j'\) (i.e., \(j \notin \{1, 2, 1', 2'\}\)):
\begin{align*}
  \text{RHS} =\;  & (T_w\otimes T_z)(\Delta(e_0))[ c_{1j}(\frac{z}{w})v_1\otimes v_{1'}+c_{2j}(\frac{z}{w}) v_2\otimes v_{2'} \\
   & + c_{1j'}(\frac{z}{w})v_{1'}\otimes v_1+ c_{2'j}(\frac{z}{w})v_{2'}\otimes v_2] \\
  =\; & -rs\,awc_{1j}(\frac{z}{w})v_{2'}\otimes v_{1'}+awc_{2j}(\frac{z}{w})v_{1'}\otimes v_{2'}+r^2az c_{2'j}(\frac{z}{w})v_{2'}\otimes v_{1'}\\
  &-rs^{-1}azc_{1'j}(\frac{z}{w})v_{1'}\otimes v_{2'}\\
  =\; & \;0,
\end{align*}
where $c_{ij}(\frac{z}{w})$ denotes the coefficient of $E_{ij'}\otimes E_{i'j}$  in $R(\frac{z}{w})$.

(II) We assume that $i, j\in \{1, 2\}$.
Without loss of generality, we may assume that $i=1, j=2$. Then
\begin{align*}
  \text{LHS} =\; & R\Bigl(\frac{z}{w}\Bigr) ( -rs\,az v_{2'}\otimes v_2+s^2 aw v_1\otimes v_{1'}) \\
  =\; & -rs\,az\bigl(\sum_{i=1}^{1'} c_{i2}(\frac{z}{w}) v_i\otimes v_{i'}\bigr) \\
   &+ s^2aw\bigl(\sum_{i=1}^{1'} c_{i1'}(\frac{z}{w}) v_i\otimes v_{i'}\bigr). \\
   =\; &\frac{r^2s^2\,aw(z-w)}{r^2z-s^2w}v_{1'}\otimes v_1-\frac{rs\,azw(r^2-s^2)}{r^2z-s^2w}v_{2'}\otimes v_2\\
   &-\frac{rs^3az(z-w)}{r^2z-s^2w}v_2\otimes v_{2'}+\frac{s^2az^2(r^2-s^2)}{r^2z-s^2w}v_1\otimes v_{1'}\\
   =\; &\text{RHS}.
\end{align*}
Here we remark that $-rs\, azc_{i2}(\frac{z}{w})+s^2awc_{i1'}(\frac{z}{w})=0$ if $i\notin \{1, 2, 1', 2'\}$.

(III) We assume that $i\in \{1, 2\}$ while $j\notin \{1, 2\}$.
Without loss of generality, we may assume that $i=1$. Then
$$
  \text{LHS} = R\Bigl(\frac{z}{w}\Bigr) (T_z(e_0)\otimes 1)(v_1\otimes v_j)
  =-rs\,az R\Bigl(\frac{z}{w}\Bigr) (v_{2'}\otimes v_j).
$$
Then we can calculate the value of the LHS based on the different values of $j$.
This calculation process is similar to the previous cases, so we omit it here.

(IV) We assume $j\in \{1, 2\}$ while $i\notin \{1, 2\}$.
This verification is similar to that in (III).
\end{proof}

\begin{remark}\label{rmk intert}
In Ge-Wu-Xue \cite{GeWX1991}, they provided two different affinizations when the braid group representation $S$ has three different eigenvalues, see (\ref{Rz 1}) and (\ref{Rz 2}).
In the previous section, we have used (\ref{Rz 1}) to give the $RLL$ realization of $U_{r, s}(\widehat{\mathfrak{so}_{2n+1}})$.

Using (\ref{Rz 2}), we can obtain another spectral parameter dependent $R$-matrix $\widehat{R}_{new}(z)$.
Indeed, the differences of matrix coefficients between $\widehat{R}_{new}(z)$ and $\widehat R(z)$ occur in those entries $E_{i'j'}\otimes E_{ij}$.
\end{remark}

\begin{prop}\label{prop spectR2}
 Another  spectral parameter dependent $\widehat{R}_{new}(z)$ is given by
\begin{align*}
\widehat{R}_{new}(z)=&\;\sum_{i \atop i\neq i'}E_{ii}\otimes E_{ii}+\frac{z-1}{r^2z-s^2}\Bigl\{\Bigl(\sum_{1\leq i\leq n \atop i+1\leq j\le n}
E_{jj}\otimes E_{ii}+\sum_{2\leq i\leq n \atop (i-1)'\leq j\leq 1'}E_{jj}\otimes E_{ii}\\
&+\sum_{1\leq i\leq n-1 \atop n'\leq j\leq (i+1)'}E_{ii}\otimes E_{jj}+\sum_{n'\leq i\leq 2'\atop i+1\leq j\leq 1'}E_{ii}\otimes E_{jj}\Bigr)+(rs)^2\Bigl(\sum_{1\leq i\leq n\atop i+1\leq j\leq n}E_{ii}\otimes E_{jj}\\
&+\sum_{n'\leq i\leq 2'\atop i+1\leq j\leq 1'}E_{jj}\otimes E_{ii}+\sum_{2\leq i\leq n\atop (i-1)'\leq j\leq 1'}E_{ii}\otimes E_{jj}+\sum_{1\leq i\leq n-1\atop n'\leq j\leq (i+1)'}E_{j j}\otimes E_{i i}\Bigr)\\
&+(rs)\Bigl(\sum_{i\atop i\neq i'} E_{n+1, n+1}\otimes E_{i,i}+\sum_{j\atop j\neq j'}E_{j, j}\otimes E_{n+1, n+1}\Bigr)\Bigr\}\\
&+\frac{r^2-s^2}{r^2z-s^2}\Bigl\{\sum_{i<j \atop i' \neq j}E_{ij}\otimes E_{ji}+z\sum_{i>j \atop i' \neq j }E_{ij}\otimes E_{ji}\Bigr\}\\
&+\frac{1}{(r^2s^{-2}z+(r^{-1}s)^{2n-1})(r^2z-s^2)}\sum_{i,j=1}^{2n+1}d_{ij}(z, 1)E_{i' j'}\otimes E_{i j},
\end{align*}
where $d_{ij}(z, 1)=\left\{
                      \begin{array}{ll}
                        (s^2{-}r^2)z\Bigl\{(z{-}1)(r^{-1}s)^{\rho_i-\rho_j-2}-\delta_{i, j'}[r^2s^{-2}z+(r^{-1}s)^{2n-1}]\Bigr\}, & \hbox{$i<j;$} \\
                        (s^2{-}r^2)\Bigl\{(1{-}z)(r^{-1}s)^{2n-1+\rho_i-\rho_j}-\delta_{i, j'}[r^2s^{-2}z+(r^{-1}s)^{2n-1}]\Bigr\}, & \hbox{$i>j;$} \\
                        r^2(z{-}1)[z+(r^{-1}s)^{2n-1}], & \hbox{$i=j\neq i';$} \\
                       (rs)(z{-}1)[r^2s^{-2}z+(r^{-1}s)^{2n-1}]+(r^2{-}s^2)z[r^2s^{-2}+(r^{-1}s)^{2n-1}], & \hbox{$i=j=i'.$}
                      \end{array}
                    \right.$
\end{prop}

\begin{prop}\label{prop nonintert}
We need to discard $\widehat{R}_{new}(z)$ because $R_{new}\Bigl(\frac{z}{w}\Bigr)=P\circ \widehat{R}_{new}(\frac{z}{w})$  does not satisfy the intertwining property. That is to say, we have
\begin{equation}\label{equ nonintert}
  R_{new}\Bigl(\frac{z}{w}\Bigr)(T_z\otimes T_w)(\Delta(x))(v_i\otimes v_j)\neq(T_w\otimes T_z)(\Delta(x))R_{new}\Bigl(\frac{z}{w}\Bigr)(v_i\otimes v_j),
\end{equation}
for some $x\in U_{r,s}\mathcal(\widehat {\frak{so}_{2n+1}})$, $v_i\in V(z)$ and $v_j\in V(\omega)$.
That is, $R_{new}(\frac{z}{w})$ is not a module homomorphism.
\end{prop}
\begin{proof}
To verify that the inequality (\ref{equ nonintert}) holds, it suffices to find specific values for which the left-hand side and the right-hand side are not equal.
For example, one can verify that
\begin{equation}\label{equ nonintert2}
  R_{new}\Bigl(\frac{z}{w}\Bigr)(T_z\otimes T_w)(\Delta(e_0))(v_{2'}\otimes v_2)\neq(T_w\otimes T_z)(\Delta(e_0))R_{new}\Bigl(\frac{z}{w}\Bigr)(v_{2'}\otimes v_2).
\end{equation}
From  Lemma \ref{lemm evamod}, we know $T_z(e_0)=az\Bigl(E_{1',2}-rsE_{2',1}\Bigr)$, where $a=-r^{-\frac{1}{2}}s^{-\frac{3}{2}}$.
Then we compute the left-hand side of (\ref{equ nonintert2}), which is equal to
\begin{align*}
   & R_{new}\Bigl(\frac{z}{w}\Bigr)T_z(\omega_0)\otimes T_\omega(e_0)(v_{2'}\otimes v_2) \\
  =\; & r^2 aw R_{new}\Bigl(\frac{z}{w}\Bigr) (v_{2'}\otimes v_2) \\
  =\; & r^2 aw \Bigl[ \frac{z-w}{r^2z-s^2w} v_{1'}\otimes v_{2'}+\frac{(r^2-s^2)z}{r^2z-s^2w}v_{2'}\otimes v_{1'}\Bigr].
\end{align*}
While the right-hand side of (\ref{equ nonintert2}) is equal to
\begin{align*}
   &(T_w\otimes T_z)(\Delta(e_0))\Bigl[ c'_{12}(\frac{z}{w})v_1\otimes v_{1'}+c'_{22}(\frac{z}{w})v_2\otimes v_{2'}+ c'_{2'2}(\frac{z}{w}) v_{2'}\otimes v_2+ c'_{1'2}(\frac{z}{w}) v_{1'}\otimes v_1\Bigr]\\
  =\; & c'_{12}(\frac{z}{w})T_w(e_0)v_1\otimes T_z(1)v_{1'}+c'_{22}(\frac{z}{w})T_w(e_0)v_2\otimes T_z(1){v_{2'}}+c'_{2'2}(\frac{z}{w}) T_w(\omega_0)v_{2'}\otimes T_z(e_0)v_2\\
   &+c'_{1'2}(\frac{z}{w})T_w(\omega_0)v_{1'}\otimes T_z(e_0)v_1 \\
   =\; &-rs\, aw c'_{12}(\frac{z}{w}) v_{2'}\otimes v_{1'}+ awc'_{22}(\frac{z}{w})v_{1'}\otimes v_{2'}+r^2az c'_{2'2}(\frac{z}{w}) v_{2'}\otimes v_{1'}-rs^{-1}az c'_{1'2}(\frac{z}{w})v_{1'}\otimes v_{2'},\\
   =\; & \frac{r^2aw(z-w)[r^{-2}s^2z+(r^{-1}s)^{2n-1}w]}{[r^2s^{-2}z+(r^{-1}s)^{2n-1}w](r^2z-s^2w)}v_{1'}\otimes v_{2'}+\frac{\ast}{[r^2s^{-2}z+(r^{-1}s)^{2n-1}w](r^2z-s^2w)}v_{2'}\otimes v_{1'},
\end{align*}
where $c'_{ij}(\frac{z}{w})$ denotes the coefficient of $E_{ij'}\otimes E_{i'j}$  in $R_{new}(\frac{z}{w})$,  and $\ast=r^2azw(r^2-s^2)\{z-w+r^{-2}s^2(z-w)+r^2s^{-2}z+w(r^{-1}s)^{2n-1}\}$.
After comparing the corresponding coefficients of two sides, we get the inequality (\ref{equ nonintert2}). Thus we get (\ref{equ nonintert}).
\end{proof}

Moreover, Proposition 1 of \cite{Jimbo 1986} tells us that the dimension of the solution space (\ref{equ intert}) is at most $1$.
That is to say, the intertwining operator $R$ is determined up to a scalar multiple.

\section{The $RLL$ realization and Ding-Frenkel Isomorphism Theorem for $U_{r, s}(\widehat{\mathfrak{so}_{2n+1}})$}\label{sec $RLL$}
In this section, we study the commutation relations between Gaussian generators and give the $RLL$ realization
of $U_{r, s}(\widehat{\mathfrak{so}_{2n+1}})$.
Consequently, we can establish the isomorphism between the $RLL$-realization and the Drinfeld presentation.

The Gaussian generators can be defined as:
\begin{equation}\label{equ Xi def}
X_i^+(z)=e_{i, i+1}^+(z_+)-e_{i, i+1}^-(z_-),\quad X_i^-(z)=f_{i+1,i }^+(z_-)-f_{i+1,i }^-(z_+)
\end{equation}
for $1\leq i\leq n$, as in Jing-Liu-Molev \cite{JingLM SIGMA 2020},
and $k_i^\pm(z)$ are defined directly from Proposition \ref{Prop Gauss decomp}.

The following theorem establishes the commutation relations satisfied by these Gaussian generators.
\begin{theorem}\label{thm relations}
The generators $k_i^\pm(z),\ X_j^\pm(z) \ (1\leq i\leq n+1,\ 1\leq j\leq n)$ of $\mathcal{U}(\widehat{R}_{r, s}(z))$ satisfy the following relations:
$$k_i^\pm(z)k_{\ell}^\pm(w)=k_{\ell}^\pm(w)k_i^\pm(z), \quad 1\leq i,\;\ell\leq n+1,$$
$$k_i^\pm(z)k_i^\mp(w)=k_i^\mp(w)k_i^\pm(z), \quad  i\neq n+1,$$
$$\frac{z_\pm-w_\mp}{r^2z_\pm-s^2w_\mp}k_i^\pm(z)k_{\ell}^\mp(w)=k_{\ell}^\mp(w)k_{i}^\pm(z)\frac{z_\mp-w_\pm}{r^2z_\mp-s^2w_\pm},\quad 1\leq i<\ell\leq n+1,$$
$$\frac{s^2z_\pm-r^2w_\mp}{r^2z_\pm-s^2w_\mp}\frac{rz_\pm-sw_\mp}{sz_\pm-rw_\mp}k_{n+1}^\pm(z)k_{n+1}^\mp(w)=\frac{s^2z_\mp-r^2w_\pm}{r^2z_\mp-s^2w_\pm}\frac{rz_\mp-sw_\pm}{sz_\mp-rw_\pm}k_{n+1}^\mp(w)k_{n+1}^\pm(z).$$

The relations involving $k_i^\pm(z)$ and $X_j^\pm(w)$ can be stated as:

$(1)$ If $i-j\leq -1\; (j\neq n)$, or $i-j\geq 2$, then $k_i^\pm(z)$ and $X_j^\pm(w)$ are commutative:
\begin{align*}
k^{\pm}_i(z)X^{+}_j(w)&=X^{+}_j(w)k^{\pm}_i(z),\\
k^{\pm}_i(z)X^{-}_j(w)&=X^{-}_j(w)k^{\pm}_i(z),
\end{align*}

$(2)$ If $1\leq i\leq n-1$, then $k_i^\pm(z)$ and $X_n^\pm(w)$ are quasi-commutative:
\begin{align*}
rsk^{\pm}_i(z)X^{+}_n(w)&=X^{+}_n(w)k^{\pm}_i(z),\\
k^{\pm}_i(z)X^{-}_n(w)&=rsX^{-}_n(w)k^{\pm}_i(z),
\end{align*}

$(3)$ For $1\leq i\leq n-1$, we have
\begin{align*}
k^{\pm}_{i}(z)X^{+}_i(w)&=\frac{z-w_\pm}{s^{-2}z-r^{-2}w_\pm}X^{+}_i(w)k^{\pm}_{i}(z),\\
k^{\pm}_{i}(z)X^{-}_i(w)&=\frac{s^{-2}z-r^{-2}w_\mp}{z-w_\mp}X^{+}_i(w)k^{\pm}_{i}(z), \\
k^{\pm}_{i+1}(z)X^{+}_i(w)&=\frac{z-w_\pm}{r^{-2}z-s^{-2}w_\pm}X^{+}_i(w)k^{\pm}_{i+1}(z),\\
k^{\pm}_{i+1}(z)X^{-}_i(w)&=\frac{r^{-2}z-s^{-2}w_\mp}{z-w_\mp}X^{-}_i(w)k^{\pm}_{i+1}(z),
\end{align*}

$(4)$ For $i=n,\;n+1$ and $j=n$, these relations hold$:$
\begin{align*}
&k^{\pm}_{n}(z)X^{+}_n(w)=\frac{z-w_{\pm}}{rs^{-1}z-r^{-1}sw_{\pm}}X^{+}_n(w)k^{\pm}_{n}(z),\\
&k^{\pm}_{n}(z)X^{-}_n(w)=\frac{rs^{-1}z-r^{-1}sw_{\mp}}{z-w_{\mp}}X^{-}_n(z)k^{\pm}_{n}(w),\\
k^{\pm}_{n+1}(z)&X^{+}_n(w)=\frac{rs\bigl(z-w_\pm\bigr)\bigl(rz-sw_\pm\bigr)}{\bigl(r^2z-s^2w_\pm\bigr)\bigl(sz-rw_\pm\bigr)}X^{+}_n(w)k^{\pm}_{n+1}(z),\\
k^{\pm}_{n+1}(z)&X^{-}_n(w)=\frac{\bigl(r^2z-s^2w_\mp\bigr)\bigl(sz-rw_\mp\bigr)}{rs\bigl(z-w_\mp\bigr)\bigl(rz-sw_\mp\bigr)}X^{-}_n(w)k^{\pm}_{n+1}(z).
\end{align*}

As for $X_i^\pm(z), X_j^\pm(w)$, their commutation relations can be established as follows:
\begin{align*}
X_i^\pm(z)X_j^\pm(w) &= X_j^\pm(w)X_i^\pm(z), \quad |i-j|\geq 2, \\
X_i^\pm(z)X_{i+1}^\pm(w) &= \Bigl(\frac{z-w}{s^{-2}z-r^{-2}w}\Bigr)^\pm X_{i+1}^\pm(w)X_i^\pm(z), \quad 1\leq i\leq n-1, \\
X_i^\pm(z)X_i^\pm(w) &= \Bigl(\frac{r^2z-s^2w}{s^2z-r^2w}\Bigr)^\pm X_i^\pm(w)X_i^\pm(z),  \quad 1\leq i\leq n-1, \\
X_n^\pm(z)X_n^\pm(w) &= \Bigl(\frac{rz-sw}{sz-rw}\Bigr)^\pm X_n^\pm(w)X_n^\pm(z), \\
\bigl[X_i^+(z), X_j^-(w)\bigr] &= 0, \quad i\neq j,
\end{align*}

\begin{align*}
\bigl[X_i^+(z), X_i^-(w)\bigr] &= \bigl(r^2 - s^2\bigr) \Bigl\{
  \delta\Bigl(\frac{z_-}{w_+}\Bigr) k_{i+1}^-(w_+) \bigl(k_i^-(w_+)\bigr)^{-1} \\
  &\quad - \delta\Bigl(\frac{z_+}{w_-}\Bigr) k_{i+1}^+(z_+) \bigl(k_i^+(z_+)\bigr)^{-1}
\Bigr\}, \quad 1\leq i\leq n-1, \\
\bigl[X_n^+(z), X_n^-(w)\bigr] &= \bigl(rs^{-1} - r^{-1}s\bigr) \Bigl\{
  \delta\Bigl(\frac{z_-}{w_+}\Bigr) k_{n+1}^-(w_+) \bigl(k_n^-(w_+)\bigr)^{-1} \\
  &\quad - \delta\Bigl(\frac{z_+}{w_-}\Bigr) k_{n+1}^+(z_+) \bigl(k_n^+(z_+)\bigr)^{-1}
\Bigr\}.
\end{align*}

We can also derive the $(r,s)$-Serre relations:

\noindent \ $~~~Sym_{z_1, z_2}\Bigl\{(r_is_i)^{\pm1}X_i^{\pm}(z_1)
X_i^{\pm}(z_2)X_j^{\pm}(w)-(r_i^{\pm1}{+}s_i^{\pm1})\,X_i^{\pm}(z_1)X_j^{\pm}(w)X_i^{\pm}(z_2)$ \\

$\quad +\, X_j^{\pm}(w)X_i^{\pm}(z_1)X_i^{\pm}(z_2)\Bigr\}=0, \quad\textit{for }
\ a_{ij}=-1 \ \textit{and}\  1\leqslant j < i \leqslant n;$

\medskip
\noindent \ $~~~Sym_{z_1, z_2}\Bigl\{X_i^{\pm}(z_1)
X_i^{\pm}(z_2)X_j^{\pm}(w)-(r_i^{\pm1}{+}s_i^{\pm1})\,X_i^{\pm}(z_1)X_j^{\pm}(w)X_i^{\pm}(z_2)$ \\

$\quad +\, (r_is_i)^{\pm1}X_j^{\pm}(w)X_i^{\pm}(z_1)X_i^{\pm}(z_2)\Bigr\}=0, \quad\textit{for }
\ a_{ij}=-1 \ \textit{and}\  1\leqslant i< j \leqslant n;$

\medskip
\noindent  \ $Sym_{z_1, z_2,z_3}\Bigl\{X^{\pm}_{n-1}(w)X^{\pm}_{n}(z_1)X^{\pm}_{n}(z_2)X^{\pm}_{n}(z_3)-(r^{\pm2}+s^{\pm2}+r^\pm s^\pm)X^{\pm}_{n}(z_1)X^{\pm}_{n-1}(w)X^{\pm}_{n}(z_2)X^{\pm}_{n}(z_3)$\\

$\hskip2.4cm+\,(rs)^\pm(r^{\pm2}+s^{\pm2}+r^\pm s^\pm)X^{\pm}_{n}(z_1)X^{\pm}_{n}(z_2)X^{\pm}_{n-1}(w)X^{\pm}_{n}(z_3)$\\

$\hskip3.4cm-\,(rs)^{\pm3}X^{\pm}_{n}(z_1)X^{\pm}_{n}(z_2)X^{\pm}_{n}(z_3)X^{\pm}_{n-1}(w)\Bigr\}=0,$

\vspace{1em}
\noindent where $(a_{ij})$ is the Cartan matrix of type $B$.
\end{theorem}

  We first verify the theorem for $n=3$.

\subsection{$B_3^{(1)}$ case}\label{ssec b3}
Firstly, we write down $L^\pm(z)$ and $L^\pm(z)^{-1}$ by the Gauss decomposition:
$$L^\pm(z)=\left(
             \begin{array}{ccc}
               k_1^\pm(z) & k_1^\pm(z)e_{12}^\pm(z) & \cdots \\
               f_{21}^\pm(z)k_1^\pm(z) & \vdots & \vdots \\
               \vdots & \vdots & \vdots \\
             \end{array}
           \right),$$
and
$$
L^\pm(z)^{-1}=\left(
                 \begin{array}{ccc}
                   \cdots & \cdots & \cdots \\
                   \cdots& \cdots & -e_{67}^\pm(z)k_{7}^\pm(z)^{-1} \\
                   \cdots & -k_{7}^\pm(z)^{-1}f_{76}^\pm(z) & k_{7}^\pm(z)^{-1} \\
                 \end{array}
               \right).$$

Using (\ref{equ 5.2}) and (\ref{equ 5.4}),
we complete the verification by the following lemmas.

\begin{lemm}\label{lemm kikj} One has
  \begin{align*}
k_i^\pm(w)k_j^\pm(z)&=k_j^\pm(z)k_i^\pm(w), \quad 1\leq i,\; j\leq 4, \;(i, j)\neq (4, 4), \\
k_i^\pm(w)k_i^\mp(z)&=k_i^\mp(z)k_i^\pm(w), \quad 1\leq i< 4,  \\
\frac{z_\mp-w_\pm}{r^2z_\mp-s^2w_\pm}k_j^\mp(w)k_{i}^\pm(z)&=\frac{z_\pm-w_\mp}{r^2z_\pm-s^2w_\mp}k_i^\pm(z)k_j^\mp(w),
\quad 1\leq i<j\leq 4.
  \end{align*}
\end{lemm}
\begin{proof}
We only prove the case when $i=1, j=2$, and the other cases can be calculated in the same way.

From (\ref{equ 5.3}), and  $M_{11}=M_{11}'$, we get the following equation:
$$\left(
    \begin{array}{cc}
      a_{11}(\frac{z_\pm}{w_\mp})\ell_{11}^\pm(z) & 0 \\
      b_{12}(\frac{z_\pm}{w_\mp})\ell_{21}^\pm(z) & a_{12}(\frac{z_\pm}{w_\mp})\ell_{11}^\pm(z) \\
    \end{array}
  \right)
\left(
  \begin{array}{cc}
    \ell_{11}^\mp(w) & \ell_{12}^\mp(w) \\
    \ell_{21}^\mp(w) & \ell_{22}^\mp(w) \\
  \end{array}
\right)
$$
$$=
\left(
  \begin{array}{cc}
    \ell_{11}^\mp(w) & \ell_{12}^\mp(w) \\
    \ell_{21}^\mp(w) & \ell_{22}^\mp(w) \\
  \end{array}
\right)
\left(
  \begin{array}{cc}
    a_{11}(\frac{z_\mp}{w_\pm})\ell_{11}^\pm(z) & b_{21}(\frac{z_\mp}{w_\pm})\ell_{12}^\pm(z) \\
   0  & a_{12}(\frac{z_\mp}{w_\pm})\ell_{11}^\pm(z) \\
  \end{array}
\right).
$$
We can thus derive
$$a_{12}\Bigl(\frac{z_\pm}{w_\mp}\Bigr)k_1^\pm(z)k_2^\mp(w)=a_{12}\Bigl(\frac{z_\mp}{w_\pm}\Bigr)k_2^\mp(w)k_1^\pm(z).$$
From (\ref{Rz 1}), the similar process leads to
$$a_{12}\Bigl(\frac{z_\pm}{w_\mp}\Bigr)k_1^\pm(z)k_2^\pm(w)=a_{12}\Bigl(\frac{z_\pm}{w_\mp}\Bigr)k_2^\pm(w)k_1^\pm(z).$$

Bringing the coefficients coming from the spectral parameter-dependent $R(z)$ into them, we can finally obtain the desired equations.
\end{proof}

\begin{lemm}\label{lemm B3 k1X3} One has
  \begin{gather}
    rs\,k_1^\pm(z)X_3^+(w)=X_3^+(w)k_1^\pm(z), \nonumber\\
    k_1^\pm(z)X_3^-(w)=rs\,X_3^-(w)k_1^\pm(z),\nonumber\\
    k_1^\pm(z)X_2^+(w)=X_2^+(w)k_1^\pm(z), \nonumber\\
    k_1^\pm(z)X_2^-(w)=X_2^-(w)k_1^\pm(z),\nonumber\\
    rs\,k_2^\pm(z)X_3^+(w)=X_3^+(w)k_2^\pm(z), \nonumber\\
    k_2^\pm(z)X_3^-(w)=rs\,X_3^-(w)k_2^\pm(z). \nonumber
  \end{gather}
\end{lemm}
\begin{proof}
   We only prove the first equation since the others can be obtained by the same token.
   Taking the equation $M_{11}=M_{11}'$, we get
$$
a_{13}\Bigl(\frac{z}{w}\Bigr)k_1^\pm(z)k_3^\pm(w)e_{34}^\pm(w)=a_{14}\Bigl(\frac{z}{w}\Bigr)k_3^\pm(w)e_{34}^\pm(w)k_1^\pm(z).$$

 Using the invertibility of $k_3^\pm(w)$, and the fact that
$$k_1^\pm(z)k_3^\pm(w)=k_3^\pm(w)k_1^\pm(z),$$
we have
$$rsk_1^\pm(z)e_{34}^\pm(w)=e_{34}^\pm(w)k_1^\pm(z).$$

 Similarly, we conclude that
$$a_{13}\Bigl(\frac{z_\pm}{w_\mp}\Bigr)k_1^\pm(z)k_3^\mp(w)e_{34}^\mp(w)=a_{14}\Bigl(\frac{z_\mp}{w_\pm}\Bigr)k_3^\mp(w)e_{34}^\mp(w)k_1^\pm(z).$$

 Again using the invertibility of $k_3^\mp(w)$, and
$$k_1^\pm(z)k_3^\mp(w)\frac{z_\pm-w_\mp}{r^2z_\pm-s^2w_\mp}=k_3^\mp(w)k_1^\pm(z)\frac{z_\mp-w_\pm}{r^2z_\mp-s^2w_\pm},$$
we also have
$$rsk_1^\pm(z)e_{34}^\mp(w)=e_{34}^\mp(w)k_1^\pm(z),$$
so that $rsk_1^\pm(z)X_3^+(w)=X_3^+(w)k_1^\pm(z)$.
\end{proof}

\begin{lemm}\label{lemm B3 k4X1} One has
\begin{gather}\label{}
  k_3^\pm(w)X_1^\pm(z)=X_1^\pm(z)k_3^\pm(w), \nonumber \\
  k_3^\pm(w)X_1^\mp(z)=X_1^\mp(z)k_3^\pm(w), \nonumber \\
  k_4^\pm(w)X_i^\pm(z)=X_i^\pm(z)k_4^\pm(w), \nonumber \\
  k_4^\pm(w)X_i^\mp(z)=X_i^\mp(z)k_4^\pm(w), \nonumber
\end{gather}
where $1\leq i\leq 2$.
\end{lemm}
\begin{proof}
This Lemma can be proved similarly.
\end{proof}

\begin{lemm} \label{lemm b3 x1x3}
One has
  \begin{gather}\label{}
    X_1^\pm(z)X_3^\pm(w)=X_3^\pm(w)X_1^\pm(z), \nonumber\\
    X_1^\mp(z)X_3^\pm(w)=X_3^\pm(w)X_1^\mp(z), \nonumber
  \end{gather}
\end{lemm}
\begin{proof}
  From $M_{12}=M_{12}'$, we have
  $$a_{13}\Bigl(\frac{z_\pm}{w_\mp}\Bigr)\ell_{12}^\pm(z)k_3^\mp(w)e_{34}^\mp(w)=a_{24}\Bigl(\frac{z_\mp}{w_\pm}\Bigr)k_3^\mp(w)e_{34}^\mp(w)\ell_{12}^\pm(z).$$

  Noticing that
  $$a_{13}\Bigl(\frac{z_\pm}{w_\mp}\Bigr)\ell_{12}^\pm(z)k_3^\mp(w)=a_{23}\Bigl(\frac{z_\mp}{w_\pm}\Bigr)k_3^\mp(w)\ell_{12}^\pm(z),$$
  $$a_{13}\Bigl(\frac{z_\pm}{w_\mp}\Bigr)k_1^\pm(z)k_3^\mp(w)e_{34}^\mp(w)=a_{14}\Bigl(\frac{z_\mp}{w_\pm}\Bigr)k_3^\mp(w)e_{34}^\mp(w)k_1^\pm(z),$$
  and
  $$\frac{z_\mp-w_\pm}{r^2z_\mp-s^2w_\pm}k_3^\mp(w)k_{1}^\pm(z)=\frac{z_\pm-w_\mp}{r^2z_\pm-s^2w_\mp}k_1^\pm(z)k_{3}^\mp(w),$$
  we derive that
  \begin{equation*}\label{}
    e_{12}^\pm(z)e_{34}^\mp(w)=e_{34}^\mp(w)e_{12}^\pm(z).
  \end{equation*}

 Similarly we have
    \begin{equation*}\label{}
    e_{12}^\pm(z)e_{34}^\pm(w)=e_{34}^\pm(w)e_{12}^\pm(z).
  \end{equation*}

  Thus, we arrive at
  \begin{equation*}\label{}
    X_1^+(z)X_3^+(w)=X_3^+(w)X_1^+(z).
  \end{equation*}

 The other cases can be proved similarly.
\end{proof}

\begin{lemm} \label{lemm B3 kiXi}
One has
 \begin{gather*}\label{}
  k_i^\pm(z) X_i^+(w)=\frac{z-w_\pm}{s^{-2}z-r^{-2}w_\pm}X_i^+(w)k_i^\pm(z), \\
  k_i^\pm(z)X_i^-(w)=\frac{s^{-2}z-r^{-2}w_\mp}{z-w_\mp}X_i^-(w)k_i^\pm(z),
\end{gather*}
 where $i=1, 2$.
\end{lemm}
\begin{proof}
   We only consider the case $i=2$.
   Taking the equation $M_{22}=M'_{22}$, we can get
   \begin{equation}\label{M22 23}
b_{32}\Bigl(\frac{z_\mp}{w_\pm}\Bigr)\ell_{22}^\mp(w)\ell_{23}^\pm(z)+a_{23}\Bigl(\frac{z_\mp}{w_\pm}\Bigr)\ell_{23}^\mp(w)\ell_{22}^{\pm}(z)
=\ell_{22}^\pm(z)\ell_{23}^{\mp}(w),\\
\end{equation}
\begin{equation}\label{M22 13}
b_{32}\Bigl(\frac{z_\mp}{w_\pm}\Bigr)\ell_{12}^\mp(w)\ell_{23}^\pm(z)+a_{23}\Bigl(\frac{z_\mp}{w_\pm}\Bigr)\ell_{13}^\mp(w)\ell_{22}^\pm(z)
=a_{21}\Bigl(\frac{z_\pm}{w_\mp}\Bigr)\ell_{22}^\pm(z)\ell_{13}^\mp(w)+b_{21}\Bigl(\frac{z_\pm}{w_\mp}\Bigr)\ell_{12}^\pm(z)\ell_{23}^\mp(w).
\end{equation}

Then, by subtracting $f_{21}^\mp\left(\frac{z}{w}\right) \cdot (\ref{M22 13})$ from $(\ref{M22 23})$, we have
\begin{equation}
\begin{aligned}\label{k2 X2 1}
&b_{32}\Bigl(\frac{z_\mp}{w_\pm}\Bigr)k_2^\mp(w)\ell_{23}^\pm(z)+a_{23}\Bigl(\frac{z_\mp}{w_\pm}\Bigr)k_2^\mp(w)e_{23}^\mp(w)\ell_{22}^\pm(z)\\
=\;&\ell_{22}^\pm(z)\ell_{23}^\mp(w)-a_{21}\Bigl(\frac{z_\pm}{w_\mp}\Bigr)f_{21}^\mp(w)\ell_{22}^\pm(z)\ell_{13}^\mp(w)-b_{21}\Bigl(\frac{z_\pm}{w_\mp}\Bigr)f_{21}^\mp(w)\ell_{12}^\pm(z)\ell_{23}^\mp(w)\\
=\;&\ell_{22}^\pm(z)k_2^\mp(w)e_{23}^\mp(w)+b_{12}\Bigl(\frac{z_\mp}{w_\pm}\Bigr)k_2^\mp(w)\ell_{21}^\pm(z)e_{13}^\mp(w)-b_{21}\Bigl(\frac{z_\pm}{w_\mp}\Bigr)f_{21}^\mp(w)\ell_{12}^\pm(z)k_2^\mp(w)e_{23}^\mp(w).
\end{aligned}
\end{equation}

 Similarly, taking the equation $M_{12}=M'_{12}$,
and then we have
\begin{equation*}
\begin{aligned}\label{}
&b_{32}\Bigl(\frac{z_\mp}{w_\pm}\Bigr)k_2^\mp(w)\ell_{13}^\pm(z)+a_{23}\Bigl(\frac{z_\mp}{w_\pm}\Bigr)k_2^\mp(w)e_{23}^\mp(w)\ell_{12}^\pm(z)\\
=\;&a_{12}\Bigl(\frac{z_\pm}{w_\mp}\Bigr)\ell_{12}^\pm(z)\ell_{23}^\mp(w)+b_{12}\Bigl(\frac{z_\pm}{w_\mp}\Bigr)\ell_{22}^\pm(z)\ell_{13}^\mp(w)-f_{21}^\mp(w)\ell_{12}^\pm(z)\ell_{13}^\mp(w)\\
=\;&a_{12}\Bigl(\frac{z_\pm}{w_\mp}\Bigr)\ell_{12}^\pm(z)k_2^\mp(w)e_{23}^\mp(w)+b_{12}\Bigl(\frac{z_\pm}{w_\mp}\Bigr)k_2^\mp(w)k_1^\pm(z)e_{13}^\mp(w).
\end{aligned}
\end{equation*}

 Then we have
\begin{equation}\label{k2 X2 2}
\begin{aligned}
&b_{32}\Bigl(\frac{z_\mp}{w_\pm}\Bigr)k_2^\mp(w)f_{21}^\pm(z)\ell_{13}^\pm(z)+a_{23}\Bigl(\frac{z_\mp}{w_\pm}\Bigr)k_2^\mp(w)e_{23}^\mp(w)f_{21}^\pm(z)\ell_{12}^\pm(z)\\
=\;&a_{12}\Bigl(\frac{z_\pm}{w_\pm}\Bigr)k_2^\mp(w)f_{21}^\pm(z)k_2^\mp(w)^{-1}l_{12}^\pm(z)k_2^\mp(w)e_{23}^\mp(w)
+b_{12}\Bigl(\frac{z_\pm}{w_\mp}\Bigr)k_2^\mp(w)f_{21}^\pm(z)k_1^\pm(z)e_{13}^\mp(w)\\
=\;&f_{21}^\pm(z)l_{12}^\pm(z)k_2^\mp(w)e_{23}^\mp(w)-b_{21}\Bigl(\frac{z_\pm}{w_\mp}\Bigr)f_{21}^\mp(w)l_{12}^\pm(z)k_2^\mp(w)e_{23}^\mp(w).
\end{aligned}
\end{equation}

 Then, by subtracting (\ref{k2 X2 2}) from (\ref{k2 X2 1}), we have
\begin{equation*}\label{}
b_{32}\Big(\frac{z_\mp}{w_\pm}\Big)k_2^\mp(w)k_2^\pm(z)e_{23}^\pm(z)+a_{23}\Big(\frac{z_\mp}{w_\pm}\Big)k_2^\mp(w)e_{23}^\mp(w)k_2^\pm(z)=k_2^\pm(z)k_2^\mp(w)e_{23}^\mp(w).
\end{equation*}

 Using the invertibility of $k_2^\mp(w)$, we have
\begin{equation}\label{k2 X2 4}
b_{32}\Big(\frac{z}{w_\pm}\Big)k_2^\pm(z)e_{23}^\pm(z)+a_{23}\Big(\frac{z}{w_\pm}\Big)e_{23}^\mp(w_\mp)k_2^\pm(z)=k_2^\pm(z)e_{23}^\mp(w_\mp).
\end{equation}

 Similarly, we have
\begin{equation}\label{k2 X2 3}
b_{32}\Big(\frac{z}{w_\pm}\Big)k_2^\pm(z)e_{23}^\pm(z)+a_{23}\Big(\frac{z}{w_\pm}\Big)e_{23}^\pm(w_\pm)k_2^\pm(z)=k_2^\pm(z)e_{23}^\pm(w_\pm).
\end{equation}
After subtracting (\ref{k2 X2 4}) from (\ref{k2 X2 3}), we complete our proof.
\end{proof}

By the same token, one can also prove
\begin{lemm} \label{lemm B3 ki+1xi}
One has
\begin{align*}
   &k^{\pm}_{i+1}(z)X^{+}_i(w)=\frac{z-w_\pm}{r^{-2}z-s^{-2}w_\pm}X^{+}_i(w)k^{\pm}_{i+1}(z),  \\
   &k^{\pm}_{i+1}(z)X^{-}_i(w)=\frac{z-w_\mp}{r^{-2}z-s^{-2}w_\mp}X^{-}_i(w)k^{\pm}_{i+1}(z),
\end{align*}
  where $i=1, 2.$
\end{lemm}

\begin{lemm}\label{lemma e23e34} One has
\begin{align}
e_{i,i+1}^\pm(z)e_{i+1,i+2}^\pm(w)=\;& \frac{z-w}{s^{-2}z-r^{-2}w}e_{i+1,i+2}^\pm(w)e_{i,i+1}^\pm(z)+\frac{\bigl(r^2-s^2\bigr)z}{r^2z-s^2w}e_{i,i+2}^\pm(z) \label{e24+} \\
&+\frac{\bigl(r^2-s^2)w}{r^2z-s^2w}e_{i,i+1}^\pm(w)e_{i+1,i+2}^\pm(w)-\frac{\bigl(r^2-s^2\bigr)w}{r^2z-s^2w}e_{i,i+2}^\pm(w),\nonumber\\
f_{i+2,i+1}^\pm(w)f_{i+1,i}^\pm(z)=\; &\frac{z-w}{s^{-2}z-r^{-2}w}f_{i+1,i}^\pm(z)f_{i+2,i+1}^\pm(w)+\frac{\bigl(r^2-s^2\bigr)w}{r^2z-s^2w}f_{i+2,i}^\pm(z) \label{f42+}\\
&+\frac{\bigl(r^2-s^2)z}{r^2z-s^2w}f_{i+2,i+1}^\pm(w)f_{i+1,i}^\pm(w)-\frac{\bigl(r^2-s^2\bigr)z}{r^2z-s^2w}f_{i+2,i}^\pm(w), \nonumber\\
e_{i,i+1}^\pm(z)e_{i+1,i+2}^\mp(w)=\;&\frac{z_\mp-w_\pm}{s^{-2}z_\mp-r^{-2}w_\pm}e_{i+1,i+2}^\mp(w)e_{i,i+1}^\pm(z)+\frac{\bigl(r^2-s^2\bigr)z_\mp}{r^2z_\mp-s^2w_\pm}e_{i,i+2}^\pm(z)\label{e24-}\\
&+\frac{\bigl(r^2-s^2\bigr)w_\pm}{r^2z_\mp-s^2w_\pm}e_{i,i+1}^\mp(w)e_{i+1,i+2}^\mp(w)-\frac{\bigl(r^2-s^2\bigr)w_\pm}{r^2z_\mp-s^2w_\pm} e_{i,i+2}^\mp(w),\nonumber\\
f_{i+2,i+1}^\mp(w)f_{i+1,i}^\pm(z)=\; &\frac{z_\pm-w_\mp}{s^{-2}z_\pm-r^{-2}w_\mp}f_{i+1,i}^\pm(z)f_{i+2,i+1}^\mp(w)+\frac{\bigl(r^2-s^2\bigr)w_\mp}{r^2z_\pm-s^2w_\mp}f_{i+2,i}^\pm(z) \label{f42-}\\
&+\frac{\bigl(r^2-s^2)z_\pm}{r^2z_\pm-s^2w_\mp}f_{i+2,i+1}^\mp(w)f_{i+1,i}^\mp(w)-\frac{\bigl(r^2-s^2\bigr)z_\pm}{r^2z_\pm-s^2w_\mp}f_{i+2,i}^\mp(w), \nonumber
\end{align}
where $i=1,2$.
\end{lemm}
\begin{proof}
Here we only prove the first equation for $i=2$ since the others can be proved similarly.
 $M_{23}=M'_{23}$ leads to
   \begin{equation}\label{M23 34}
b_{43}\Bigl(\frac{z}{w}\Bigr)\ell_{33}^\pm(w)\ell_{24}^\pm(z)+a_{34}\Bigl(\frac{z}{w}\Bigr)\ell_{34}^\pm(w)\ell_{23}^\pm(z)=b_{23}\Bigl(\frac{z}{w}\Bigr)\ell_{33}^\pm(z)\ell_{24}^\pm(w)+a_{23}\Bigl(\frac{z}{w}\Bigr)\ell_{23}^\pm(z)\ell_{34}^\pm(w),\\
\end{equation}
\begin{equation}\label{M23 24}
b_{43}\Bigl(\frac{z}{w}\Bigr)\ell_{23}^\pm(w)\ell_{24}^\pm(z)+a_{34}\Bigl(\frac{z}{w}\Bigr)\ell_{24}^\pm(w)\ell_{23}^\pm(z)=\ell_{23}^\pm(z)\ell_{24}^\pm(w),
\end{equation}
\begin{equation}\label{M23 14}
b_{43}\Bigl(\frac{z}{w}\Bigr)\ell_{13}^\pm(w)\ell_{24}^\pm(z)+a_{34}\Bigl(\frac{z}{w}\Bigr)\ell_{14}^\pm(w)\ell_{23}^\pm(z)=b_{21}\Bigl(\frac{z}{w}\Bigr)\ell_{13}^\pm(z)\ell_{24}^\pm(w)+a_{21}\Bigl(\frac{z}{w}\Bigr)\ell_{23}^\pm(z)\ell_{14}^\pm(w).
\end{equation}

By subtracting $f_{31}^\pm(w)\cdot(\ref{M23 14})$ from (\ref{M23 34}), we conclude that
\begin{equation}
\begin{aligned}\label{X3 X2 1}
&b_{43}\Big(\frac{z}{w}\Big)\Bigl\{f_{32}^\pm(w)k_2^\pm(w)e_{23}^\pm(w)+k_3^\pm(w)\Bigr\}\ell_{24}^\pm(z)
+a_{34}\Big(\frac{z}{w}\Big)\Bigl\{f_{32}^\pm(w)k_2^\pm(w)e_{24}^\pm(w)+k_3^\pm(w)e_{34}^\pm(w)\Bigr\}\ell_{23}^\pm(z)\\
&=b_{23}\Big(\frac{z}{w}\Big)\ell_{33}^\pm(z)k_2^\pm(w)e_{24}^\pm(w)-b_{21}\Big(\frac{z}{w}\Big)f_{31}^\pm(w)\ell_{13}^\pm(z)k_2^\pm(w)e_{24}^\pm(w)
+a_{23}\Big(\frac{z}{w}\Big)\ell_{23}^\pm(z)\Bigl\{f_{32}^\pm k_2^\pm e_{24}^\pm(w)\\
&\hspace{12em}+k_3^\pm(w)e_{34}^\pm(w)\Bigr\}+b_{13}\Big(\frac{z}{w}\Big)\Bigl\{f_{32}^\pm(w)k_2^\pm(w)e_{24}^\pm(w)
+k_3^\pm(w)\Bigr\}\ell_{21}^\pm(z)e_{14}^\pm(w).
\end{aligned}
\end{equation}

 Also, by subtracting $f_{32}^\pm(w)f_{21}^\pm(w)\cdot(\ref{M23 14})$ from $f_{32}^\pm(w)\cdot(\ref{M23 24})$, we have
\begin{equation}\label{X3 X2 2}
\begin{aligned}
&b_{43}\Bigl(\frac{z}{w}\Bigr)f_{32}^\pm(w)k_2^\pm(w)e_{23}^\pm(w)\ell_{24}^\pm(z)
+a_{34}\Bigl(\frac{z}{w}\Bigr)f_{32}^\pm(w)k_2^\pm(w)e_{24}^\pm(w)\ell_{23}^\pm(z)\\
&=f_{32}^\pm(w)\ell_{23}^\pm(z)k_2^\pm(w)e_{24}^\pm(w)-b_{21}\Bigl(\frac{z}{w}\Bigr)f_{32}^\pm(w)f_{21}^\pm(w)\ell_{13}^\pm(z)k_2^\pm(w)e_{24}^\pm(w)\\
&\hspace{12em}\;+b_{13}\Bigl(\frac{z}{w}\Bigr)f_{32}^\pm(w)k_2^\pm(w)e_{23}^\pm(w)\ell_{21}^\pm(z)e_{14}^\pm(w)\\
&=f_{32}^\pm(w)k_2^\pm(z)e_{23}^\pm(z)k_2^\pm(w)e_{24}^\pm(w)+b_{12}\Bigl(\frac{z}{w}\Bigr)f_{31}^\pm(z)\ell_{13}^\pm(z)k_2^\pm(w)e_{24}^\pm(w)\\
&\quad\;+a_{13}\Bigl(\frac{z}{w}\Bigr)f_{21}^\pm(z)\ell_{13}^\pm(z)f_{32}^\pm(w)k_2^\pm(w)e_{24}^\pm(w)
+b_{13}\Bigl(\frac{z}{w}\Bigr)f_{21}^\pm(z)k_1^\pm(z)k_3^\pm(w)\\
&\hspace{5em}\cdot\Bigl(e_{12}^\pm(w)-e_{12}^\pm(z)\Bigr)e_{24}^\pm(w)-b_{21}\Bigl(\frac{z}{w}\Bigr)f_{31}^\pm(w)\ell_{13}^\pm(z)k_2^\pm(w)e_{24}^\pm(w)\\
&\hspace{13.5em}+b_{13}\Bigl(\frac{z}{w}\Bigr)f_{32}^\pm(w)k_2^\pm(w)e_{23}^\pm(w)\ell_{21}^\pm(z)e_{14}^\pm(w).
\end{aligned}
\end{equation}

 And $M_{13}=M'_{13}$ yields
   \begin{equation}\label{M13 34}
b_{43}\Bigl(\frac{z}{w}\Bigr)\ell_{33}^\pm(w)\ell_{14}^\pm(z)+a_{34}\Bigl(\frac{z}{w}\Bigr)\ell_{34}^\pm(w)\ell_{13}^\pm(z)=b_{13}\Bigl(\frac{z}{w}\Bigr)\ell_{33}^\pm(z)\ell_{14}^\pm(w)+a_{13}\Bigl(\frac{z}{w}\Bigr)\ell_{13}^\pm(z)\ell_{34}^\pm(w),\\
\end{equation}
\begin{equation}\label{M13 24}
b_{43}\Bigl(\frac{z}{w}\Bigr)\ell_{23}^\pm(w)\ell_{14}^\pm(z)+a_{34}\Bigl(\frac{z}{w}\Bigr)\ell_{24}^\pm(w)\ell_{13}^\pm(z)=b_{12}\Bigl(\frac{z}{w}\Bigr)\ell_{23}^\pm(z)\ell_{14}^\pm(w)+a_{21}\Bigl(\frac{z}{w}\Bigr)\ell_{13}^\pm(z)\ell_{24}^\pm(w),
\end{equation}
\begin{equation}\label{M13 14}
b_{43}\Bigl(\frac{z}{w}\Bigr)\ell_{13}^\pm(w)\ell_{14}^\pm(z)+a_{34}\Bigl(\frac{z}{w}\Bigr)\ell_{14}^\pm(w)\ell_{13}^\pm(z)=\ell_{13}^\pm(z)\ell_{14}^\pm(w).
\end{equation}
By subtracting $f_{31}^\pm(w)\cdot(\ref{M13 14})$ from $(\ref{M13 34})$, we get
\begin{equation}
\begin{aligned}\label{X3 X2 3}
&b_{43}\Big(\frac{z}{w}\Big)\Bigl\{f_{32}^\pm(w)k_2^\pm(w)e_{23}^\pm(w)+k_3^\pm(w)\Bigr\}\ell_{14}^\pm(z)
+a_{34}\Big(\frac{z}{w}\Big)\Bigl\{f_{32}^\pm(w)k_2^\pm(w)e_{24}^\pm(w)+k_3^\pm(w)e_{34}^\pm(w)\Bigr\}\ell_{13}^\pm(z)\\
&\hspace{19em}=a_{13}\Big(\frac{z}{w}\Big)\ell_{13}^\pm(z)\Bigl\{f_{32}^\pm(w)k_2^\pm(w)e_{24}^\pm(w)+k_3^\pm(w)e_{34}^\pm(w)\Bigr\}\\
&\hspace{19em}+b_{13}\Big(\frac{z}{w}\Big)\Bigl\{f_{32}^\pm(w)k_2^\pm(w)e_{24}^\pm(w)+k_3^\pm(w)\Bigr\}k_{1}^\pm(z)e_{14}^\pm(w).
\end{aligned}
\end{equation}

 Similarly, it follows that
\begin{equation}\label{X3 X2 4}
\begin{aligned}
&b_{43}\Bigl(\frac{z}{w}\Bigr)f_{32}^\pm(w)k_2^\pm(w)e_{23}^\pm(w)\ell_{14}^\pm(z)
+a_{34}\Bigl(\frac{z}{w}\Bigr)f_{32}^\pm(w)k_2^\pm(w)e_{24}^\pm(w)\ell_{13}^\pm(z)\\
&=a_{13}\Bigl(\frac{z}{w}\Bigr)\ell_{13}^\pm(z)f_{32}^\pm(w)k_2^\pm(w)e_{24}^\pm(w)
+b_{13}\Bigl(\frac{z}{w}\Bigr)k_1^\pm(z)k_3^\pm(w)\Bigl(e_{12}^\pm(w)-e_{12}^\pm(z)\Bigr)\\
&\hspace{11.5em}\cdot e_{24}^\pm(w)+b_{13}\Bigl(\frac{z}{w}\Bigr)f_{32}^\pm(w)k_2^\pm(w)e_{23}^\pm(w)k_{1}^\pm(z)e_{14}^\pm(w).
\end{aligned}
\end{equation}

 Then, subtracting $(\ref{X3 X2 2})$ and $f_{21}^\pm(z)\Bigl\{(\ref{X3 X2 3})-(\ref{X3 X2 4})\Bigr\}$ from (\ref{X3 X2 1}),
we arrive at
\begin{equation*}\label{}
\begin{aligned}
&a_{23}\Bigl(\frac{z}{w}\Bigr)k_2^\pm(z)k_3^\pm(w)e_{34}^\pm(w)e_{23}^\pm(z)\\
&=a_{23}\Bigl(\frac{z}{w}\Bigr)k_2^\pm(z)e_{23}^\pm(z)k_3^\pm(w)e_{34}^\pm(w)
+b_{23}\Bigl(\frac{z}{w}\Bigr)k_2^\pm(z)k_3^\pm(w)e_{24}^\pm(w)-b_{43}\Bigl(\frac{z}{w}\Bigr)k_3^\pm(w)k_2^\pm(z)e_{24}^\pm(z)\\
&=k_3^\pm(w)k_2^\pm(z)e_{23}^\pm(z)e_{34}^\pm(w)
-b_{23}\Bigl(\frac{z}{w}\Bigr)k_3^\pm(w)k_2^\pm(z)e_{23}^\pm(w)e_{34}^\pm(w)+b_{23}\Bigl(\frac{z}{w}\Bigr)k_2^\pm(z)k_3^\pm(w)e_{24}^\pm(w)\\
&\hspace{27.5em}-b_{43}\Bigl(\frac{z}{w}\Bigr)k_3^\pm(w)k_2^\pm(z)e_{24}^\pm(z).
\end{aligned}
\end{equation*}

 Finally, we get the desired equation by using the invertibility of $k_2^\pm(z)$ and $k_3^\pm(w)$.
\end{proof}

It follows from Lemma \ref{lemma e23e34} that
\begin{lemm}\label{lemm B3 X2X3} One has
\begin{gather*}\label{}
  \bigl(s^{-2}z-r^{-2}w\bigr)X_i^+(z)X_{i+1}^+(w)=\bigl(z-w\bigr)X_{i+1}^+(w)X_i^+(z), \\
  \bigl(z-w\bigr)X_i^-(z)X_{i+1}^-(w)=\bigl(s^{-2}z-r^{-2}w\bigr)X_{i+1}^-(w)X_i^-(z),
\end{gather*}
where $i=1,2$.
\end{lemm}

By establishing similar identities listed in Lemma \ref{lemma e23e34}, one can also prove
\begin{lemm} \label{lemm b3 x1x2}
One has
\begin{gather*}\label{}
   \bigl(s^2z-r^2w\bigr)X_i^+(z)X_i^+(w)=\bigl(r^2z-s^2w\bigr)X_i^+(w)X_i^+(z), \\
  \bigl(r^2z-s^2w\bigr)X_i^-(z)X_i^-(w)=\bigl(s^2z-r^2w\bigr)X_i^-(w)X_i^-(z),
\end{gather*}
where $i=1, 2$.
\end{lemm}

\begin{lemm} \label{lemm b3 x2+x3-}
One has
  \begin{gather*}\label{}
  X_i^\pm(z)X_{i+1}^\mp(w)=X_{i+1}^\mp(w)X_i^\pm(z), \quad i=1,2.
\end{gather*}
\end{lemm}
\begin{proof}
Here we only prove $i=2$ as an example. $M_{23}=M_{23}'$ leads to
  \begin{equation}\label{X2+ X3- 1}
    a_{23}\Bigl(\frac{z}{w}\Bigr)k_4^\pm(w)^{-1}f_{43}^\pm(w)\ell_{23}^\pm(z)=a_{34}\Bigl(\frac{z}{w}\Bigr)\ell_{23}^\pm(z)k_4^\pm(w)^{-1}f_{43}^\pm(w).
  \end{equation}
 By means of $M_{13}=M_{13}'$, we have
  \begin{equation}\label{X2+ X3- 2}
    a_{13}\Bigl(\frac{z}{w}\Bigr)k_4^\pm(w)^{-1}f_{43}^\pm(w)\ell_{13}^\pm(z)=a_{34}\Bigl(\frac{z}{w}\Bigr)\ell_{13}^\pm(z)k_4^\pm(w)^{-1}f_{43}^\pm(w).
  \end{equation}
 Then by subtracting $f_{21}^\pm(w)\cdot(\ref{X2+ X3- 2})$ from $(\ref{X2+ X3- 1})$ yields
  $$f_{43}^\pm(w)e_{23}^\pm(z)=e_{23}^\pm(z)f_{43}^\pm(w).$$

   Similarly, we can conclude that
  $$f_{43}^\mp(w)e_{23}^\pm(z)=e_{23}^\pm(z)f_{43}^\mp(w).$$

 We can thus derive the desired equation.
\end{proof}

\begin{lemm}\label{lemma B3 k3X3}
One has
\begin{gather*}\label{}
  k_3^\pm(z) X_3^+(w)=\frac{z-w_\pm}{rs^{-1}z-r^{-1}sw_\pm}X_3^+(w)k_3^\pm(z), \\
  k_3^\pm(z)X_3^-(w)=\frac{rs^{-1}z-r^{-1}sw_\mp}{z-w_\mp}X_3^-(w)k_3^\pm(z).
\end{gather*}
\end{lemm}
\begin{proof}
Here we only prove the first equation since the other one can be proved similarly.

 From $M_{33}=M'_{33}$, we have
 \begin{equation}\label{B3 k3X3 3}
   \begin{aligned}
   &a_{34}\Big(\frac{z}{w}\Big)e_{34}^\pm(w)k_4^\pm(w)^{-1}\ell_{33}^\pm(z)\\
   &=b_{13}\Big(\frac{z}{w}\Big)\ell_{31}^\pm(z)\Bigl\{e_{12}^\pm(w)e_{23}^\pm(w)e_{34}^\pm(w)-e_{12}^\pm(w)e_{24}^\pm(w)-e_{13}^\pm(w)e_{34}^\pm(w)+e_{14}^\pm(w)\Bigr\}k_4^\pm(w)^{-1}\\
   &-b_{23}\Big(\frac{z}{w}\Big)\ell_{32}^\pm(z)\Bigl\{e_{23}^\pm(w)e_{34}^\pm(w)
   -e_{24}^\pm(w)\Bigr\}k_4^\pm(w)^{-1}+\ell_{33}^\pm(z)e_{34}^\pm(w)k_4^\pm(w)^{-1}-b_{43}\Big(\frac{z}{w}\Big)\ell_{34}^\pm(z)k_4^\pm(w)^{-1}.
   \end{aligned}
 \end{equation}

 By $M_{23}=M'_{23}$,
 \begin{equation}\label{B3 k3X3 2}
   \begin{aligned}
   &a_{34}\Big(\frac{z}{w}\Big)e_{34}^\pm(w)k_4^\pm(w)^{-1}\ell_{23}^\pm(z)\\
   &=b_{13}\Big(\frac{z}{w}\Big)\ell_{21}^\pm(z)\Bigl\{e_{12}^\pm(w)e_{23}^\pm(w)e_{34}^\pm(w)-e_{12}^\pm(w)e_{24}^\pm(w)
   -e_{13}^\pm(w)e_{34}^\pm(w)+e_{14}^\pm(w)\Bigr\}k_4^\pm(w)^{-1}\\
   &-b_{23}\Big(\frac{z}{w}\Big)\ell_{22}^\pm(z)\Bigl\{e_{23}^\pm(w)e_{34}^\pm(w)-e_{24}^\pm(w)\Bigr\}k_4^\pm(w)^{-1}
   +\ell_{23}^\pm(z)e_{34}^\pm(w)k_4^\pm(w)^{-1}-b_{43}\Big(\frac{z}{w}\Big)\ell_{24}^\pm(z)k_4^\pm(w)^{-1}.
   \end{aligned}
 \end{equation}

 And $M_{13}=M'_{13}$ leads to
 \begin{equation}\label{B3 k3X3 1}
   \begin{aligned}
   &a_{34}\Big(\frac{z}{w}\Big)e_{34}^\pm(w)k_4^\pm(w)^{-1}\ell_{13}^\pm(z)\\
   &=b_{13}\Big(\frac{z}{w}\Big)\ell_{11}^\pm(z)\Bigl\{e_{12}^\pm(w)e_{23}^\pm(w)e_{34}^\pm(w)-e_{12}^\pm(w)e_{24}^\pm(w)
   -e_{13}^\pm(w)e_{34}^\pm(w)+e_{14}^\pm(w)\Bigr\}k_4^\pm(w)^{-1}\\
   &-b_{23}\Big(\frac{z}{w}\Big)\ell_{12}^\pm(z)\Bigl\{e_{23}^\pm(w)e_{34}^\pm(w)
   -e_{24}^\pm(w)\Bigr\}k_4^\pm(w)^{-1}+\ell_{13}^\pm(z)e_{34}^\pm(w)k_4^\pm(w)^{-1}-b_{43}\Big(\frac{z}{w}\Big)\ell_{14}^\pm(z)k_4^\pm(w)^{-1}.
   \end{aligned}
 \end{equation}

 Then, subtracting $f_{32}^\pm(z)\cdot(\ref{B3 k3X3 2})$ and $\Bigl\{f_{31}^\pm(z)-f_{32}^\pm(z)f_{21}^\pm(z)\Bigr\}\cdot(\ref{B3 k3X3 1})$ from $(\ref{B3 k3X3 3})$,
 we can get
 \begin{equation}\label{k3 e34}
   k_3^\pm(z)e_{34}^\pm(w_\pm)=a_{34}\Bigl(\frac{z}{w_\pm}\Bigr)e_{34}^\pm(w_\pm)k_3^\pm(z)+b_{43}\Big(\frac{z}{w_\pm}\Big)k_3^\pm(z)e_{34}^\pm(z).
 \end{equation}

   Similarly, we conclude that
\begin{equation*}
k_3^\pm(z)e_{34}^\mp(w_\mp)=a_{34}\Bigl(\frac{z}{w_\pm}\Bigr)e_{34}^\mp(w_\mp)k_3^\pm(z)+b_{43}\Big(\frac{z}{w_\pm}\Big)k_3^\pm(z)e_{34}^\pm(z).
\end{equation*}

 This completes the proof.
\end{proof}

\begin{lemm}\label{lemma e34f43}
One has
\begin{align*}
 &\Bigl[e_{i, i+1}^\pm(z), f_{i+1, i}^\pm(w)\Bigr]=\frac{\bigl(r^2-s^2\bigr)w}{\bigl(z-w\bigr)}\Bigl(k_i^\pm(w)^{-1}k_{i+1}^\pm(w)-k_i^\pm(z)^{-1}k_{i+1}^\pm(z)\Bigr),\\
   &\Bigl[e_{i, i+1}^\pm(z), f_{i+1, i}^\mp(w)\Bigr]=\frac{\bigl(r^2-s^2\bigr)w_\pm}{\bigl(z_\mp-w_\pm\bigr)}k_i^\mp(w)^{-1}k_{i+1}^\mp(w)-\frac{\bigl(r^2-s^2\bigr)w_\mp}{\bigl(z_\pm-w_\mp\bigr)}k_i^\pm(z)^{-1}k_{i+1}^\pm(z),\\
   &\Bigl[e_{34}^\pm(z), f_{43}^\pm(w)\Bigr]=\frac{\bigl(rs^{-1}-r^{-1}s\bigr)w}{\bigl(z-w\bigr)}\Bigl(k_3^\pm(w)^{-1}k_4^\pm(w)-k_3^\pm(z)^{-1}k_4^\pm(z)\Bigr),\\
   &\Bigl[e_{34}^\pm(z), f_{43}^\mp(w)\Bigr]=\frac{\bigl(rs^{-1}-r^{-1}s\bigr)w_\pm}{\bigl(z_\mp-w_\pm\bigr)}k_3^\mp(w)^{-1}k_4^\mp(w)-\frac{\bigl(rs^{-1}-r^{-1}s\bigr)w_\mp}{\bigl(z_\pm-w_\mp\bigr)}k_3^\pm(z)^{-1}k_4^\pm(z),
\end{align*}
where $1\leq i< 3$.
\end{lemm}
\begin{proof}
   We only prove the last equation since the other one can be proved similarly.
    By $M_{34}=M'_{34}$, we  get
   \begin{equation}\label{M34 43}
a_{43}\Bigl(\frac{z_\mp}{w_\pm}\Bigr)\ell_{43}^\mp(w)\ell_{34}^\pm(z)+b_{34}\Bigl(\frac{z_\mp}{w_\pm}\Bigr)\ell_{44}^\mp(w)\ell_{33}^{\pm}(z)
=b_{34}\Bigl(\frac{z_\pm}{w_\mp}\Bigr)\ell_{44}^\pm(z)\ell_{33}^{\mp}(w)+a_{34}\Bigl(\frac{z_\pm}{w_\mp}\Bigr)\ell_{34}^\pm(z)\ell_{43}^\mp(w),\\
\end{equation}
\begin{equation}\label{M34 33}
a_{43}\Bigl(\frac{z_\mp}{w_\pm}\Bigr)\ell_{33}^\mp(w)\ell_{34}^\pm(z)+b_{34}\Bigl(\frac{z_\mp}{w_\pm}\Bigr)\ell_{34}^\mp(w)\ell_{33}^{\pm}(z)
=\ell_{34}^\pm(z)\ell_{33}^\mp(w),\\
\end{equation}

\begin{equation}\label{M34 23}
a_{43}\Bigl(\frac{z_\mp}{w_\pm}\Bigr)\ell_{23}^\mp(w)\ell_{34}^\pm(z)+b_{34}\Bigl(\frac{z_\mp}{w_\pm}\Bigr)\ell_{24}^\mp(w)\ell_{33}^{\pm}(z)
=b_{32}\Bigl(\frac{z_\pm}{w_\mp}\Bigr)\ell_{24}^\pm(z)\ell_{33}^{\mp}(w)+a_{32}\Bigl(\frac{z_\pm}{w_\mp}\Bigr)\ell_{34}^\pm(z)\ell_{23}^\mp(w),\\
\end{equation}

\begin{equation}\label{M34 13}
a_{43}\Bigl(\frac{z_\mp}{w_\pm}\Bigr)\ell_{13}^\mp(w)\ell_{34}^\pm(z)+b_{34}\Bigl(\frac{z_\mp}{w_\pm}\Bigr)\ell_{14}^\mp(w)\ell_{33}^\pm(z)
=a_{31}\Bigl(\frac{z_\pm}{w_\mp}\Bigr)\ell_{34}^\pm(z)\ell_{13}^\mp(w)+b_{31}\Bigl(\frac{z_\pm}{w_\mp}\Bigr)\ell_{14}^\pm(z)\ell_{33}^\mp(w).
\end{equation}

By subtracting $ f_{43}^\mp(w)\cdot (\ref{M34 33})$, $\Bigl(f_{42}^\mp(w)-f_{43}^\mp(w)f_{32}^\mp(w)\Bigr)\cdot (\ref{M34 23})$,
and $\Bigl\{f_{41}^\mp(w)-f_{42}^\mp(w)f_{21}^\mp(w)-f_{43}^\mp(w)f_{31}^\mp(w)+f_{43}^\mp(w)f_{32}^\mp(w)f_{21}^\mp(w)\Bigr\}\cdot (\ref{M34 13})$ from (\ref{M34 43}), we derive that

\begin{align}\label{X3+ X3- 3}
&b_{34}\Bigl(\frac{z_\mp}{w_\pm}\Bigr)k_{4}^\mp(w)\ell_{33}^\pm(z)-a_{34}\Bigl(\frac{z_\pm}{w_\mp}\Bigr)\ell_{34}^\pm(z)f_{43}^\mp(w)k_3^\mp(w)-b_{24}\Bigl(\frac{z_\mp}{w_\pm}\Bigr)k_{4}^\mp(w)\ell_{32}^\pm(z)e_{23}^\mp(w)\\
&\hspace{14.5em}-b_{14}\Bigl(\frac{z_\mp}{w_\pm}\Bigr)k_{4}^\mp(w)\ell_{31}^\pm(z)\Bigl\{e_{13}^\mp(w)-e_{12}^\mp(w)e_{23}^\mp(w)\Bigr\}\nonumber\\
&=b_{34}\Bigl(\frac{z_\pm}{w_\mp}\Bigr)\ell_{44}^\pm(z)k_{3}^\mp(w)\nonumber\\
&\quad-b_{31}\Bigl(\frac{z_\pm}{w_\mp}\Bigr)\Bigl\{f_{41}^\mp(w)-f_{42}^\mp(w)f_{21}^\mp(w)-f_{43}^\mp(w)f_{31}^\mp(w)+f_{43}^\mp(w)f_{32}^\mp(w)f_{21}^\mp(w)\Bigr\}\ell_{14}^\pm(z)k_{3}^\mp(w)\nonumber\\
&\quad\hspace{7.5em}-b_{32}\Bigl(\frac{z_\pm}{w_\mp}\Bigr)\Bigl\{f_{42}^\mp(w)-f_{43}^\mp(w)f_{32}^\mp(w)\Bigr\}\ell_{24}^\pm(z)k_3^\mp(w)-f_{43}^\mp(w)\ell_{34}^\pm(z)k_3^\mp(w).\nonumber
\end{align}

 Similarly, using $M_{24}=M'_{24}$ and $ M_{14}=M'_{14}$,  we obtain
\begin{equation}
\begin{aligned}\label{X3+ X3- 2}
b_{34}\Bigl(\frac{z_\mp}{w_\pm}\Bigr)k_4^\mp(w)\ell_{23}^\pm(z)=&a_{24}\Bigl(\frac{z_\pm}{w_\mp}\Bigr)\Bigl[\ell_{24}^\pm(z),f_{43}^\mp(w)\Bigr]k_3^\mp(w)+b_{14}\Bigl(\frac{z_\mp}{w_\pm}\Bigr)k_{4}^\mp(w)\\
&\cdot \Bigl\{\ell_{21}^\pm(z)\Bigl(e_{13}^\mp(w)-e_{12}^\mp(w)e_{23}^\mp(w)\Bigr)+\ell_{22}^\pm(z)e_{23}^\mp(w)\Bigr\},
\end{aligned}
\end{equation}
 and
\begin{equation}
\begin{aligned}\label{X3+ X3- 1}
b_{34}\Bigl(\frac{z_\mp}{w_\pm}\Bigr)k_4^\mp(w)\ell_{13}^\pm(z)=&a_{14}\Bigl(\frac{z_\pm}{w_\mp}\Bigr)\Bigl[\ell_{14}^\pm(z),f_{43}^\mp(w)\Bigr]k_3^\mp(w)+b_{14}\Bigl(\frac{z_\mp}{w_\pm}\Bigr)k_4^\mp(w)\\
&\cdot \Bigl\{k_1^\pm(z)\Bigl(e_{13}^\mp(w)-e_{12}^\mp(w)e_{23}^\mp(w)\Bigr)+\ell_{12}^\pm(z)e_{23}^\mp(w)\Bigr\}.
\end{aligned}
\end{equation}

 Combining the result of $(\ref{X3+ X3- 3})-f_{32}^\pm(z)\cdot(\ref{X3+ X3- 2})-\Bigl(f_{31}^\pm(z)-f_{32}^\pm(z)f_{21}^\pm(z)\Bigr)\cdot(\ref{X3+ X3- 1})$ with Lemma \ref{lemma B3 k3X3}, we find that
\begin{align*}\label{X3+ X3- 4}
&b_{34}\Bigl(\frac{z_\mp}{w_\pm}\Bigr)k_{4}^\mp(w)k_{3}^\pm(z)
=b_{34}\Bigl(\frac{z_\pm}{w_\mp}\Bigr)k_4^\pm(z)k_3^\mp(w)+a_{34}\Bigl(\frac{z_\pm}{w_\mp}\Bigr)k_{3}^\pm(z)\Bigl[e_{34}^\pm(z),f_{43}^\mp(w)\Bigr]k_3^\mp(w).
\end{align*}
This is exactly the desired equation.
\end{proof}

Similarly one can prove

\begin{lemm}\label{lemm xi+xi+}
 One has
  \begin{gather*}
  \Bigl[X_i^+(z), X_i^-(w)\Bigr] =\bigl(r^2-s^2\bigr)\Bigl\{\delta\Bigl(\frac{z_-}{w_+}\Bigr)k_i^-(w_+)^{-1}k_{i+1}^-(w_+)-\delta\Bigl(\frac{z_+}{w_-}\Bigr)k_i^+(z_+)^{-1}k_{i+1}^+(z_+)\Bigr\}, \\
  \Bigl[X_n^+(z), X_n^-(w)\Bigr]  =\bigl(rs^{-1}-r^{-1}s\bigr)\Bigl\{\delta\Bigl(\frac{z_-}{w_+}\Bigr)k_n^-(w_+)^{-1}k_{n+1}^-(w_+)-\delta\Bigl(\frac{z_+}{w_-}\Bigr)k_n^+(z_+)^{-1}k_{n+1}^+(z_+)\Bigr\},
\end{gather*}
  where $1\leq i< 3$.
\end{lemm}

\begin{proof}
  By Lemma \ref{lemma e34f43}, we check it similarly to that of Proposition 4.10 in \cite{JingLJA 2017}.
\end{proof}

\begin{lemm}\label{lemma B3 X3+X3+} One has
\begin{gather*}\label{}
  X_3^+(z)X_3^+(w)=\frac{rz-sw}{sz-rw}X_3^+(w)X_3^+(z), \\
  X_3^-(z)X_3^-(w)=\frac{sz-rw}{rz-sw}X_3^-(w)X_3^-(z).
\end{gather*}
\end{lemm}
\begin{proof}
  Here we only prove the first equation since another can be proved by the same token.
  From $M_{34}=M_{34}'$, we have
  \begin{equation}\label{M34 34}
\sum_{i=1}^{7}c_{i4}\Bigl(\frac{z}{w}\Bigr)\ell_{3i}^\pm(w)\ell_{3i'}^\pm(z)=\ell_{34}^\pm(z)\ell_{34}^\pm(w),
\end{equation}
\begin{equation}\label{M34 24}
\sum_{i=1}^{7}c_{i4}\Bigl(\frac{z}{w}\Bigr)\ell_{2i}^\pm(w)\ell_{3i'}^\pm(z)
=a_{32}\Bigl(\frac{z}{w}\Bigr)\ell_{34}^\pm(z)\ell_{24}^\pm(w)+b_{32}\Big(\frac{z}{w}\Big)\ell_{24}^\pm(z)\ell_{34}^\pm(w),
\end{equation}
\begin{equation}\label{M34 14}
\sum_{i=1}^{7}c_{i5}\Bigl(\frac{z}{w}\Bigr)\ell_{1i}^\pm(w)\ell_{3i'}^\pm(z)
=a_{31}\Bigl(\frac{z}{w}\Bigr)\ell_{34}^\pm(z)\ell_{14}^\pm(w)+b_{31}\Big(\frac{z}{w}\Big)\ell_{14}^\pm(z)\ell_{34}^\pm(w).
\end{equation}

Subsequently, subtracting $f_{32}^\pm(w)\cdot (\ref{M34 24})$ and $\Bigl(f_{31}^\pm(w)-f_{32}^\pm(w)f_{21}^\pm(w)\Bigr)\cdot(\ref{M34 14})$ from (\ref{M34 34}), \\ we obtain
\begin{equation}
\begin{aligned}\label{X3+X3+ 3}
&\sum_{i=3}^{7}c_{i4}\Bigl(\frac{z}{w}\Bigr)e_{3i}^\pm(w)\ell_{3i'}^\pm(z)
=b_{14}\Bigl(\frac{z}{w}\Bigr)e_{34}^\pm(w)\Bigl\{\ell_{31}^\pm(z)e_{14}^\pm(w)+\Bigl(\ell_{32}^\pm(z)
-\ell_{31}^\pm(z)e_{12}^\pm(w)\Bigr)e_{24}^\pm(w)\Bigr\}\\
&+b_{31}\Bigl(\frac{z}{w}\Bigr)k_3^\pm(w)^{-1}\Bigl\{\Bigl(f_{32}^\pm(w)f_{21}^\pm(w)
-f_{31}^\pm(w)\Bigr)\ell_{14}^\pm(z)+\ell_{34}^\pm(z)-f_{32}^\pm(w)\ell_{24}^\pm(z)\Bigr\}k_3^\pm(w)e_{34}^\pm(w).
\end{aligned}
\end{equation}

Taking $M_{24}=M_{24}'$, we conclude that
\begin{equation}
\begin{aligned}\label{X3+X3+ 2}
\sum_{i=3}^{7}c_{i4}\Bigl(\frac{z}{w}\Bigr)e_{3i}^\pm(w)\ell_{2i'}^\pm(z)&=a_{23}\Bigl(\frac{z}{w}\Bigr)k_3^\pm(w)^{-1}
\ell_{24}^\pm(z)k_3^\pm(w)e_{34}^\pm(w)+b_{14}\Bigl(\frac{z}{w}\Bigr)e_{34}^\pm(w)\\
&\quad\cdot\Bigl\{\ell_{21}^\pm(z)\Bigl(e_{14}^\pm(w)-e_{12}^\pm(w)e_{24}^\pm(w)\Bigr)+\ell_{22}^\pm(z)e_{24}^\pm(w)\Bigr\}.
\end{aligned}
\end{equation}

Similarly, using $M_{14}=M_{14}'$, we derive that
\begin{equation}\label{X3+X3+ 1}
\begin{aligned}
\sum_{i=3}^{7}c_{i4}\Bigl(\frac{z}{w}\Bigr)e_{3i}^\pm(w)\ell_{1i'}^\pm(z)
&=a_{13}\Bigl(\frac{z}{w}\Bigr)k_3^\pm(w)^{-1}\ell_{14}^\pm(z)k_3^\pm(w)e_{34}^\pm(w)+b_{14}\Bigl(\frac{z}{w}\Bigr)e_{34}^\pm(w)\\
&\quad\cdot k_{1}^\pm(z)\Bigl\{\Bigl(e_{14}^\pm(w)-e_{12}^\pm(w)e_{24}^\pm(w)\Bigr)+e_{12}^\pm(z)e_{24}^\pm(w)\Bigr\}.
\end{aligned}
\end{equation}

Subtracting $f_{32}^\pm(w)\cdot(\ref{X3+X3+ 2})$ and $\Bigl(f_{31}^\pm(w)-f_{32}^\pm(w)f_{21}^\pm(w)\Bigr)\cdot(\ref{X3+X3+ 1})$ from $(\ref{X3+X3+ 3})$ leads to
\begin{equation}\label{X3+X3+ 4}
\begin{aligned}
\sum_{i=3}^{7}c_{i4}\Bigl(\frac{z}{w}\Bigr)k_3^\pm(w)\Bigl\{e&_{3i}^\pm(w)\ell_{3i'}^\pm(z)
-\Bigl(f_{31}^\pm(z)-f_{32}^\pm(z)f_{21}^\pm(z)\Bigr)e_{3i}^\pm(w)\ell_{1i'}^\pm(z)\\
&-f_{32}^\pm(z)e_{3i}^\pm(w)\ell_{2i'}^\pm(z)\Bigr\}=k_3^\pm(z)e_{34}^\pm(z)k_3^\pm(w)e_{34}^\pm(w).
\end{aligned}
\end{equation}

Using $M_{33}=M_{33}'$, we have
\begin{equation*}\label{M33 35}
\sum_{i=1}^{7}c_{i5}\Bigl(\frac{z}{w}\Bigr)\ell_{3i}^\pm(w)\ell_{3i'}^\pm(z)=\ell_{33}^\pm(z)\ell_{35}^\pm(w).
\end{equation*}

 It follows that
\begin{equation}\label{X3+X3+ 5}
\begin{aligned}
\sum_{i=3}^{7}c_{i5}\Bigl(\frac{z}{w}\Bigr)k_3^\pm(w)\Bigl\{&e_{3i}^\pm(w)\ell_{3i'}^\pm(z)-\Bigl(f_{31}^\pm(z)-f_{32}^\pm(z)f_{21}^\pm(z)\Bigr)e_{3i}^\pm(w)\ell_{1i'}^\pm(z)\\
&\quad-f_{32}^\pm(z)e_{3i}^\pm(w)\ell_{2i'}^\pm(z)\Bigr\}=k_3^\pm(z)k_3^\pm(w)e_{35}^\pm(w).
\end{aligned}
\end{equation}

Noticing that $c_{64}\bigl(\frac{z}{w}\bigr)=\bigl(rs^{-1}\bigr)^{\frac{1}{2}}c_{65}\bigl(\frac{z}{w}\bigr)$, and $c_{74}\bigl(\frac{z}{w}\bigr)=\bigl(rs^{-1}\bigr)^{\frac{1}{2}}c_{75}\bigl(\frac{z}{w}\bigr)$,
and subtracting $\bigl(rs^{-1}\bigl)^{\frac{1}{2}}\\ \cdot(\ref{X3+X3+ 5})$ from $(\ref{X3+X3+ 4})$, we arrive at
\begin{equation}\label{X3+X3+ 6}
\begin{aligned}
&\Bigl\{c_{44}\Bigl(\frac{z}{w}\Bigr)-\Bigl(rs^{-1}\Bigr)^{\frac{1}{2}}c_{45}\Bigl(\frac{z}{w}\Bigr)\Bigr\}k_3^\pm(w)e_{34}^\pm(w)k_3^\pm(z)e_{34}^\pm(z)\\
&\quad+\Bigl\{c_{54}\Bigl(\frac{z}{w}\Bigr)-\Bigl(rs^{-1}\Bigr)^{\frac{1}{2}}c_{55}\Bigl(\frac{z}{w}\Bigr)\Bigr\}k_3^\pm(w)e_{35}^\pm(w)k_3^\pm(z)\\
&\quad+\Bigl\{c_{34}\Bigl(\frac{z}{w}\Bigr)-\Bigl(rs^{-1}\Bigr)^{\frac{1}{2}}c_{35}\Bigl(\frac{z}{w}\Bigr)\Bigr\} k_3^\pm(w)k_3^\pm(z)e_{35}^\pm(z)\\
&=k_3^\pm(z)e_{34}^\pm(z)k_3^\pm(w)e_{34}^\pm(w)-\Bigl(rs^{-1}\Bigl)^{\frac{1}{2}}k_3^\pm(z)k_3^\pm(w)e_{35}^\pm(w).
\end{aligned}
\end{equation}

To get the commutation relations between $k_3^\pm(w)e_{35}^\pm(w)k_3^\pm(z)$ and $k_3^\pm(w)e_{34}^\pm(w)k_3^\pm(z)e_{34}^\pm(z)$,
from $M_{31}=M_{31}'$, $M_{32}=M_{32}'$ and $M_{35}=M_{35}'$, we have
\begin{equation}\label{X3+X3+ 7}
\begin{aligned}
\sum_{i=3}^{7}c_{i3}\Bigl(\frac{z}{w}\Bigr)k_3^\pm(w)\Bigl\{&e_{3i}^\pm(w)\ell_{3i'}^\pm(z)-\Bigl(f_{31}^\pm(z)-f_{32}^\pm(z)f_{21}^\pm(z)\Bigr)e_{3i}^\pm(w)\ell_{1i'}^\pm(z)\\
&\quad-f_{32}^\pm(z)e_{3i}^\pm(w)\ell_{2i'}^\pm(z)\Bigr\}=k_3^\pm(z)e_{35}^\pm(z)k_3^\pm(w).
\end{aligned}
\end{equation}
\begin{equation}\label{X3+X3+ 8}
\begin{aligned}
\sum_{i=3}^{7}c_{ij}\Bigl(\frac{z}{w}\Bigr)k_3^\pm(w)\Bigl\{&e_{3i}^\pm(w)\ell_{3i'}^\pm(z)-\Bigl(f_{31}^\pm(z)-f_{32}^\pm(z)f_{21}^\pm(z)\Bigr)e_{3i}^\pm(w)\ell_{1i'}^\pm(z)\\
&\quad-f_{32}^\pm(z)e_{3i}^\pm(w)\ell_{2i'}^\pm(z)\Bigr\}=0,
\end{aligned}
\end{equation}
where $j=6, 7$.

Combining with (\ref{X3+X3+ 4}), (\ref{X3+X3+ 5}), (\ref{X3+X3+ 7}) and (\ref{X3+X3+ 8}), we conclude that
\begin{equation}\label{X3+X3+ 9}
\begin{aligned}
k_3^\pm(w)e_{35}^\pm(w)k_3^\pm(z)=&\;\frac{r^{\frac{3}{2}}s^{-\frac{3}{2}}-r^{-\frac{1}{2}}s^{\frac{1}{2}}}{z-rs^{-1}}k_3^\pm(w)e_{34}^\pm(w)k_3^\pm(z)e_{34}^\pm(z)+k_3^\pm(z)k_3^\pm(w)e_{35}^\pm(w)\\
&+\ast \;k_3^\pm(z)k_3^\pm(w)e_{35}^\pm(z),
\end{aligned}
\end{equation}
where $\ast$ denote some coefficients.

If we plug (\ref{X3+X3+ 9}) back into (\ref{X3+X3+ 6}), then we have:
\begin{equation}\label{}
\begin{aligned}
  k_3^\pm(z)e_{34}^\pm(z)k_3^\pm(w)e_{34}^\pm(w)=&\frac{(s^2z-r^2w)(rz-sw)}{(r^2z-s^2w)(sz-rw)}k_3^\pm(w)e_{34}^\pm(w)k_3^\pm(z)e_{34}^\pm(z)\\
&+\ast_1\; k_3^\pm(w)k_3^\pm(z)e_{35}^\pm(w)+\ast_2\; k_3^\pm(w)k_3^\pm(z)e_{35}^\pm(z),
\end{aligned}
\end{equation}
where $\ast_1, \ast_2$ denote some coefficients.

Finally, using (\ref{k3 e34}), we can obtain the desired equation.
\end{proof}

\begin{lemm} \label{lemm b3 k4x3}
One has
\begin{gather*}\label{}
  k_4^\pm(z)X_3^+(w)=\frac{rs\bigl(z-w_\pm\bigr)\bigl(rz-sw_\pm\bigr)}{(r^2z-s^2w_\pm\bigr)\bigl(sz-rw_\pm\bigr)}X_3^+(w)k_4^\pm(z), \\
  k_4^\pm(z)X_3^-(w)=\frac{\bigl(r^2z-s^2w_\mp\bigr)\bigl(sz-rw_\mp\bigr)}{rs\bigl(z-w_\mp\bigr)\bigl(rz-sw_\mp\bigr)}X_3^-(w)k_4^\pm(z).
\end{gather*}
\end{lemm}

\begin{proof}
Here we only prove the first equation since the another  can be proved similarly.
 Using $M_{34}=M_{34}'$, we have

 \begin{equation}\label{M34 44}
\sum_{i=1}^{7}c_{i4}\Bigl(\frac{z}{w}\Bigr)\ell_{4i}^\pm(w)\ell_{3i'}^\pm(z)=b_{34}\Bigl(\frac{z}{w}\Bigr)\ell_{44}^\pm(z)\ell_{34}^\pm(w)+a_{34}\Bigl(\frac{z}{w}\Bigr)\ell_{34}^\pm(z)\ell_{44}^\pm(w).\\
\end{equation}

Similarly, we can obtain:
\begin{align}\label{k4 X3 1}
\sum_{i=4}^{7}&c_{i4}\Bigl(\frac{z}{w}\Bigr)k_4^\pm(w)\Bigl\{e_{4i}^\pm(w)\ell_{4i'}^\pm(z)-\Bigl[f_{41}^\pm(z)-f_{42}^\pm(z)f_{21}^\pm(z)-f_{43}^\pm(z)f_{31}^\pm(z)+f_{43}^\pm(z)\nonumber\\
&\cdot f_{32}^\pm(z)f_{21}^\pm(z)\Bigr] e_{4i}^\pm(w)\ell_{1i'}^\pm(z)-\Bigl(f_{42}^\pm(z)-f_{43}^\pm(z)f_{32}^\pm(z)\Bigr)e_{4i}^\pm(w)\ell_{2i'}^\pm(z)-f_{43}^\pm(z) \nonumber\\
&\cdot e_{4i}^\pm(w)\ell_{3i'}^\pm(z)\Bigr\} =a_{34}\Bigl(\frac{z}{w}\Bigr)k_3^\pm(z)e_{34}^\pm(w)k_4^\pm(w)+b_{34}\Bigl(\frac{z}{w}\Bigr)k_4^\pm(w)k_3^\pm(z)e_{34}^\pm(w).
\end{align}
\begin{align}\label{k4 X3 2}
\sum_{i=4}^{7}&c_{i5}\Bigl(\frac{z}{w}\Bigr)k_4^\pm(w)\Bigl\{e_{4i}^\pm(w)\ell_{4i'}^\pm(z)-\Bigl[f_{41}^\pm(z)-f_{42}^\pm(z)f_{21}^\pm(z)-f_{43}^\pm(z)f_{31}^\pm(z)+f_{43}^\pm(z)\nonumber\\
&\cdot f_{32}^\pm(z)f_{21}^\pm(z)\Bigr] e_{4i}^\pm(w)\ell_{1i'}^\pm(z)-\Bigl(f_{42}^\pm(z)-f_{43}^\pm(z)f_{32}^\pm(z)\Bigr)e_{4i}^\pm(w)\ell_{2i'}^\pm(z)-f_{43}^\pm(z) \nonumber\\
&\cdot e_{4i}^\pm(w)\ell_{3i'}^\pm(z)\Bigr\} =a_{34}\Bigl(\frac{z}{w}\Bigr)k_3^\pm(z)k_4^\pm(w)e_{45}^\pm(w).
\end{align}
\begin{align}\label{k4 X3 3}
  \sum_{i=4}^{7}&c_{ij}\Bigl(\frac{z}{w}\Bigr)k_4^\pm(w)\Bigl\{e_{4i}^\pm(w)\ell_{4i'}^\pm(z)-\Bigl[f_{41}^\pm(z)-f_{42}^\pm(z)f_{21}^\pm(z)-f_{43}^\pm(z)f_{31}^\pm(z)+f_{43}^\pm(z)\nonumber\\
&\cdot f_{32}^\pm(z)f_{21}^\pm(z)\Bigr] e_{4i}^\pm(w)\ell_{1i'}^\pm(z)-\Bigl(f_{42}^\pm(z)-f_{43}^\pm(z)f_{32}^\pm(z)\Bigr)e_{4i}^\pm(w)\ell_{2i'}^\pm(z)-f_{43}^\pm(z) \nonumber\\
&\cdot e_{4i}^\pm(w)\ell_{3i'}^\pm(z)\Bigr\}=0,
\end{align}
where $j=6,7$.

 Then subtracting $\bigl(rs^{-1}\bigr)^{\frac{1}{2}}(\ref{k4 X3 2})$ from $(\ref{k4 X3 1})$ leads to
\begin{equation}\label{k4 X3 4}
\begin{aligned}
\Bigl\{c&_{44}\Bigl(\frac{z}{w}\Bigr)-\bigl(rs^{-1}\bigr)^{\frac{1}{2}}c_{45}\Bigl(\frac{z}{w}\Bigr)\Bigr\}k_4^\pm(w)k_3^\pm(z)e_{34}^\pm(z)+\Bigl\{c_{54}\Bigl(\frac{z}{w}\Bigr)-\bigl(rs^{-1}\bigr)^{\frac{1}{2}}c_{55}\Bigl(\frac{z}{w}\Bigr)\Bigr\}k_4^\pm(w)e_{45}^\pm(w) \\
&\cdot k_3^\pm(z)=b_{34}\Bigl(\frac{z}{w}\Bigr)k_4^\pm(w)k_3^\pm(z)e_{34}^\pm(w)+a_{34}\Bigl(\frac{z}{w}\Bigr)k_3^\pm(z) \Bigl\{e_{34}^\pm(z)k_4^\pm(w)-\bigl(rs^{-1}\bigr)^{\frac{1}{2}}k_4^\pm(w)e_{45}^\pm(w)\Bigr\}.
\end{aligned}
\end{equation}

Combining with (\ref{k4 X3 1}), (\ref{k4 X3 2}) and (\ref{k4 X3 3}), we conclude that:
\begin{equation}\label{k4 X3 5}
\begin{aligned}
k_4^\pm(w)e_{45}^\pm(w)k_3^\pm(z)=\frac{r^{\frac{3}{2}}s^{-\frac{3}{2}}-r^{-\frac{1}{2}}s^{\frac{1}{2}}}{z-rs^{-1}}k_4^\pm(w)k_3^\pm(z)e_{34}^\pm(z)+a_{34}\Bigl(\frac{z}{w}\Bigr)k_3^\pm(z)k_4^\pm(w)e_{45}^\pm(w).
\end{aligned}
\end{equation}

Then substituting  (\ref{k4 X3 5}) back into (\ref{k4 X3 4}), and using the invertibility of $k_3^\pm(z)$,
we obtain the following identity:
\begin{equation*}
\begin{aligned}
k_4^\pm(w)e_{34}^\pm(z)=\frac{rs(z-w)(sz-rw)}{(s^2z-r^2w)(rz-sw)}e_{34}^\pm(z)k_4^\pm(w)+\ast_1 k_4^\pm(w)e_{45}^\pm(w)+\ast_2 k_4^\pm(w)e_{34}^\pm(w),
\end{aligned}
\end{equation*}
where $\ast_1, \ast_2$ denote some nonzero coefficients.

We finally arrive at the desired equation after exchanging $z$ with $w$.
\end{proof}

\begin{lemm} \label{lemm k4k4}
One has
\begin{gather*}\label{}
  k_4^\pm(z)k_4^\pm(w)=k_4^\pm(w)k_4^\pm(z), \\
  \frac{s^2z_\pm-r^2w_\mp}{r^2z_\pm-s^2w_\mp}\frac{rz_\pm-sw_\mp}{sz_\pm-rw_\mp}k_4^\pm(z)k_4^\mp(w)=\frac{s^2z_\mp-r^2w_\pm}{r^2z_\mp-s^2w_\pm}\frac{rz_\mp-sw_\pm}{sz_\mp-rw_\pm}k_4^\mp(w)k_4^\pm(z).
\end{gather*}
\end{lemm}

\begin{proof}
  Here we only prove the first equation since the other one can be proved similarly.
  Using $M_{44}=M_{44}'$, we have
  \begin{equation*}\label{M44 44}
    \sum_{i=4}^{7}c_{i4}\Bigl(\frac{z}{w}\Bigr)\ell_{4i}^\pm(w)\ell_{4i'}^\pm(z)=\sum_{i=4}^{7}c_{4i}\Bigl(\frac{z}{w}\Bigr)\ell_{i' 4}^\pm(z)\ell_{i 4}^\pm(w).
  \end{equation*}

  Through some calculations,  we conclude that
  \begin{align*}
  \sum_{i=4}^{7}c_{i4}&\Bigl(\frac{z}{w}\Bigr)k_4^\pm(w)\Bigl\{e_{4i}^\pm(w)\ell_{4i'}^\pm(z)-\Bigl[f_{41}^\pm(z)-f_{42}^\pm(z)f_{21}^\pm(z)-f_{43}^\pm(z)f_{31}^\pm(z)+f_{43}^\pm(z)f_{32}^\pm(z)f_{21}^\pm(z)\Bigr]\\
&\quad\cdot e_{4i}^\pm(w)\ell_{1i'}^\pm(z)-\Bigl(f_{42}^\pm(z)-f_{43}^\pm(z)f_{32}^\pm(z)\Bigr)e_{4i}^\pm(w)\ell_{2i'}^\pm(z)-f_{43}^\pm(z)e_{4i}^\pm(w)\ell_{3i'}^\pm(z)\Bigr\} \nonumber\\
=\sum_{i=4}^{7}c_{4i}&\Bigl(\frac{z}{w}\Bigr)
\Bigl\{\ell_{i'4}^\pm(z)f_{i4}^\pm(w)-\ell_{i'1}^\pm(z)f_{i4}^\pm(w)\Bigl[e_{14}^\pm(z)
-e_{12}^\pm(z)e_{24}^\pm(z)+e_{12}^\pm(z)e_{23}^\pm(z)e_{34}^\pm(z)-e_{13}^\pm(z)\nonumber\\
&\quad\cdot e_{34}^\pm(z)\Bigr]-\ell_{i'2}^\pm(z)f_{i4}^\pm(w)\Bigl(e_{24}^\pm(z)-e_{23}^\pm(z)e_{34}^\pm(z)\Bigr)-\ell_{i'3}^\pm(z)f_{i4}^\pm(w)e_{34}^\pm(z)\Bigr\}k_4^\pm(w)\nonumber.
  \end{align*}

Similarly, from $M_{14}=M_{14}', M_{24}=M_{24}', M_{34}=M_{34}', M_{41}=M_{41}', M_{42}=M_{42}'\;\text{as well as}\; M_{43}=M_{43}'$, we find that
   \begin{align*}
  \sum_{i=4}^{7}c_{ji}&\Bigl(\frac{z}{w}\Bigr)\Bigl\{\ell_{i'4}^\pm(z)f_{i4}^\pm(w)
  -\ell_{i'1}^\pm(z)f_{i4}^\pm(w)\Bigl[e_{14}^\pm(z)-e_{12}^\pm(z)e_{24}^\pm(z)-e_{13}^\pm(z)e_{34}^\pm(z)+e_{12}^\pm(z)e_{23}^\pm(z)\\
  &\quad\cdot e_{34}^\pm(z)\Bigr]-\ell_{i'2}^\pm(z)f_{i4}^\pm(w)\Bigl(e_{24}^\pm(z)-e_{23}^\pm(z)e_{34}^\pm(z)\Bigr)-\ell_{i'3}^\pm(z)f_{i4}^\pm(w)e_{34}^\pm(z)\Bigr\}k_4^\pm(w)\nonumber=0\nonumber,
   \end{align*}
   and
   \begin{align*}
  \sum_{i=4}^{7}c_{ij}&\Bigl(\frac{z}{w}\Bigr)k_4^\pm(w)\Bigl\{e_{4i}^\pm(w)\ell_{4i'}^\pm(z)
  -\Bigl[f_{41}^\pm(z)-f_{42}^\pm(z)f_{21}^\pm(z)-f_{43}^\pm(z)f_{31}^\pm(z)+f_{43}^\pm(z)f_{32}^\pm(z)f_{21}^\pm(z)\Bigr]\\
  &\quad \cdot e_{4i}^\pm(w)\ell_{1i'}^\pm(z)-\Bigl(f_{42}^\pm(z)-f_{43}^\pm(z)f_{32}^\pm(z)\Bigr)e_{4i}^\pm(w)\ell_{2i'}^\pm(z)-f_{43}^\pm(z)e_{4i}^\pm(w)\ell_{3i'}^\pm(z)\Bigr\}=0\nonumber,
   \end{align*}
   where $j=5, 6, 7$.

Using the above equations, we derive the desired equation.
\end{proof}

Finally, we need to check the $(r, s)$-Serre relations listed in Theorem \ref{thm relations} for $n=3$.
Here we only check the next Lemma since the others can be checked similarly as Proposition 4.20 in \cite{JingLJA 2017}.
\begin{lemm}\label{lemm b3 Serre}
One has
\noindent  \vspace{1em}\\ $Sym_{z_1, z_2,z_3}\Bigl\{X^{\pm}_{2}(w)X^{\pm}_{3}(z_1)X^{\pm}_{3}(z_2)X^{\pm}_{3}(z_3)-(r^{\pm2}+s^{\pm2}+r^\pm s^\pm)X^{\pm}_{3}(z_1)X^{\pm}_{2}(w)X^{\pm}_{3}(z_2)X^{\pm}_{3}(z_3)$\\

$\hskip2.3cm+\,(rs)^\pm(r^{\pm2}+s^{\pm2}+r^\pm s^\pm)X^{\pm}_{3}(z_1)X^{\pm}_{3}(z_2)X^{\pm}_{2}(w)X^{\pm}_{3}(z_3)$\\

$\hskip3.3cm-\,(rs)^{\pm3}X^{\pm}_{3}(z_1)X^{\pm}_{3}(z_2)X^{\pm}_{3}(z_3)X^{\pm}_{2}(w)\Bigr\}=0.$
\end{lemm}

\begin{proof}
  Using Lemmas \ref{lemm B3 X2X3} and \ref{lemma B3 X3+X3+}, it suffices to prove
    \begin{align*}
  &\sum_{\sigma\in S_3}\text{sgn}(\sigma)A\Bigl(rz_{\sigma(1)}-sz_{\sigma(2)}\Bigr)\Bigl(rz_{\sigma(1)}-sz_{\sigma(3)}\Bigr)\Bigl(rz_{\sigma(2)}-sz_{\sigma(3)}\Bigr)\\
&\qquad\Bigl\{\Bigl[rs^3z_{\sigma(1)}-\Bigl(r^3s+r^2s^2\Bigr)z_{\sigma(2)}+r^4z_{\sigma(3)}\Bigr]w^2\\
&\qquad\quad+\Bigl[s^4z_{\sigma(1)}z_{\sigma(2)}-\Bigl(rs^3+r^2s^2\Bigr)z_{\sigma(1)}z_{\sigma(3)}+r^3sz_{\sigma(2)}z_{\sigma(3)}\Bigr]w\Bigr\}=0,
\end{align*}
  where $$A=\frac{r^{-3}s^{-3}(s^2-r^2)}{\bigl(sz_3-rz_1\bigr)\bigl(sz_2-rz_1\bigr)\bigl(sz_3-rz_2\bigr)\bigl(w-z_1\bigr)\bigl(w-z_2\bigr)\bigl(w-z_3\bigr)}.$$
 By direct calculations, one can verify this identity.
\end{proof}

\subsection{From $B_3^{(1)}$ to $B_n^{(1)}$ by rank reduction}

Now we extend the rank $3$ case to the general rank $n$.
The idea of the Homomorphism Theorem by rank reduction due to Jing-Molev-Liu in \cite{JingLM SIGMA 2020} is still applicable to the two-parameter setting.
For convenience, we use the superscript $[n]$ for $\widehat{R}_{r, s}(z)$ to emphasize its corresponding rank.

Fix a positive integer $m$ such that $m<n$.
Suppose that the generators $\ell_{ij}^\pm(z)$ of the algebra $U(\widehat{R}_{r, s}^{[n-m]}(z))$ are labelled by the indices $m+1\leq i, j\leq (m+1)'$.
\begin{theorem}
  The mapping
  \begin{equation}\label{equ embedhom}
\ell_{ij}^\pm(z)\mapsto\left|\begin{array}{cccc}\ell^{\pm}_{11}(z)&\cdots&\ell^{\pm}_{1m}(z)&\ell^{\pm}_{1j}(z)\\ \cdots& \cdots &\cdots&\cdots\\ \ell_{m1}^\pm(z) &\cdots  &\ell_{mm}^\pm(z) &\ell_{mj}^\pm(z) \\ \ell^{\pm}_{i1}(z)&\cdots&\ell^{\pm}_{im}(z)&\boxed{\ell^{\pm}_{ij}(z)}\end{array}\right|, \quad m+1\leq i, j\leq (m+1)',
\end{equation}
defines a homomorphism $\psi_m: U(\widehat{R}_{r, s}^{[n-m]}(z))\rightarrow U(\widehat{R}_{r, s}^{[n]}(z))$.
\end{theorem}
\begin{proof}
  Similar to the proof in \cite{JingLM JMP 2020}.
\end{proof}

We will use the superscript to denote square submatrices corresponding to rows and columns labelled by $m+1, m+2, \cdots, (m+1)'$.
Then we set
$$F^{\pm [n-m]}(z)=\left(
                     \begin{array}{cccc}
                       1 & 0 & \cdots & 0 \\
                       f_{m+2, m+1}^\pm(z) & 1 & \cdots& 0 \\
                       \vdots & \vdots & \vdots & \vdots \\
                       f_{(m+1)', m+1}^\pm(z) & \cdots & f_{(m+1), (m+2)'}^\pm(z) & 1 \\
                     \end{array}
                   \right),
$$
$$E^{\pm[n-m]}(z)=\left(
                     \begin{array}{cccc}
                       1 & e_{m+1, m+2}^\pm(z)& \cdots & e_{m+1, (m+1)'}^\pm(z) \\
                       0 & 1 & \cdots& \vdots\\
                       \vdots & \vdots & \vdots & e_{(m+2)', (m+1)}^\pm(z) \\
                       0& 0 &\cdots & 1 \\
                     \end{array}
                   \right),
$$
and $K^{\pm [n-m]}(z)=\text{diag} [k_{m+1}^\pm(z), \cdots, k_{(m+1)'}^\pm(z)].$
Moreover, introduce the products of these matrices by
$$L^{\pm [n-m]}(z)=F^{\pm [n-m]}(z) K^{\pm [n-m]}(z) E^{\pm [n-m]}(z).$$
The entries of $L^{\pm [n-m]}(z)$ will be denoted by $\ell_{ij}^{\pm [n-m]}(z)$.

\begin{prop}
   The series $\ell_{ij}^{[n-m]}(z)$ coincides with the image of the generator series $\ell_{ij}^\pm(z)$ of $U(\widehat{R}_{r,s}^{[n-m]}(z))$ under the homomorphism (\ref{equ embedhom}),
   $$\ell_{ij}^{\pm [n-m]}(z)=\psi_m(\ell_{ij}^\pm(z)), \quad m+1\leq i, j\leq (m+1)'.$$
\end{prop}

\begin{coro}
  The following relations hold in $U(\widehat{R}_{r,s}^{[n]}(z)):$
  $$\widehat{R}_{r,s}^{[n-m]}(\frac{z}{w})L_1^{\pm [n-m]}(z)L_2^{\pm [n-m]}(w)=L_2^{\pm [n-m]}(w)L_1^{\pm [n-m]}(z)\widehat{R}_{r,s}^{[n-m]}(\frac{z}{w}),$$
  $$\widehat{R}_{r,s}^{[n-m]}(\frac{z_+}{w_-})L_1^{\pm [n-m]}(z)L_2^{\pm [n-m]}(w)=L_2^{\pm [n-m]}(w)L_1^{\pm [n-m]}(z)\widehat{R}_{r,s}^{[n-m]}(\frac{z_-}{w_+}).$$
\end{coro}

 By the Homomorphism Theorem by rank reduction, we can establish the $RLL$ formalism of $U(\widehat{R}_{r,s}^{[n]}(z))$ from the special case $n=3$,
 see Theorem \ref{thm relations}.

\subsection{The Ding-Frenkel Isomorphism Theorem for $U_{r, s}(\widehat{\mathfrak{so}_{2n+1}})$}\label{ssec isoaffine}
Based on the $RLL$ realization,
we can derive the Drinfeld realization  $\mathcal{U}_{r,s}\mathcal(\widehat {\frak{so}_{2n+1}})$ (\cite{HuCMP 2008,HuZhang2014,Zhang phd 2007}).

By direct calculations, we have
\begin{lemm}
  $\widehat{R}_{r, s}(z)$ satisfies the following crossing symmetry relation:
  \begin{equation}\label{}
  \widehat{R}_{r, s}(z)C_1\widehat{R}_{r, s}(\xi z)^{t_1}C_1^{-1}=\frac{(s^2z-r^2 )(\xi z-1)}{(1-z)(s^2-r^2\xi z)},
  \end{equation}
  where $\xi=(r^{-1}s)^{2n-1}$, and $t_1$ is the standard matrix transposition on the first tensor factor.
\end{lemm}

In accordance with the one-parameter case \cite{JingLM SIGMA 2020}, we denote $\widetilde{R}_{r, s}(z)=f(z)\widehat{R}_{r,s}(z)$ (see \cite{Frenkel CMP 1992}),
where $f(z)$ is uniquely determined by the relation
$$f(z)f(\xi z)=\frac{1}{(1-r^{-2}s^2z)(1-r^2s^{-2}z)(1-\xi z)(1-\xi^{-1}z)}.$$
Equivalently, we have
\begin{gather*}\label{fz}
f(z)=\prod_{t=0}^{\infty}\frac{\big(1-\xi^{2t}z\big)\big(1-r^{-2}s^2\xi^{2t+1}z\big)
\big(1- r^2s^{-2}\xi^{2t+1}z\big)\big(1-\xi^{2t+2}z\big)}
{\big(1-\xi^{2t-1}z\big)\big(1-\xi^{2t+1}z\big)
\big(1- r^2s^{-2}\xi^{2t}z\big)\big(1- r^{-2}s^2\xi^{2t}z\big)}.
\end{gather*}
\begin{defi}\label{def phirs}
Define the map $\Phi_{r, s}: \mathcal{U}_{r,s}\mathcal(\widehat {\frak{so}_{2n+1}})\longrightarrow \mathcal{U}(\widetilde{R}_{r, s}(z))$ as follows:
\begin{equation*}
\begin{aligned}
x^{\pm}_i(z)&\mapsto (r^2-s^2)^{-1}X^{\pm}_i\Bigl(z(rs^{-1})^i\Bigr), \\
x^{\pm}_n(z)&\mapsto (r-s)^{-1}(r^{-1}+s^{-1})^{-\frac{1}{2}}X^{\pm}_n\Bigl(z(rs^{-1})^n\Bigr), \\
\omega'_j(z)&\mapsto k^{+}_{j+1}\Bigl(z(rs^{-1})^j\Bigr)k^+_j\Bigl(z(rs^{-1})^j\Bigr)^{-1},\\
\omega_j(z)&\mapsto k^{-}_{j+1}\Bigl(z(rs^{-1})^{j}\Bigr)k^-_j\Bigl(z(rs^{-1})^{j}\Bigr)^{-1},
\end{aligned}
\end{equation*}
where $1\le i \le n-1$ and $1\le j\le n$.
Additionally, $\Phi_{r, s}$ fixes the central elements $\gamma, \gamma'$.
\end{defi}

\begin{prop}\label{prop Phiphi}
  The isomorphism $\phi_{r, s}$ between $U_{r, s}(\mathfrak{so}_{2n+1})$ and $U(\widehat{R}_{r, s})$
established in Theorem \ref{thm FRTISO} is obtained from $\Phi_{r,s}$
by taking the constant terms of the corresponding generating series,
or equivalently, by applying $\operatorname{Res}_{z=0} z^{-1}(\cdot)$ to these series,
up to a rescaling of generators.
\end{prop}

\begin{proof}
Indeed, upon taking the constant terms of the generating functions $x_i^\pm(z)$, $\omega_i(z)$, and $\omega'_i(z)$,
they specialize to the Chevalley generators of $U_{r, s}(\mathfrak{so}_{2n+1})$.
Furthermore, according to Proposition \ref{prop quasidet z=0},
the images of $\operatorname{Res}_{z=0} z^{-1}\Phi_{r,s}(\cdot)$ are consistent with those of $\phi_{r, s}$, up to a rescaling of generators.
For instance, assume that $1\leq i\leq n-1$,
\begin{align*}
  \operatorname{Res}_{z=0} z^{-1}\Phi_{r,s}(x_i^+(z)) =\;& (r^2-s^2)^{-1}\Bigl(e_{i, i+1}^+(0)-e_{i,i+1}^-(0)\Bigr)  \\
   =\; & (r^2-s^2)^{-1}\Bigl((\ell_{ii}^+[0])^{-1}\ell_{i, i+1}^+[0]-(\ell_{ii}^-[0])^{-1}\ell_{i, i+1}^-[0]\Bigr)  \\
   =\; & (r^2-s^2)^{-1}\Bigl((\ell_{ii}^+[0])^{-1}\ell_{i, i+1}^+[0]\Bigr)\\
    =\; & r^{-2}(r^2-s^2)^{-1}\Bigl(\ell_{i, i+1}^+[0](\ell_{ii}^+[0])^{-1}\Bigr)\\
   =\; & r^{-2}\phi_{r,s}(e_i),
\end{align*}
where the first equality follows from Definition \ref{def phirs} and Equation (\ref{equ Xi def}),
the second equality is due to Proposition \ref{prop quasidet z=0},
the third equality holds since $\ell_{i, i+1}^-[0]=0$,
and the fourth equality is due to $(\ell_{ii}^+[0])^{-1}\ell_{i, i+1}^+[0]=r^{-2}\ell_{i, i+1}^+[0](\ell_{ii}^+[0])^{-1}$, see (\ref{equ ellii++}).
This shows that $\operatorname{Res}_{z=0} z^{-1}\Phi_{r,s}(\cdot)$ agrees with $\phi_{r,s}$ up to a rescaling of generators.
\end{proof}

\begin{remark}
Strictly speaking, the map obtained by applying $\operatorname{Res}_{z=0} z^{-1}(\cdot)$ to $\Phi_{r,s}$
agrees with $\phi_{r,s}$ up to a rescaling of generators. In what follows, we identify them under this normalization.
\end{remark}

We will prove that $\Phi_{r,s}$ is an isomorphism. First, we will show it is a homomorphism (see Proposition \ref{prop DrinfeldRel}). Specifically, via $\Phi_{r,s}$ and the commutation relations in $\mathcal{U}(\widehat{R}_{r,s}(z))$ (see Theorem \ref{thm relations}), we can recover the Drinfeld realization of $\mathcal{U}_{r,s}(\widehat{\mathfrak{so}_{2n+1}})$,
which is originally defined in \cite{Zhang phd 2007}, and equivalent to Definition \ref{def affine} (see \cite{HuZhang2014}).

\begin{prop}[Drinfeld realization]\label{prop DrinfeldRel}
In \;$\mathcal{U}_{r,s}\mathcal(\widehat {\frak{so}_{2n+1}})$, the generating series $x^{\pm}_{i}(z), \omega_i(z),\omega'_i(z)$ $(1\le i\le n)$ satisfy
\begin{gather}
\Bigl[\omega'_i(z),\omega'_j(w)\Bigr]=0,\quad \Bigl[\omega_j(z),\omega_i(w)\Bigr]=0,\quad 1\le i, j\le n, \label{equ d1}\\
\omega'_i(z)\omega_j(w)=\frac{g_{ij}\Bigl(\frac{z_-}{w_+}\Bigr)}{g_{ij}\Bigl(\frac{z_+}{w_-}\Bigr)}\omega_j(w)\omega'_i(z),\quad 1\le i, j\le n, \label{equ d2}\\
\omega'_i(z)x^{\pm}_{j}(w)=g_{ij}\Bigl(\frac{z}{w_{\pm}}\Bigr)^{\pm}x^{\pm}_{j}(w)\omega'_i(z),\quad 1\le i, j\le n,\label{phi Drirea old}\\
\omega_i(z)x^{\pm}_{j}(w)=g_{ji}\Bigl(\frac{w_{\mp}}{z}\Bigr)^{\mp}x^{\pm}_{j}(w)\omega_i(z),\quad 1\le i, j\le n,\label{psi Drirea old}\\
x^{\pm}_{i}(z)x^{\pm}_{j}(w)=g_{ij}\Bigl(\frac{z}{w}\Bigr)^{\pm}x^{\pm}_{j}(w)x^{\pm}_{i}(z), \quad 1\le i, j\le n, \label{equ d3} \\
\Bigl[x^{+}_{i}(z),x^{-}_{j}(w)\Bigr]=\frac{\delta_{ij}}{r_i-s_i}\Bigl\{\delta\Bigl(\frac{z_{-}}{w_+}\Bigr)\omega_i(w_+){-} \delta\Bigl(\frac{z_{+}}{w_-}\Bigr)\omega'_j(z_+)\Bigr\}, \quad 1\le i, j\le n, \label{equ d4}\\
~~~Sym_{z_1, z_2}\Bigl\{(r_is_i)^{\pm1}x_i^{\pm}(z_1)
x_i^{\pm}(z_2)x_j^{\pm}(w)-(r_i^{\pm1}{+}s_i^{\pm1})\,x_i^{\pm}(z_1)x_j^{\pm}(w)x_i^{\pm}(z_2) \\
\hskip3cm +\, x_j^{\pm}(w)x_i^{\pm}(z_1)x_i^{\pm}(z_2)\Bigr\}=0, \quad\textit{for }
\ a_{ij}=-1 \ \textit{and}\  1\leqslant j < i \leqslant n; \nonumber\\
~~~Sym_{z_1, z_2}\Bigl\{x_i^{\pm}(z_1)
x_i^{\pm}(z_2)x_j^{\pm}(w)-(r_i^{\pm1}{+}s_i^{\pm1})\,x_i^{\pm}(z_1)x_j^{\pm}(w)x_i^{\pm}(z_2) \\
\hskip3cm +\, (r_is_i)^{\pm1}x_j^{\pm}(w)x_i^{\pm}(z_1)x_i^{\pm}(z_2)\Bigr\}=0, \quad\textit{for }
\ a_{ij}=-1 \ \textit{and}\  1\leqslant i< j \leqslant n;\nonumber\\
~~~Sym_{z_1, z_2,z_3}\Bigl\{x^{\pm}_{n-1}(w)x^{\pm}_{n}(z_1)x^{\pm}_{n}(z_2)x^{\pm}_{n}(z_3)-(r^{\pm2}+s^{\pm2}+r^\pm s^\pm)x^{\pm}_{n}(z_1)x^{\pm}_{n-1}(w)x^{\pm}_{n}(z_2)x^{\pm}_{n}(z_3) \\
\hskip1.5cm +\,(rs)^\pm(r^{\pm2}+s^{\pm2}+r^\pm s^\pm)x^{\pm}_{n}(z_1)x^{\pm}_{n}(z_2)x^{\pm}_{n-1}(w)x^{\pm}_{n}(z_3) \nonumber\\
\hskip3.5cm -\,(rs)^{\pm3}x^{\pm}_{n}(z_1)x^{\pm}_{n}(z_2)x^{\pm}_{n}(z_3)x^{\pm}_{n-1}(w)\Bigr\}=0,\nonumber
\end{gather}
where $z_+=zr^{\frac{c}{2}}$ and $z_-=zs^{\frac{c}{2}}$, we set the generating function $g^{\pm}_{ij}(z)=\sum\limits_{n\in\mathbb{Z}_{+}}c^{\pm}_{ijn}z^{n}$,
a formal power series in $z$, the expression is as follows:
$$g_{ij}(z)=\frac{\langle \omega^{\prime}_j,\omega_i\rangle z-\Bigl(\langle \omega^{\prime}_j,\omega_i\rangle \langle \omega^{\prime}_i,\omega_j\rangle^{-1}\Bigr)^{\frac{1}{2}}}{z-\Bigl(\langle \omega^{\prime}_i,\omega_j\rangle \langle \omega^{\prime}_j,\omega_i\rangle\Bigr)^{\frac{1}{2}}},$$
where $(\langle \omega_i',\,\omega_j \rangle)_{n\times n}$ is structure constants matrix.
\end{prop}

\begin{proof}
It suffices to verify this for $n=3$.
Using the explicit formula of the images of $\Phi_{r, s}$, we conclude the following:

$\bullet$ (\ref{equ d1}) and (\ref{equ d2}) are implied by  Lemmas \ref{lemm kikj} and \ref{lemm k4k4}.

$\bullet$ (\ref{phi Drirea old}) and (\ref{psi Drirea old}) can be derived using Lemmas \ref{lemm B3 k1X3}, \ref{lemm B3 k4X1}, \ref{lemm B3 kiXi},
\ref{lemm B3 ki+1xi}, \ref{lemma B3 k3X3}, and \ref{lemm b3 k4x3}.

$\bullet$ (\ref{equ d3}) is a consequence of  Lemmas \ref{lemm b3 x1x3}, \ref{lemm B3 X2X3}, \ref{lemm b3 x1x2}, \ref{lemm b3 x2+x3-},
and \ref{lemma B3 X3+X3+}.

$\bullet$ (\ref{equ d4}) follows from Lemma \ref{lemm xi+xi+}.

$\bullet$ The remaining equations are Serre relations, see Lemma \ref{lemm b3 Serre}.
\end{proof}

\begin{remark}
  Taking $r=q^{\frac{1}{2}},\; s=q^{-{\frac{1}{2}}}$, we can  recover that the Drinfeld realization of $U_q\mathcal(\widehat {\frak{so}_{2n+1}})$, \cite{Frenkel 1988}.
\end{remark}

\begin{theorem}\label{thm affine iso}
  $\Phi_{r, s}: \mathcal{U}_{r,s}\mathcal(\widehat {\frak{so}_{2n+1}})\longrightarrow \mathcal{U}(\widetilde{R}_{r, s}(z))$ defined in Definition \ref{def phirs} is an isomorphism.
\end{theorem}

\begin{proof}
 (I) We already established that $\Phi_{r, s}$ is a homomorphism in Proposition \ref{prop DrinfeldRel}.
 To verify $\Phi_{r, s}$ is bijective,
  our proof is based on the isomorphism $\Phi_q:  \mathcal{U}_q\mathcal(\widehat {\frak{so}_{2n+1}})\longrightarrow \mathcal{U}(\widetilde{R}_q(z))$ established by Jing-Liu-Molev in \cite{JingLM SIGMA 2020},
and the following commutative diagram:
\[\begin{tikzcd}
	{\mathcal{U}_{r,s}\mathcal(\widehat {\frak{so}_{2n+1}})} && {U(\widetilde{R}_{r,s})} \\
	\\
	{\mathcal{U}_{q, q^{-1}}\mathcal(\widehat {\frak{so}_{2n+1}})} && {U(\widetilde{R}_{q, q^{-1}})} \\
	\\
	{\mathcal{U}_q\mathcal(\widehat {\frak{so}_{2n+1}})} && {U(\widetilde{R}_q)}
	\arrow["{\Phi_{r,s}}", from=1-1, to=1-3]
	\arrow["{\pi_1}"', two heads, from=1-1, to=3-1]
	\arrow["{\pi'_1}", two heads, from=1-3, to=3-3]
	\arrow["{\Phi_{q, q^{-1}}}", from=3-1, to=3-3]
	\arrow["{\pi_2}"', two heads, from=3-1, to=5-1]
	\arrow["{\pi'_2}", two heads, from=3-3, to=5-3]
	\arrow["{\Phi_q}", from=5-1, to=5-3]
\end{tikzcd}\]
where $\pi_1$ and $\pi'_1$ are the specialization at $(r, s)=(q^{\frac{1}{2}}, q^{-\frac{1}{2}})$,
and $\pi_2$ is the canonical algebra morphism,
$$\mathcal{U}_{q, q^{-1}}\mathcal(\widehat {\frak{so}_{2n+1}})/\langle \omega_i\omega_i'-1,  \gamma\gamma'-1\mid 1\leq i\leq n\rangle\cong \mathcal{U}_q\mathcal(\widehat {\frak{so}_{2n+1}}).$$
According to \cite{HuZhang2014}, $\gamma\gamma'=(rs)^c$, where $c$ is the canonical central element of $\mathfrak{\hat{g}}$.
That is, $\pi:=\pi_2\pi_1$ is the canonical $\mathbb{Q}$-morphism, such that
$\text{ker}\pi=\langle rs-1, \omega_i\omega_i'-1 \mid 1\leq i\leq n\rangle$.
$\pi':=\pi'_2\pi'_1$ is also a canonical $\mathbb{Q}$-algebra morphism, since the commutation relations for generators $k_i^\pm(z)$ and $X_i^\pm(z)$ of $U(\widetilde{R}_q(z))$ (see Theorem 4.29 in \cite{JingLM SIGMA 2020}) form a specialization of those for $U(\widetilde{R}_{r,s}(z))$ (see Theorem \ref{thm relations}) at $(r, s)=(q^{\frac{1}{2}}, q^{-\frac{1}{2}})$.

One can easily verify that this diagram is indeed commutative on the generators of $\mathcal{U}_{r,s}(\widehat{\mathfrak{so}_{2n+1}})$.
For instance,
\begin{align*}
  \pi'\Phi_{r,s}(x_i^\pm(z)) &= \pi'\Bigl((rs)^{\frac{1}{2}}(r^2 - s^2)^{-1}X_i^\pm(z(rs^{-1})^i)\Bigr)  \\
   &= (q - q^{-1})^{-1}X_i^\pm(zq^i)  \\
   &= \Phi_q(x_i^\pm(z)) \\
   &= \Phi_q\pi(x_i^\pm(z)).
\end{align*}

(II) To characterize $\ker\pi'$, we use Proposition \ref{prop Phiphi}, which identifies
$\phi_{r,s}$ with the map obtained from $\Phi_{r,s}$ by applying
$\operatorname{Res}_{z=0} z^{-1}(\cdot)$ to the generating series.
It follows that
$\langle rs-1,\ \phi_{r,s}(\omega_i\omega_i')-1 \mid 1\leq i\leq n\rangle
\subseteq \ker \pi'.$
Moreover, by the construction of the canonical morphism $\pi'$ and by comparing
the defining relations of $\mathcal{U}(\widetilde{R}_q(z))$ (see Theorem 4.29 in \cite{JingLM SIGMA 2020})
and $\mathcal{U}(\widetilde{R}_{r,s}(z))$ (see Theorem \ref{thm relations}),
we conclude that the reverse inclusion also holds. Therefore,
\[
\ker\pi' = \left\langle rs-1,\ \Phi_{r,s}(\omega_i\omega_i')-1 \mid 1\leq i\leq n \right\rangle.
\]
That is, we have $\ker\pi' = \Phi_{r,s}(\ker\pi)$.

(III) Finally, it remains to prove that $\Phi_{r,s}$ is bijective if and only if $\text{ker}\pi'=\Phi_{r,s}(\text{ker}\pi)$.
The necessity is obvious.
Conversely, assume that $\text{ker}\pi'=\Phi_{r,s}(\text{ker}\pi)$.
If $\Phi_{r,s}$ is not surjective,
then there exists some $b\in U(\widetilde{R}_{r,s}(z))$ such that $b\notin \text{Im}\Phi_{r,s}$.
Since $\pi$ is surjective,
we can find some $a\in \mathcal{U}_{r,s}\mathcal(\widehat {\frak{so}_{2n+1}})$
such that $\pi(a)=\Phi_q^{-1}\pi'(b)$.
Accordingly, $\pi'\Phi_{r,s}(a)=\Phi_q\pi(a)=\pi'(b)$, implying that $\Phi_{r,s}(a)-b\in \text{ker}\pi'=\Phi_{r,s}(\text{ker}\pi)$.
Thus $b\in \text{Im}\Phi_{r,s}$, which is a contradiction.

Now we will prove $\Phi_{r, s}$ is injective.
According to the isomorphism theorem between the Drinfeld-Jimbo presentation $U_{r, s}(\widehat{\mathfrak{so}_{2n+1}})$ and Drinfeld realization $\mathcal{U}_{r, s}(\widehat{\mathfrak{so}_{2n+1}})$ \cite{Zhang phd 2007}, we have
\begin{align*}
  \Phi_{r, s}(\omega_0) &=\; \Phi_{r, s}(\gamma'^{-1}) \Phi_{r, s}(\omega_\theta^{-1})=\gamma'^{-1}\phi_{r, s}(\omega_\theta^{-1}),  \\
  \Phi_{r, s}(\omega'_0) & =\; \Phi_{r, s}(\gamma^{-1}) \Phi_{r, s}(\omega_\theta^{\prime-1})=\gamma^{-1}\phi_{r, s}(\omega_\theta^{\prime-1}),  \\
  \Phi_{r, s}(e_0) &=\;(rs)^{n-2}(r+s)^{-1}\gamma'^{-1}\Phi_{r, s}(x_\theta^-(1))\Phi_{r, s}(\omega_\theta^{-1}), \\
  \Phi_{r, s}(f_0) & =\;(rs)^{n-2}(r+s)^{-1}\gamma^{-1}\Phi_{r, s}(\omega_\theta^{\prime-1})\Phi_{r, s}(x_\theta^+(-1)),
\end{align*}
and by the use of the definition of quantum affine root vectors (\cite{HuCMP 2008, HuZhang2014, Zhang phd 2007}), we obtain that
\begin{align*}
  \Phi_{r, s}(x_\theta^-(1))& =[\Phi_{r, s}(f_2), \Phi_{r, s}(x^-_{\theta-\alpha_2}(1))]_{r^{-2}} \\
   & =[\phi_{r, s}(f_2), \cdots \phi_{r, s}(f_n), \phi_{r, s}(f_n), \cdots, \phi_{r, s}(f_2), \Phi_{r, s}(x_1^-(1))]_{(s^2, \cdots, s^2, rs, r^{-2}, \cdots, r^{-2})}.
\end{align*}
Moreover, according to  (\ref{equ Xi def}), (\ref{equ quasidet eij}), (\ref{equ quasidet fji}), and Definition \ref{def phirs},
we deduce that $\Phi_{r, s}(x_1^-(1))=(s^2-r^2)^{-1}(r^{-1}s)f_{21}^-(1)r^{-\frac{c}{2}}$.
Similarly, we have
\begin{align*}
  \Phi_{r, s}(x_\theta^+(-1))& =[\Phi_{r, s}(x^+_{\theta-\alpha_2}(-1)), \Phi_{r, s}(e_2)]_{s^{-2}} \\
   & =[\Phi_{r, s}(x_1^+(-1)), \phi_{r, s}(e_2), \cdots \phi_{r, s}(e_n),\phi_{r, s}(e_n), \cdots \phi_{r, s}(e_2)]_{(r^2, \cdots, r^2, rs, s^{-2}, \cdots, s^{-2})}.
\end{align*}
Then we obtain that $\Phi_{r, s}(x_1^+(-1))=(r^2-s^2)^{-1}(rs^{-1})e_{12}^+(-1)r^{\frac{c}{2}}$.
With the characterization of the $\text{ker}\,\pi$ and $\text{ker}\,\pi'$,
we conclude that $\Phi_{r, s}$ is injective on the real and imaginary root vectors.
Then $\text{Im}\, \Phi_{r, s}$ contains a subalgebra,
generated by $\Phi_{r, s}(e_i), \Phi_{r, s}(f_i), \Phi_{r, s}(\omega_i)$ and $\Phi_{r, s}(\omega'_i)\, (0\leq i\leq n)$
which is isomorphic to the Drinfeld-Jimbo presentation of $U_{r, s}(\widehat{\mathfrak{so}_{2n+1}})$.
Therefore, $\Phi_{r,s}$ is injective, and thus bijective.
\end{proof}

\begin{remark}
 The conditions $\omega_i\omega'_i=1$ and $rs=1$ can be described by a fashionable cross-disciplinary term from algebraic geometry: hyperbolic singularity.
Namely, the one-parameter quantum groups arise as degenerations of their two-parameter counterparts when the parameters lie at the hyperbolic singularity locus.
Since one-parameter quantum groups emerge as hyperbolic singularities of their two-parameter counterparts, a more sophisticated representation-theoretic perspective is required to precisely recover the two-parameter structures in the generic position, especially in affine cases.
\end{remark}

\section*{Appendix: The verification of the metric condition in $U_{r,s}(\mathfrak{so}_{2n+1})$}
This appendix is devoted to verifying that the distribution law of Lyndon bases in $L^+$ (see Theorem \ref{thm Lyn})
is compatible with the metric condition, namely $L^+ C(L^+ )^t C^{-1}=I$, see Theorem \ref{thm metric}.

For the proof of Theorem \ref{thm metric}, we need to rewrite those portions of $\mathcal{E}'_\gamma$ that we use in the subsequent proof in the form of $\mathcal{E}_\beta$, where $\gamma, \beta \in \Phi$.

\begin{lemm}\label{lemm sub eij e'ij}
  \begin{equation}\label{equ sub e'ij}
    \mathcal{E}'_{i-j, i}=\sum_{k=1}^{j}\sum_{t=0}^{k-1}(-1)^t \binom{k-1}{t} \Bigl(rs^{-1}\Bigr)^{2j-2k+2t}\sum_{\gamma_1\prec \cdots \prec \gamma_k} \mathcal{E}_{\gamma_1}\cdots\mathcal{E}_{\gamma_k}, \;\; 2\leq i\leq n+1,\; 1\leq j\leq i-1;
  \end{equation}

  \begin{align}\label{equ sub e'in'}
    \mathcal{E}'_{n+1-j, n'}=\; & (1-rs^{-1})\sum_{k=2}^{j} \sum_{t=0}^{k-1}(-1)^t \binom{k-1}{t} \Bigl(rs^{-1}\Bigr)^{2j-2k+2t} \sum_{\gamma_1\prec \cdots \prec \gamma_{k-1}} \mathcal{E}_{\gamma_1}\cdots\mathcal{E}_{\gamma_{k-1}}e_n^2  \nonumber\\
     &+\sum_{k=2}^{j} \sum_{t=0}^{k-1}(-1)^t \binom{k-1}{t} \Bigl(rs^{-1}\Bigr)^{2j-2k+2t+1}\sum_{\gamma_1\prec \cdots \prec \gamma_{k-2}} \mathcal{E}_{\gamma_1}\cdots\mathcal{E}_{\gamma_{k-2}}\mathcal{E}_{n-a, n+1} e_n \\
     & +\sum_{k=1}^{j-1} \sum_{t=0}^{k-1}(-1)^t \binom{k-1}{t} \Bigl(rs^{-1}\Bigr)^{2j-2k+2t}\sum_{\gamma_1\prec \cdots \prec \gamma_{k-1}} \mathcal{E}_{\gamma_1}\cdots\mathcal{E}_{\gamma_{k-1}}\mathcal{E}_{n-b, n'}, \quad 2\leq j\leq n; \nonumber
  \end{align}
  where $\binom{i}{j}$ is the usual binomial coefficient, and
  \begin{align}\label{equ sub e' n-i-1}
    \mathcal{E}'_{n-i-1, (n-i)'}=\;  &r^{4i+2}s^{-4i-2}\mathcal{E}_{n-i-1, (n-i)'} \nonumber \\
     &+\sum_{j=n-i+1}^{n}(-1)^{n-j-i}r^{2(n+i-j+1)}s^{-4i-2}(r^2-s^2)\mathcal{E}_{n-i-1, j'}\mathcal{E}_{n-i, j} \nonumber  \\
     &+(-1)^{i+1}r^{2i+1}s^{-4i-3}(r^2-s^2)\mathcal{E}_{n-i-1, n+1}\mathcal{E}_{n-i, n+1}  \nonumber\\
     &+\sum_{j=n-i+1}^{n}(-1)^{n-j-i}r^{2i-1}s^{2n-2j-4i-3}(r^2-s^2)\mathcal{E}_{n-i-1, j}\mathcal{E}_{n-i, j'} \\
     &+r^{2i+1}s^{-2i-5}(r^2-s^2)(r^{-2}-s^{-2})e_{n-i-1}e_{n-i}\mathcal{E}_{n-i, (n-i+1)'} \nonumber \\
     &+\sum_{k=2}^{i}(-1)^{k+1}r^{2i-1}s^{-2i-2k-1}(r^2-s^2)(s^{-2}-r^{-2})e_{n-i-1}\mathcal{E}_{n-i, n-i+k}\mathcal{E}_{n-i, (n-i+k)'} \nonumber \\
     &+(-1)^ir^{2i}s^{-4i-3}(r-s)(r^2-s^2)e_{n-i-1}\mathcal{E}_{n-i, n+1}^2, \quad 1\leq i\leq n-2. \nonumber
  \end{align}
\end{lemm}

\begin{proof}
  (I) Without loss of generality, we assume that $2\leq i\leq n$ since the case $i=n+1$ can be proved similarly.
  We use induction on $j$.
  If $j=1$, (\ref{equ sub e'ij}) is obviously true.
  Now we assume (\ref{equ sub e'ij}) holds for some $j\geq 1$.
  Now using the case $j$ equation and
  $$\mathcal{E}_{i-j, \alpha}e_{i-j-1}=s^{-2}e_{i-j-1}\mathcal{E}_{i-j, b}-s^{-2}\mathcal{E}_{i-j-1, \alpha}, \quad i-j< \alpha \leq n,$$
  we have
  \begin{align*}
    \mathcal{E}'_{i-j-1, i}= \;&e_{i-j-1}\mathcal{E}'_{i-j, i}-r^2\mathcal{E}'_{i-j, i}e_{i-j-1}  \\
    =\; &\sum_{k=1}^{j}\Bigl(1-r^2s^{-2}\Bigr) \sum_{t=0}^{k-1} (-1)^t \binom{k-1}{t} \Bigl(rs^{-1}\Bigr)^{2j-2k+2t}\sum_{\gamma_1\prec \cdots \prec \gamma_{k+1}} \mathcal{E}_{\gamma_1}\cdots\mathcal{E}_{\gamma_{k+1}} \\
    =\;  & \sum_{k=1}^{j+1} \sum_{t=0}^{k}(-1)^t \Bigl[\binom{k-1}{t}+\binom{k-1}{t-1}\Bigr]\Bigl(rs^{-1}\Bigr)^{2j-2k+2t} \sum_{\gamma_1\prec \cdots \prec \gamma_{k+1}}\mathcal{E}_{\gamma_1}\cdots\mathcal{E}_{\gamma_{k+1}} \\
     =\; &\sum_{k=1}^{j+1} \sum_{t=0}^{k} (-1)^t \binom{k}{t} \Bigl(rs^{-1}\Bigr)^{2(j+1)-2(k+1)+2t} \sum_{\gamma_1\prec \cdots \prec \gamma_{k+1}}\mathcal{E}_{\gamma_1}\cdots\mathcal{E}_{\gamma_{k+1}}.
  \end{align*}
  Here we use the well-known binomial identity $\binom{k}{t}=\binom{k-1}{t}+\binom{k-1}{t-1}$
  and define $\binom{k}{t}=0$ when $t<0$.
  We conclude that (\ref{equ sub e'ij}) holds for $1\leq j\leq i-1$.

  (II) From the definition and (\ref{equ sub e'ij}), we have
  \begin{align*}
    \mathcal{E}'_{n+1-j, n'}=\; & \mathcal{E}'_{n+1-j, n+1}e_n-rs e_n \mathcal{E}'_{n+1-j, n+1} \\
   =\; & \sum_{k=1}^{j}\sum_{t=0}^{k-1}(-1)^t \binom{k-1}{t} \Bigl(rs^{-1}\Bigr)^{2j-2k+2t}\sum_{\gamma_1\prec \cdots \prec \gamma_k} \mathcal{E}_{\gamma_1}\cdots\mathcal{E}_{\gamma_k}e_n \\
     & -rs \sum_{k=1}^{j}\sum_{t=0}^{k-1}(-1)^t \binom{k-1}{t} \Bigl(rs^{-1}\Bigr)^{2j-2k+2t} e_n\sum_{\gamma_1\prec \cdots \prec \gamma_k} \mathcal{E}_{\gamma_1}\cdots\mathcal{E}_{\gamma_k}.
  \end{align*}
  Then we need to adjust $e_n\sum\limits_{\gamma_1\prec \cdots \prec \gamma_k} \mathcal{E}_{\gamma_1}\cdots\mathcal{E}_{\gamma_k}$ to the standard normal order.
  Here we have the following two cases:

  (a1) If $\mathcal{E}_{\gamma_k}=\mathcal{E}_{n-1, n}=e_n$, then $\mathcal{E}_{\gamma_{k-1}}=\mathcal{E}_{n-a, n} \; (1\leq a \leq j-1)$.
  Since $e_n \mathcal{E}_{n-a, n}=s^{-2}\mathcal{E}_{n-a, n}e_n-s^{-2}\mathcal{E}_{a, n+1}$,
  we have $e_n \mathcal{E}_{\gamma_1}\cdots\mathcal{E}_{\gamma_k}
  =s^{-2}\mathcal{E}_{\gamma_1}\cdots\mathcal{E}_{\gamma_{k-2}}\mathcal{E}_{n-a, n}e_n^2
  -s^{-2}\mathcal{E}_{\gamma_1}\cdots\mathcal{E}_{\gamma_{k-2}}\mathcal{E}_{n-a, n+1}e_n.$

  (a2) If $\mathcal{E}_{\gamma_k}=\mathcal{E}_{n-b, n+1} \; (1\leq b \leq j-1)$.
  Since $e_n \mathcal{E}_{n-b, n+1}=r^{-1}s^{-1} \mathcal{E}_{n-b, n+1}e_n-r^{-1}s^{-1}\mathcal{E}_{n-b, n'}$,
  we have $e_n\mathcal{E}_{\gamma_1}\cdots\mathcal{E}_{\gamma_k}
  =r^{-1}s^{-1}\mathcal{E}_{\gamma_1}\cdots\mathcal{E}_{\gamma_k}e_n
  -r^{-1}s^{-1}\mathcal{E}_{\gamma_1}\cdots\mathcal{E}_{\gamma_{k-1}}\mathcal{E}_{n-b, n'}.$

  Based on the previous analysis, we have
  \begin{align*}
    \mathcal{E}'_{n+1-j, n'}=\; & (1-rs^{-1})\sum_{k=2}^{j} \sum_{t=0}^{k-1}(-1)^t \binom{k-1}{t} \Bigl(rs^{-1}\Bigr)^{2j-2k+2t} \sum_{\gamma_1\prec \cdots \prec \gamma_{k-1}} \mathcal{E}_{\gamma_1}\cdots\mathcal{E}_{\gamma_{k-1}}e_n^2  \\
     &+\Bigl\{rs^{-1} \sum_{t=0}^{k-1}(-1)^t \binom{k-1}{t}\Bigl(rs^{-1}\Bigr)^{2j-2k+2t} +\sum_{t=0}^{k-2}(-1)^t \binom{k-2}{t}\Bigl(rs^{-1}\Bigr)^{2j-2k+2t}\\
     &-rs (r^{-1}s^{-1})\sum_{t=0}^{k-2}(-1)^t \binom{k-2}{t}\Bigl(rs^{-1}\Bigr)^{2j-2k+2t}\Bigr\} \sum_{\gamma_1\prec \cdots \prec \gamma_{k-2}} \mathcal{E}_{\gamma_1}\cdots\mathcal{E}_{\gamma_{k-2}}\mathcal{E}_{n-a, n+1} e_n  \\
     & -rs (-r^{-1}s^{-1})\sum_{k=1}^{j-1} \sum_{t=0}^{k-1}(-1)^t \binom{k-1}{t} \Bigl(rs^{-1}\Bigr)^{2j-2k+2t}\sum_{\gamma_1\prec \cdots \prec \gamma_{k-1}} \mathcal{E}_{\gamma_1}\cdots\mathcal{E}_{\gamma_{k-1}}\mathcal{E}_{n-b, n'}.
  \end{align*}
  After some simple simplification, we obtain the desired expansion of $\mathcal{E}'_{n+1-j, n'}$.

  (III) We use induction on $i$. For $i=1$, by definition,
  $$\mathcal{E}'_{n-2, (n-1)'}=\mathcal{E}'_{n-2, n'}e_{n-1}-s^{-2}e_{n-1}\mathcal{E}'_{n-2, n'}.$$
  Also, from (\ref{equ sub e'in'}),
  \begin{align*}
    \mathcal{E}'_{n-2, n'}=\; & (1-rs^{-1})\Bigl\{ (r^2s^{-2}-r^4s^{-4})\mathcal{E}_{n-2, n}e_n^2+\Bigl(1-r^2s^{-2}\Bigr)^2e_{n-2}e_{n-1}e_n^2\Bigr\} \\
     & +\Bigl(r^3s^{-3}-r^5s^{-5}\Bigr)\mathcal{E}_{n-2, n+1}e_n+rs^{-1}\Bigl(1-r^2s^{-2}\Bigr)^2e_{n-2}\mathcal{E}_{n-1, n+1}e_n \\
     & +r^4s^{-4}\mathcal{E}_{n-2, n'}+(r^2s^{-2}-r^4s^{-4})e_{n-2}\mathcal{E}_{n-1, n'}.
  \end{align*}
  Then we have
  \begin{align*}
    \mathcal{E}'_{n-2, (n-1)'}=\; & r^{-1}s^{-3}\Bigl(1-r^2s^{-2}\Bigr)^2e_{n-2}e_{n-1}\mathcal{E}_{n-1, n'}
     +r^{-1}s^{-3}(r^2s^{-2}-r^4s^{-4})\mathcal{E}_{n-2, n}\mathcal{E}_{n-1, n'} \\
     &+(rs^{-3}-s^{-2})(r^2s^{-2}-r^4s^{-4})e_{n-2}\mathcal{E}_{n-1, n+1}^2\\
     &+rs^{-3}(r^4s^{-4}-r^2s^{-2})\mathcal{E}_{n-2, n+1}\mathcal{E}_{n-1, n+1} \\
     &+r^2s^{-2}(r^2s^{-2}-r^4s^{-4})\mathcal{E}_{n-2, n'}e_{n-1}+r^6s^{-6} \mathcal{E}_{n-2, (n-1)'}.
  \end{align*}
  This computation relies primarily on the following commutation relations (Some trivial relations are omitted):

  $\bullet$ Using (\ref{Lyn 22}) and (\ref{Lyn 23}), we have
  $$e_n^2e_{n-1}=s^{-4}e_{n-1}e_n^2-(s^{-4}+r^{-1}s^{-3})\mathcal{E}_{n-1, n+1}e_n+r^{-1}s^{-3}\mathcal{E}_{n-1, n'}.$$

  $\bullet$ From (\ref{Lyn 22}) and the Serre relation, we have
  \begin{align*}
    \mathcal{E}_{n-1, n+1}e_n e_{n-1}=\; & s^{-2}\mathcal{E}_{n-1, n+1}e_{n-1}e_n-s^{-2}\mathcal{E}_{n-1, n+1}^2 \\
    =\; & r^2s^{-2}e_{n-1}\mathcal{E}_{n-1, n+1}e_n-s^{-2}\mathcal{E}_{n-1, n+1}^2.
  \end{align*}

  Moreover, Serre relation implies $e_{n-1}\mathcal{E}_{n-2, n}=r^{-2}\mathcal{E}_{n-2, n}e_{n-1}.$

  $\bullet$ From Lemma 3.3(1) of \cite{HuWangJGP 2010}, we know
  $$e_{n-1}\mathcal{E}_{n-2, n+1}=\mathcal{E}_{n-2, n+1}e_{n-1}.$$
  Thus (\ref{equ sub e' n-i-1}) holds for $i=1$.

  Assume that (\ref{equ sub e' n-i-1}) is true for some $i\geq 1$.
  From definition, we have
  $$\mathcal{E}'_{n-i-2, (n-i-1)'}=\mathcal{E}'_{n-i-2, (n-i)'}e_{n-i-1}-s^{-2}e_{n-i-1}\mathcal{E}'_{n-i-2, (n-i)'}.$$
  According to Lemma 3.1 (4) of \cite{HuWangJGP 2010}, Definitions \ref{def Lynbas}, and \ref{def auto}, we have
  $$\mathcal{E}'_{n-i-2, (n-i)'}=e_{n-i-2}\mathcal{E}'_{n-i-1, (n-i)'}-r^2\mathcal{E}'_{n-i-1, (n-i)'}e_{n-i-2}.$$
  Then
  \begin{align}\label{equ E' n-i-2}
    \mathcal{E}'_{n-i-2, (n-i-1)'}=\; &e_{n-i-2}\mathcal{E}'_{n-i-1, (n-i)'}e_{n-i-1}-r^2\mathcal{E}'_{n-i-1, (n-i)'}e_{n-i-2}e_{n-i-1} \\
     &-s^{-2}e_{n-i-1}e_{n-i-2}\mathcal{E}'_{n-i-1, (n-i)'}+r^2s^{-2}e_{n-i-1}\mathcal{E}'_{n-i-1, (n-i)'}e_{n-i-2}. \nonumber
  \end{align}
  Now we can substitute the expression of $\mathcal{E}'_{n-i-1, (n-i)'}$ into (\ref{equ E' n-i-2}).
  Through some basic commutation relations, we can prove that (\ref{equ sub e' n-i-1}) is also true for $i+1$.
\end{proof}

\begin{theorem}\label{thm metric}
  The distribution law of Lyndon bases in $L^+$ (see Theorem \ref{thm Lyn})
is compatible with the metric condition, namely $L^+ C(L^+ )^t C^{-1}=I$.
\end{theorem}

\begin{proof}
Let $A^+_{j', i'}$ denote the $(j', i')$-entry of $L^+ C (L^+)^t$.
Obviously $A^+_{j', i'}=0$ when $j<i'$ and $A^+_{i, i'}=(r^{-1}s)^{\rho_i} \ell_{ii}^+ \ell_{i'i'}^+=(r^{-1}s)^{\rho_i}$.
Then it suffices to verify $A^+_{j', i'}=0$ for $j>i'$.
Without loss of generality, we assume $j=1'$.

(i) We first prove $A^+_{1, i'}=0$ when $2\leq i\leq n$:
\begin{equation}\label{}
  A^+_{1, i'}=\Big[\sum_{t=1}^{i}(r^{-1}s)^{\rho_t} B_{1, t}^+ B_{i', t'}^+ r^{2\delta_{1t}} \mathcal{E}_{1, t}\mathcal{E}'_{t, i}\Bigr] \omega'_{\epsilon_i}(\omega'_{\epsilon_1})^{-1},
\end{equation}
where $\delta_{1t}$ is the Kronecker delta,
and we adopt the notation that $B_{1, 1}^+=B_{i', i'}^+=\mathcal{E}_{1, 1}=\mathcal{E}'_{i, i}=1$.

Substituting the expression for $\mathcal{E}'_{t, i}$ from (\ref{equ sub e'ij}), we calculate the coefficient of $\mathcal{E}_{1, i}\omega'_{\epsilon_i}(\omega'_{\epsilon_1})^{-1}$ in $A^+_{1, i'}$:
\begin{align*}
   & (r^{-1}s)^{\rho_i} B_{1, i}^++(r^{-1}s)^{\rho_1} r^2 B_{i', 1'}^+ (rs^{-1})^{2i-4} \\
  =\; & (r^{-1}s)^\frac{2n+1-2i}{2}(r^2-s^2)+(r^{-1}s)^\frac{2n-1}{2}r^2\Bigl[-r^{1-i}s^{i-3}(r^2-s^2)\Bigr](rs^{-1})^{2i-4} \\
  =\; & 0.
\end{align*}
Here, we use the fact that $B_{i', 1'}^+ = (r^{-1}s)^{i-2} B_{2', 1'}^+ = -r^{1-i} s^{i-3} (r^2 - s^2)$, which follows from (\ref{equ ell j'k'}) and (\ref{equ B(i+1)' i'}).

In the case $3\leq i\leq n$, we further calculate the coefficient of $\mathcal{E}_{1, t} \mathcal{E}_{\gamma_1}\cdots \mathcal{E}_{\gamma_k}\omega'_{\epsilon_i}(\omega'_{\epsilon_1})^{-1}\; (\text{for}\;2\leq t\leq i-1,\; 1\leq k\leq i-t,\; \text{and}\;\gamma_1\prec \cdots \prec\gamma_k)$ in $A^+_{1, i'}$:
\begin{align}\label{equ A1i'=0}
   & (r^{-1}s)^{\rho_t}B_{1, t}^+ B_{i', t'}^+ \sum_{j=0}^{k-1}(-1)^j\binom{k-1}{j}\Bigl(rs^{-1}\Bigr)^{2i-2k-2t+2j} \\
   &+(r^{-1}s)^{\rho_1}r^2B_{i', 1'}^+ \sum_{j=0}^{k}(-1)^j\binom{k}{j}\Bigl(rs^{-1}\Bigr)^{2i-2k-4+2j} \nonumber \\
   =\;& (r^{-1}s)^{\rho_2}B_{i', 2'}^+\Bigl\{ (r^2-s^2) \sum_{j=0}^{k-1}(-1)^j \binom{k-1}{j} \Bigl(rs^{-1}\Bigr)^{2i-2k-4+2j} +s^2\sum_{j=0}^{k}(-1)^j\binom{k}{j}\Bigl(rs^{-1}\Bigr)^{2i-2k-4+2j}  \Bigr\} \nonumber \\
   =\;& (r^{-1}s)^{\rho_2}B_{i', 2'}^+\Bigl\{ \sum_{j=1}^{k-1} (-1)^j \Bigl[\binom{k}{j}-\binom{k-1}{j}-\binom{k-1}{j-1}\Bigr]r^{2i-2k+2j-4}s^{-2i+2k-2j+6} \Bigr\} \nonumber \\
  =\; & 0.\nonumber
\end{align}
For this calculation, we use the following facts:

$\bullet$\; We have $B_{1, t}^+ = B_{1, 2}^+=r^2-s^2\; (2\leq t\leq n)$ from (\ref{equ ell kj});

$\bullet$\; We have  $B_{i', t'}^+ = (r^{-1}s)^{i-t-1} B_{(t+1)', t'}^+\; (t< i\leq n)$ from (\ref{equ ell j'k'});

$\bullet$\; We have  $B_{i', 2'}^+ = (r^{-1}s)^{i-3} B_{3', 2'}^+ = (r^{-1}s)^{i-3} B_{(t+1)', t'}^+$ from (\ref{equ ell j'k'}) and (\ref{equ B(i+1)' i'}).

Based on the previous calculations, we conclude that $A_{1, i'}^+=0$ for $2\leq i\leq n$.

(ii) We next prove $A_{1, n+1}^+=0$. We have
\begin{equation}\label{}
  A_{1, n+1}^+=\Bigl\{B_{1, n+1}^+\mathcal{E}_{1, n+1}+r^{-1}s^{-1}\Big[\sum_{t=1}^{n}(r^{-1}s)^{\rho_t} B_{1, t}^+ B_{n+1, t'}^+ r^{2\delta_{1t}} \mathcal{E}_{1, t}\mathcal{E}'_{t, n+1}\Bigr]\Bigr\}(\omega'_{\epsilon_1})^{-1}.
\end{equation}
Substituting the expression for $\mathcal{E}'_{t, n+1}$ from (\ref{equ sub e'ij}),
we first calculate the coefficient of $\mathcal{E}_{1, n+1}(\omega'_{\epsilon_1})^{-1}$:
\begin{align*}
   &B_{1, n+1}^++(r^{-1}s)^{\rho_1}rs^{-1}B_{n+1, 1'}^+(rs^{-1})^{2n-2}  \\
   =\; &(rs)^{-\frac{1}{2}} (r+s)^{\frac{1}{2}}(r-s)+(r^{-1}s)^{\frac{2n-1}{2}}rs^{-1}\Bigl[-r^{-n}s^{n-1}(r+s)^{\frac{1}{2}}(r-s)\Bigr](rs^{-1})^{2n-2}  \\
   =\;& 0.
\end{align*}

Here, we rely on the following steps:

$\bullet$ From (\ref{ell k, n+1}), (\ref{ell n-1, n+1}), and (\ref{equ Bn n+1}), we have $B_{1, n+1}^+ = B_{n, n+1}^+ = (rs)^{-\frac{1}{2}} (r+s)^{\frac{1}{2}} (r-s)$;

$\bullet$ From (\ref{equ ell n+1, k'}) and (\ref{equ Bn+1, n'}), we obtain $B_{n+1, 1'}^+ = r^{-1} s^2 (r+s)^{-\frac{1}{2}} B_{n', 1'}^+$;

$\bullet$ Via recursive application of (\ref{equ ell j'k'}), we have $B_{n', 1'}^+ = (r^{-1}s)^{n-2} B_{2', 1'}^+$;

$\bullet$ By (\ref{equ B(i+1)' i'}), substituting $B_{2', 1'}^+ = -r^{-1} s^{-1} (r^2 - s^2)$ gives:
   $$B_{n+1, 1'}^+ = r^{-1} s^2 (r+s)^{-\frac{1}{2}} \cdot (r^{-1}s)^{n-2} \cdot \left[-r^{-1} s^{-1} (r^2 - s^2)\right] = -r^{-n} s^{n-1} (r+s)^{\frac{1}{2}} (r-s).$$

We also calculate the coefficient of $\mathcal{E}_{1, t} \mathcal{E}_{\gamma_1}\cdots \mathcal{E}_{\gamma_k}(\omega'_{\epsilon_1})^{-1}\; (\text{for}\;\gamma_1\prec \cdots \prec\gamma_k,\; 1\leq t\leq n,\; \text{and}\; 1\leq k\leq n+1-t)$ in $A^+_{1, n+1}$:
\begin{align*}\label{}
   & (r^{-1}s)^{\rho_t}r^{-1}s^{-1}B_{1, t}^+ B_{n+1, t'}^+ \sum_{j=0}^{k-1}(-1)^j\binom{k-1}{j}\Bigl(rs^{-1}\Bigr)^{2(n+1)-2k-2t+2j} \\
   &+(r^{-1}s)^{\rho_1}rs^{-1}B_{n+1, 1'}^+ \sum_{j=0}^{k}(-1)^j\binom{k}{j}\Bigl(rs^{-1}\Bigr)^{2(n+1)-2k-4+2j}  \\
   =\;& (r^{-1}s)^{\rho_2}B_{n+1, 2'}^+r^{-1}s^{-1}\Bigl\{ (r^2-s^2) \sum_{j=0}^{k-1}(-1)^j \binom{k-1}{j} \Bigl(rs^{-1}\Bigr)^{2(n+1)-2k-4+2j}\\
    &\hspace{11em}+s^2\sum_{j=0}^{k}(-1)^j\binom{k}{j}\Bigl(rs^{-1}\Bigr)^{2(n+1)-2k-4+2j}  \Bigr\}  \\
  =\; & 0.
\end{align*}
Here we rely on the following steps:

$\bullet$ From (\ref{equ ell n+1, k'}) and (\ref{equ ell n+1 n-1'}), we have $B_{n+1, t'}^+=(rs^{-1})^{t-1}B_{n+1, 1'}^+$.

$\bullet$ Therefore, we have $(r^{-1}s)^{\rho_t} B_{n+1, t'}^+(r^{-1}s)^{2t-4}=(r^{-1}s)^{\rho_2} B_{n+1, 2'}^+$.

$\bullet$ Finally we can prove the expression is equal to zero essentially due to (\ref{equ A1i'=0}).

Based on the previous calculations, we conclude that $A_{1, n+1}^+=0$.

(iii) Next, we prove $A_{1, n}^+=0$.
We have
\begin{align*}
  A_{1, n}^+=\; & \Bigl\{ s^{-2}\Bigl(r^{-1}s\Bigr)^{\rho_1} B_{n, 1'}^+\mathcal{E}'_{1, n'}  +\sum_{t=2}^{n-1}r^{-2}s^{-2}\Bigl(r^{-1}s\Bigr)^{\rho_t} B_{1, t}^+ B_{n, t'}^+ \mathcal{E}_{1, t}\mathcal{E}'_{t, n'}  \\
   & +r^{-2}s^{-2}\Bigl(r^{-1}s\Bigr)^{\rho_n} B_{1, n}^+ B_{n, n'}^+\mathcal{E}_{1, n}e_n^2  +r^{-1}s^{-1}B_{1, n+1}^+ B_{n, n+1}^+\mathcal{E}_{1, n+1}e_n \\
   & + \Bigl(rs^{-1}\Bigr)^{\rho_n} B_{1, n'}^+\mathcal{E}_{1, n'}\Bigr\} (\omega'_{\epsilon_1}\omega'_n)^{-1}.
\end{align*}

Substituting the expression for $\mathcal{E}'_{t, n'}$ from (\ref{equ sub e'in'}),
we first calculate the coefficient of $\mathcal{E}_{1, n'}(\omega'_{\epsilon_1}\omega'_n)^{-1}$:
\begin{align*}
   &\Bigl(rs^{-1}\Bigr)^{\rho_n}B_{1, n'}^++ s^{-2}\Bigl(r^{-1}s \Bigr)^{\rho_1}B_{n, 1'}^+ \Bigl(rs^{-1}\Bigr)^{2n-2}   \\
   =\; &\Bigl(rs^{-1}\Bigr)^\frac{1}{2}\Bigl(-r^{-\frac{3}{2}}s^{-\frac{1}{2}}(r-s)\Bigr)+s^{-2}\Bigl(r^{-1}s \Bigr)^\frac{2n-1}{2} \Bigl(r^{-\frac{2n-1}{2}}s^{\frac{2n-1}{2}}(r-s)\Bigr)\Bigl(rs^{-1}\Bigr)^{2n-2}  \\
   =\;& 0.
\end{align*}
For this calculation, we use the following facts:

$\bullet$ We have $B_{1, n'}^+ = -r^{-\frac{3}{2}} s^{-\frac{1}{2}} (r-s)$ from (\ref{equ ell k n'}).

$\bullet$ We have $B_{n, 1'}^+=r^{-\frac{2n-1}{2}}s^{\frac{2n-1}{2}}(r-s)$ from (\ref{equ ell n k'}).

Similarly, one can calculate the coefficients of $\mathcal{E}_{1, n+1}e_n(\omega'_{\epsilon_1}\omega'_n)^{-1}$
and $\mathcal{E}_{1, n}e_n^2(\omega'_{\epsilon_1}\omega'_n)^{-1}$ are equal to zero.

We also calculate the coefficient of $\mathcal{E}_{1, t}\mathcal{E}_{\gamma_1}\cdots \mathcal{E}_{\gamma_{k-1}} e_n^2(\omega'_{\epsilon_1}\omega'_n)^{-1}$ for $k\geq 2$:
\begin{align*}
   & s^{-2}\Bigl(r^{-1}s\Bigr)^{\rho_1} B_{n, 1'}^+\Bigl(1-rs^{-1}\Bigr)\sum_{j=0}^{k}(-1)^j\binom{k}{j} \Bigl(rs^{-1}\Bigr)^{2n-2k-2+2j}\\
   & +r^{-2}s^{-2}\Bigl(r^{-1}s\Bigr)^{\rho_t} B_{1, t}^+ B_{n, t'}^+ \Bigl(1-rs^{-1}\Bigr)\sum_{j=0}^{k-1}(-1)^j\binom{k-1}{j} \Bigl(rs^{-1}\Bigr)^{2n-2k+2+2j-2t} \\
  =\; & \Bigl(1-rs^{-1}\Bigr)(r-s)r^{-2n+1}s^{2n-5}\Bigl\{(r^2-s^2) \sum_{j=0}^{k-1}(-1)^j \binom{k-1}{j} \Bigl(rs^{-1}\Bigr)^{2n-2k+2j-2}\\
    &\hspace{14em}+s^2\sum_{j=0}^{k}(-1)^j\binom{k}{j}\Bigl(rs^{-1}\Bigr)^{2n-2k+2j-2}  \Bigr\}  \\
  =\; & 0.
\end{align*}
Here we rely on the following steps:

$\bullet$ From (\ref{equ ell n k'}) and (\ref{equ ell n+1, k'}), we have
$$B_{n, j'}^+=\frac{(rs)^\frac{1}{2}(r+s)^{\frac{1}{2}}(r-s)}{s^2-r^2}B_{n+1, j'}^+, \quad B_{n+1, j'}^+=\frac{-r^{-1}s^2(r+s)^\frac{1}{2}(r-s)}{s^2-r^2}B_{n', j'}^+.$$

$\bullet$ From (\ref{equ B(i+1)' i'}), we have
\begin{align*}
  B_{n, j'}^+=\; & r^{-\frac{1}{2}}s^\frac{5}{2}(r-s) B_{n', j'}^+\\
  =\; &r^{-\frac{1}{2}}s^\frac{5}{2}(r-s)\Bigl(r^{-1}s\Bigr)^{n-j-1}B_{(j+1)', j'}^+  \\
  =\; & r^{-n+j-\frac{1}{2}}s^{n-j+\frac{1}{2}}(r-s).
\end{align*}

$\bullet$ Finally we can prove the expression is equal to zero essentially due to (\ref{equ A1i'=0}).

Similarly we can calculate the coefficients of $\mathcal{E}_{1, j}\mathcal{E}_{\gamma_1}\cdots \mathcal{E}_{\gamma_{k-2}}\mathcal{E}_{n-a, n+1}e_n(\omega'_{\epsilon_1}\omega'_n)^{-1}\; (k\geq 3)$,
and $\mathcal{E}_{1, j}\mathcal{E}_{\gamma_1}\cdots \mathcal{E}_{\gamma_{k-1}}\mathcal{E}_{n-b, n'}\; (k\geq 2)$
are equal to zero.

Based on the previous calculations, we conclude that $A_{1, n}^+=0$.

(iv) Next, we prove $A_{1, n-i}^+=0$ for $1\leq i\leq n-2$. We have
\begin{align*}
  A_{1, n-i}^+=\; &\Bigl\{ \sum_{j=n-i}^{n}\Bigl(rs^{-1}\Bigr)^{\rho_j} B_{1, j'}^+ B_{n-i, j}^+ \mathcal{E}_{1, j'}\mathcal{E}_{n-i, j}+r^{-1}s^{-1}B_{1, n+1}^+B_{n-i, n+1}^+ \mathcal{E}_{1, n+1}\mathcal{E}_{n-i, n+1} \\
   & +\sum_{j=n-i+1}^{n}r^{-2}s^{-2}\Bigl(r^{-1}s\Bigr)^{\rho_j} B_{1, j}^+ B_{n-i, j'}^+ \mathcal{E}_{1, j}\mathcal{E}_{n-i, j'}\\
   &+r^{-2}s^{-2}\Bigl(r^{-1}s\Bigr)^{\rho_{n-i}}B_{1, n-i}^+ B_{n-i, (n-i)'}^+ \mathcal{E}_{1, n-i}\theta_{n-i, (n-i)'}^+ \\
   &+\sum_{j=i+1}^{n-2}r^{-2}s^{-2}\Bigl(r^{-1}s\Bigr)^{\rho_{n-j}}B_{1, n-j}^+B_{n-i, (n-j)'}^+\mathcal{E}_{1, n-j}\mathcal{E}'_{n-j, (n-i)'} \\
   &+\Bigl(r^{-1}s\Bigr)^{\rho_1}s^{-2}B_{n-i, 1'}^+ \mathcal{E}'_{1, (n-i)'}\Bigr\} (\omega'_{\epsilon_1}\omega'_{\epsilon_{n-i}} )^{-1}.
\end{align*}

According to  Lemmas 3.7 (1) and 3.6 (8) of \cite{HuWangJGP 2010}, Definitions \ref{def Lynbas}, and \ref{def auto}, we have
\begin{equation}\label{equ E1k'}
  \mathcal{E}_{1, k'}\mathcal{E}_{1, k}=r^{-2}s^{-2}\mathcal{E}_{1, k}\mathcal{E}_{1, k'}-s^{-2}\mathcal{E}_{1, (k+1)'}\mathcal{E}_{1, k+1}+r^{-4}s^{-2}\mathcal{E}_{1, k+1}\mathcal{E}_{1, (k+1)'}, \quad 1\leq k\leq n-1.
\end{equation}
\begin{equation}\label{equ E1n'}
  \mathcal{E}_{1, n'}\mathcal{E}_{1, n}=r^{-2}s^{-2}\mathcal{E}_{1, n}\mathcal{E}_{1, n'}-r^{-1}s^{-2}(r-s)\mathcal{E}_{1, n+1}^2.
\end{equation}
Combining with the definition of $\theta_{n-i, (n-i)'}^+$, see (\ref{equ thetaii'}), we conclude that
\begin{align*}
  \theta_{n-i, (n-i)'}^+=\; & (r^{-2}s^{-2}-s^{-4})e_{n-i}\mathcal{E}_{n-i, (n-i+1)'} \\
   &+\sum_{k=2}^{i}(-1)^{k-1}r^{-2}s^{-2k+2}(s^{-2}-r^{-2})\mathcal{E}_{n-i, n-i+k}\mathcal{E}_{n-i, (n-i+k)'} \\
   &+(-1)^i r^{-1}s^{-2i}(r-s)\mathcal{E}_{n-i, n+1}^2, \quad 1\leq i\leq n-1.
\end{align*}

We first calculate the coefficient of $\mathcal{E}_{1, (n-i)'}$.
This term appears twice in $A_{1, n-i}^+$:
both in the first sum (by setting $j = n-i$) and in the expression of $\mathcal{E}'_{1, (n-i)'}$.
By Lemma 3.1 of \cite{HuWangJGP 2010}, we have the following recursive commutation relations:
\begin{equation}\label{equ recursive 1}
  \mathcal{E}'_{n-j-1, (n-i)'}=e_{n-j-1}\mathcal{E}'_{n-j, (n-i)'}-r^2\mathcal{E}'_{n-j, (n-i)'}e_{n-j-1},
\end{equation}
\begin{equation}\label{equ recursive 2}
  \mathcal{E}_{\gamma_2}\cdots\mathcal{E}_{\gamma_t}e_{n-j-1}=s^{-2}e_{n-j-1}\mathcal{E}_{\gamma_2}\cdots\mathcal{E}_{\gamma_t}-s^{-2}\mathcal{E}_{\gamma'_2}\cdots\mathcal{E}_{\gamma_t},
\end{equation}
where $\mathcal{E}_{\gamma'_2}$ is some quantum root vector satisfying $\gamma'_2\prec \gamma_3\prec\cdots \prec \gamma_t$.

From (\ref{equ sub e' n-i-1}), we know the coefficient of $\mathcal{E}_{n-i-1, (n-i)'}$ in $\mathcal{E}'_{n-i-1, (n-i)'}$ is $r^{4i+2}s^{-4i-2}$.

Substituting (\ref{equ recursive 2}) into (\ref{equ recursive 1}), we conclude that the coefficient of $\mathcal{E}_{1, (n-i)'}$ in $\mathcal{E}'_{1, (n-i)'}$ is $\Bigl(r^2s^{-2}\Bigr)^{n-i-2}r^{4i+2}s^{-4i-2}$.
Now the coefficient of $\mathcal{E}_{1, (n-i)'}$ in $A_{1, n-i}^+$ is
\begin{align*}
   & \Bigl(rs^{-1}\Bigr)^{\rho_{n-i}}B_{1, (n-i)'}^++\Bigl(r^{-1}s\Bigr)^{\rho_1}s^{-2}B_{n-i, 1'}^+ \Bigl(r^2s^{-2}\Bigr)^{n-i-2}r^{4i+2}s^{-4i-2}\\
  =\; & \Bigl(rs^{-1}\Bigr)^{i+\frac{1}{2}}(-1)^{i+1}r^{i-\frac{3}{2}}s^{i-\frac{1}{2}}(r-s)\\
  &+\Bigl(r^{-1}s\Bigr)^{n-\frac{1}{2}}s^{-2}(-1)^ir^{-n+\frac{1}{2}}s^{n+2i-\frac{1}{2}}(r-s)r^{2n-2i-4}s^{-2n+2i+4}r^{4i+2}s^{-4i-2}\\
  =\;&(-1)^i(r-s)(r^{2i-1}s^{-1}-r^{2i-1}s^{-1})\\
  =\;&0.
\end{align*}

Here we rely on the following facts:

$\bullet$ From (\ref{equ ell kj'}), we obtain that $B_{1, (n-i)'}^+=(-rs)^i B_{1, n'}^+=(-1)^{i+1}r^{i-\frac{3}{2}}s^{i-\frac{1}{2}}(r-s)$.

$\bullet$ From (\ref{equ ell j-1, k'}), we obtain that $B_{n-i, 1'}^+=\Bigl(-s^2\Bigr)^i B_{n, 1'}^+=(-1)^i r^{-n+\frac{1}{2}}s^{n+2i-\frac{1}{2}}(r-s)$.

Similarly, we can verify that the coefficients of the following terms all vanish by analogous arguments:

$\bullet \; \mathcal{E}_{1, j'}\mathcal{E}_{n-i, j}\; (n-i+1\leq j\leq n);$
\vspace{0.5em}

$\bullet \; \mathcal{E}_{1, n+1}\mathcal{E}_{n-i, n+1};$
\vspace{0.5em}

$\bullet \;\mathcal{E}_{1, j}\mathcal{E}_{n-i, j'} (n-i+1\leq j\leq n);$
\vspace{0.5em}

$\bullet \;\mathcal{E}_{1, n-i}e_{n-i}\mathcal{E}_{n-i, (n-i+1)'};$
\vspace{0.5em}

$\bullet \;\mathcal{E}_{1, n-i}\mathcal{E}_{n-i, n-i+k}\mathcal{E}_{n-i, (n-i+k)'}\; (2\leq k\leq i);$
\vspace{0.5em}

$\bullet  \;\mathcal{E}_{1, n-i}\mathcal{E}_{n-i, n+1}^2$.

Consider the coefficient of $\mathcal{E}_{1, n-j}\mathcal{E}_{\gamma_2}\cdots\mathcal{E}_{\gamma_t}\; (\mathcal{E}_{1, n-j}\prec \mathcal{E}_{\gamma_2}\prec \cdots \prec \mathcal{E}_{\gamma_t},\; i+1\leq j\leq n-2,\; t\geq 2)$.
Assume that the coefficient of $\mathcal{E}_{\gamma_2}\cdots\mathcal{E}_{\gamma_t}$ in the expression of $\mathcal{E}'_{n-j, (n-i)'}$ is $C_{n-j, \gamma_2\cdots \gamma_t}$.
Using (\ref{equ recursive 1}) and (\ref{equ recursive 2}) similarly,
we conclude that the coefficient of $e_{n-j-1}\mathcal{E}_{\gamma_2}\cdots\mathcal{E}_{\gamma_t}$ in $\mathcal{E}'_{n-j-1, (n-i)'}$ is $(1-r^2s^{-2})C_{n-j, \gamma_2\cdots \gamma_t}$.
Moreover, by applying these commutation relations recursively, the coefficient of $\mathcal{E}_{1, n-j}\mathcal{E}_{\gamma_2}\cdots\mathcal{E}_{\gamma_t}$ in $\mathcal{E}'_{1, (n-i)'}$ is $(r^2s^{-2})^{n-2-j}(1-r^2s^{-2})C_{n-j, \gamma_2\cdots \gamma_t}$.
Thus the coefficient of $\mathcal{E}_{1, n-j}\mathcal{E}_{\gamma_2}\cdots\mathcal{E}_{\gamma_t}$ in $A_{1, n-i}^+$ is
\begin{align*}
   & r^{-2}s^{-2}\Bigl(r^{-1}s\Bigr)^{\rho_{n-j}}B_{1, n-j}^+B_{n-i, (n-j)'}^+(1-r^2s^{-2})(r^2s^{-2})^{n-2-j}C_{n-j, \gamma_2\cdots \gamma_t} \\
   & +\Bigl(r^{-1}s\Bigr)^{\rho_1}s^{-2}B_{n-i, 1'}^+(1-r^2s^{-2})(r^2s^{-2})^{n-2-j}C_{n-j, \gamma_2\cdots \gamma_t} \\
  =\;&(r-s)(r^{-1}s)^{2j}(-s^2)^{i-2}(r^2s^{-2})^{n-2-j}(1-r^2s^{-2})C_{n-j, \gamma_2\cdots \gamma_t}\Bigl\{ r^{-1}s^{-3}-r^{-3}s^5+r^{-3}s^5(1-r^2s^{-2})\Bigr\} \\
  =\; & 0.
\end{align*}
Here we rely on the following fact:

$\bullet$ From (\ref{equ ell j-1, k'}), we derive
$B_{n-i, (n-j)'}^+=(-s^2)^i B_{n, (n-j)'}^+,\; 1\leq j\leq n-1, \; 1<i<j$, then
$$B_{n-i, (n-j)'}^+=(-1)^i r^{-j-\frac{1}{2}}s^{2i+j+\frac{1}{2}}(r-s).$$
In particular, $B_{n-i, 1'}^+=(-1)^i r^{-n+\frac{1}{2}}s^{n+2i-\frac{1}{2}}(r-s).$

Based on the previous calculations, we conclude that $A_{1, n-i}^+=0$ for $1\leq i\leq n-2$.

(v) Finally, we prove $A_{1, 1}^+=0$.
We have \begin{align*}
         A_{1, 1}^+=\; &\Bigl\{ \Bigl(rs^{-1}\Bigr)^{\rho_1} B_{1, 1'}^+\theta_{1, 1'}^+ +\sum_{i=2}^{n}\Bigl(rs^{-1}\Bigr)^{\rho_i}r^2B_{1, i'}^+B_{1, i}^+\mathcal{E}_{1, i'}\mathcal{E}_{1, i}+rs^{-1} \Bigl(B_{1, n+1}^+\mathcal{E}_{1, n+1}\Bigr)^2  \\
           &+\sum_{i=2}^{n}\Bigl(r^{-1}s\Bigr)^{\rho_i}s^{-2}B_{1, i'}^+B_{1, i}^+\mathcal{E}_{1i}\mathcal{E}_{1i'}+\Bigl(r^{-1}s\Bigr)^{\rho_1}r^2s^{-2} B_{1, 1'}^+\theta_{1, 1'}^+ \Bigr\}(\omega'_{\epsilon_1})^{-2}.
        \end{align*}
By the definition, $\theta_{11'}^+=\mathcal{E}_{1, 2'}e_1-r^{-4}e_1\mathcal{E}_{1, 2'}$.
Thus we have to change $\mathcal{E}_{1, i'}\mathcal{E}_{1, i}\;(2\leq i\leq n)$ to the standard normal order.
According to (\ref{equ E1k'}) and (\ref{equ E1n'}),
we first calculate the coefficient of $e_1\mathcal{E}_{1, 2'}$.
\begin{align*}
   & \Bigl[\Bigl(rs^{-1}\Bigr)^{\rho_1}+  \Bigl(r^{-1}s\Bigr)^{\rho_1} \Bigr]B_{1, 1'}^+ \Bigl(r^{-2}s^{-2}-r^{-4}\Bigr)
    +\Bigl[ \Bigl(rs^{-1}\Bigr)^{\rho_2}r^2\Bigl(r^{-2}s^{-2}\Bigr)+(r^{-1}s\Bigr)^{\rho_2}s^{-2}\Bigr]B_{1, 2'}^+B_{1, 2}^+ \\
   & =(-1)^n(r-s)(r^2-s^2)\Bigl[r^{2n-5}s^{-3}+r^{-2}s^{2n-6}-r^{2n-5}s^{-3}-r^{-2}s^{2n-6}\Bigr]\\
   &=0.
\end{align*}
Here we rely on the following facts:

$\bullet$ Using (\ref{equ ell n-1, k'}) and (\ref{equ ell j-1, k'}), we conclude that
$$ B_{1, 2'}^+=(-rs)^{n-2} B_{1, n'}^+=(-1)^{n-2}r^{n-\frac{7}{2}}s^{n-\frac{5}{2}}(r-s).$$

$\bullet$ According to (\ref{equ thetaii'}) and (\ref{equ B(i+1)' i'}), we have
$$ B_{1, 1'}^+=\frac{r^4s^2B_{1, 2'}^+ B_{2', 1'}^+}{r^2-s^2}=(-1)^n r^{n-\frac{1}{2}}s^{n-\frac{3}{2}}(r-s).$$

We also calculate the coefficient of $\mathcal{E}_{1, j}\mathcal{E}_{1, j'}$ for $3\leq j\leq n$:
\begin{align*}
   & \Bigl[ \Bigl(rs^{-1}\Bigr)^{\rho_1}+\Bigl(r^{-1}s\Bigr)^{\rho_1}r^2s^{-2}\Bigr]B_{1, 1'}^+(-1)^{j-2}r^{-4}s^{-2j+2}(r^2-s^2) \\
   & +\sum_{k=2}^{j-1}r^2B_{1, k}^+ B_{1, k'}^+(-1)^{j-k}r^{-4}s^{-2j+2k-2}(r^2-s^2)+\Bigl[ \Bigl(rs^{-1}\Bigr)^{\rho_j}+\Bigl(r^{-1}s\Bigr)^{\rho_j}\Bigr]s^{-2} B_{1, j}^+ B_{1, j'}^+ \\
   =\;&(-1)^{n+j-2}(r-s)(r^2-s^2)\Bigl\{r^{2n-5}s^{-2j+1}+r^{-2}s^{2n-2j-2}\\
   &-\sum_{k=2}^{j-1}(r^{2n-1-2k}s^{-2j-3+2k}-r^{2n-3-2k}s^{-2j-1+2k})-r^{2n-2j-1}s^{-3}-r^{-2}s^{2n-2j-2}\Bigr\}\\
   =\;&0.
\end{align*}
Similarly, the coefficient of $\mathcal{E}_{1, n+1}^2$ is equal to:
\begin{align*}
   & \Bigl[ \Bigl(rs^{-1}\Bigr)^{\rho_1}+\Bigl(r^{-1}s\Bigr)^{\rho_1}r^2s^{-2}\Bigr]B_{1, 1'}^+(-1)^{n-1}r^{-1}s^{-2n+2}(r-s) \\
   & +\sum_{k=2}^{n}\Bigl(rs^{-1}\Bigr)^{\rho_k}B_{1, k}^+ B_{1, k'}^+r^2(-1)^{n-k}r^{-1}s^{-2n+2k}(r-s)+rs^{-1} B_{1, n+1}^2 \\
  =\; & (-1)^{2n-1}(r-s)^2\Bigl\{ r^{2n-2}s^{-2n+1}+rs^{-2}-\sum_{k=2}^{n}(r^{2n+2-2k}s^{-2n-3+2k}-r^{2n-2k}s^{-2n-1+2k}) \\
   &\hspace{8em} -rs^{-2}-s^{-1}\Bigr\}\\
   =\;&0.
\end{align*}
Based on the previous calculations, we conclude that $A_{1, 1}^+=0$.
Thus we have checked that $L^+ C L^+=C$.
\end{proof}

\section*{AUTHOR DECLARATIONS}
\subsection*{Conflict of Interest}
The authors have no conflicts to disclose.

\bigskip
\noindent
{\bf Data Availability Statement}. All data generated during the study are included in the article.

\section*{Acknowledgements}
The authors are grateful to Professors D. Hernandez and M. Rosso for valuable conversations during the Sino-French Mathematical Cooperation Conference held at East China Normal University, Shanghai, from October 23 to 27, 2024. Special thanks are extended to Marc Rosso for numerous critical and constructive discussions via Zoom meetings.
The authors also thank Professors N. Jing and H. Yamane for their valuable comments.
Gratitude is further expressed to the anonymous referees for the kind comment concerning the Yang–Baxterization.
The second author additionally thanks Professor M. Liu and Hengyi Wang for helpful discussions.

\bibliographystyle{amsalpha}

\end{document}